\documentclass{article}
\usepackage{amssymb,amsfonts,amsmath,amsthm,amsopn,amstext,amscd,latexsym,xy,mathdots,stmaryrd}
\usepackage{hyperref}
\theoremstyle{plain}
\input xy
\xyoption{all}
\newtheorem{theorem}{Theorem}[section]
\newtheorem{lemma}[theorem]{Lemma}
\newtheorem{proposition}[theorem]{Proposition}
\newtheorem{corollary}[theorem]{Corollary}
\newtheorem{definition}[theorem]{Definition}

\theoremstyle{remark}
\newtheorem{remark}[theorem]{Remark}
\numberwithin{equation}{section}
\numberwithin{paragraph}{section}

\renewcommand{\)}{\textup{)}}
\DeclareMathOperator{\Frac}{Frac}
\DeclareMathOperator{\Hom}{Hom}

\DeclareMathOperator{\rec}{rec}

\DeclareMathOperator{\Gal}{Gal}
\DeclareMathOperator{\Aut}{Aut}

\DeclareMathOperator{\Irr}{Irr}
\DeclareMathOperator{\ad}{ad}

\DeclareMathOperator{\Spf}{Spf}

\DeclareMathOperator{\Res}{Res}
\DeclareMathOperator{\Fitt}{Fitt}

\DeclareMathOperator{\Ind}{Ind}
\DeclareMathOperator{\val}{val}

\DeclareMathOperator{\tr}{tr}
\DeclareMathOperator{\St}{St}

\DeclareMathOperator{\diag}{diag}
\DeclareMathOperator{\Art}{Art}
\DeclareMathOperator{\End}{End}

\DeclareMathOperator{\Frob}{Frob}
\DeclareMathOperator{\Ann}{Ann}

\DeclareMathOperator{\Spec}{Spec}

\newcommand{\cD}{{\mathcal D}}

\newcommand{\cF}{{\mathcal F}}

\newcommand{\cH}{{\mathcal H}}

\newcommand{\cJ}{{\mathcal J}}

\newcommand{\cL}{{\mathcal L}}
\newcommand{\cM}{{\mathcal M}}

\newcommand{\cO}{{\mathcal O}}

\newcommand{\cS}{{\mathcal S}}
\newcommand{\cT}{{\mathcal T}}

\newcommand{\fra}{{\mathfrak a}}

\newcommand{\ffrm}{{\mathfrak m}}

\newcommand{\frp}{{\mathfrak p}}
\newcommand{\frq}{{\mathfrak q}}

\newcommand{\bbA}{{\mathbb A}}

\newcommand{\bbC}{{\mathbb C}}

\newcommand{\bbF}{{\mathbb F}}
\newcommand{\bbG}{{\mathbb G}}

\newcommand{\bbM}{{\mathbb M}}
\newcommand{\bbN}{{\mathbb N}}

\newcommand{\bbP}{{\mathbb P}}
\newcommand{\bbQ}{{\mathbb Q}}
\newcommand{\bbR}{{\mathbb R}}

\newcommand{\bbT}{{\mathbb T}}

\newcommand{\bbZ}{{\mathbb Z}}

\newcommand{\GL}{\mathrm{GL}}

\newcommand{\PGL}{\mathrm{PGL}}

\newcommand{\ilim}{\mathop{\varinjlim}\limits}
\newcommand{\plim}{\mathop{\varprojlim}\limits}

\newcommand{\barqp}{{\overline{\bbQ}_p}}
\newcommand{\qp}{{\bbQ_p}}

\newcommand{\barfp}{{\overline{\bbF}_p}}
\newcommand{\fp}{{\bbF_p}}

\newcommand{\CNL}{{\mathrm{CNL}}}
\newcommand{\Sets}{{\mathrm{Sets}}}

\title{Automorphy of some residually dihedral Galois representations}

\author{Jack A. Thorne\footnote{\textsc{Department of Pure Mathematics and Mathematical Statistics, Wilberforce Road, Cambridge, United Kingdom.} \textit{Email address:} \texttt{thorne@dpmms.cam.ac.uk}}}
\begin{document}
\maketitle
\begin{abstract}
We establish the automorphy of some families of 2-dimensional representations of the absolute Galois group of a totally real field, which do not satisfy the so-called `Taylor--Wiles hypothesis'. We apply this to the problem of the modularity of elliptic curves over totally real fields.
\end{abstract}
\tableofcontents

\section{Introduction}
In this paper, we prove a new automorphy theorem for 2-dimensional Galois representations, and apply it to the problem of modularity of elliptic curves. We recall that Wiles first proved that all semi-stable elliptic curves over $\mathbb{Q}$ are modular by applying an automorphy lifting theorem to the Galois representations associated to their 3-adic and 5-adic Tate modules \cite{Wil95}. In order to be able to do this, he first showed that for a semi-stable elliptic curve over $\bbQ$, at least one of the associated mod 3 and mod 5 Galois representations satisfies a certain non-degeneracy condition (the so-called Taylor--Wiles hypothesis, as discussed below). 

More recently, Freitas, Le Hung and Siksek have used similar methods to study the problem of modularity of elliptic curves over general totally real fields \cite{Fre13}. By studying certain modular curves, they show in particular that there are only finitely many non-automorphic elliptic curves over any given totally real field, which necessarily do not satisfy the Taylor--Wiles hypothesis at the primes 3, 5, and 7. (They then study directly the rational points on these modular curves in order to prove the striking result that all elliptic curves defined over real quadratic fields are modular.)  In this paper, we prove a new automorphy lifting theorem which does not require the Taylor--Wiles hypothesis to hold, and apply it to prove the following theorem, which cuts down further the list of possible non-modular elliptic curves:
\begin{theorem}[Theorem \ref{thm_main_thm_applied_to_elliptic_curves}]\label{thm_curve_intro_thm}
Let $F$ be a totally real number field, and let $E$ be an elliptic curve over $F$. Suppose that:
\begin{enumerate}
\item 5 is not a square in $F$;
\item and that $E$ has no $F$-rational 5-isogeny.
\end{enumerate}
Then $E$ is modular.
\end{theorem}
We deduce Theorem \ref{thm_curve_intro_thm} from the following result:
\begin{theorem}[Theorem \ref{thm_main_theorem}]\label{thm_main_intro_thm}
Let $F$ be a totally real number field, let $p$ be a prime, and let $\barqp$ be an algebraic closure of $\qp$. Let $\rho : G_F \to \GL_2(\barqp)$ be a continuous representation satisfying the following conditions.
\begin{enumerate}
\item The representation $\rho$ is unramified almost everywhere.
\item For each place $v | p$ of $F$, $\rho|_{G_{F_v}}$ is de Rham. For each embedding $\tau : F \hookrightarrow \barqp$, we have $\mathrm{HT}_\tau(\rho) = \{ 0, 1 \}$.
\item The representation $\rho$ is totally odd: for each complex conjugation $c \in G_F$, we have $\det \rho(c) = -1$. 
\item The residual representation $\overline{\rho} : G_F \to \GL_2(\barfp)$ is irreducible, while the restriction $\overline{\rho}|_{G_{F(\zeta_p)}}$ is a direct sum of two distinct characters. Moreover, the unique quadratic subfield $K/F$ of $F(\zeta_p)$ is \emph{totally real}.
\end{enumerate}
Then $\rho$ is automorphic: there exists a cuspidal automorphic representation $\pi$ of $\GL_2(\bbA_F)$ of weight 2 and an isomorphism $\iota : \barqp \to \bbC$ such that $\rho \cong r_\iota(\pi)$.
\end{theorem}
(For our conventions regarding the Galois representations associated to automorphic forms, we refer to \S \ref{sec_normalizations} below. When we say that an elliptic curve $E$ is modular, we mean simply that for any prime $p$, the 2-dimensional Galois representation associated to the first \'etale cohomology group $H^1(E_{\overline{F}}, \barqp)$ is automorphic. We use the symbol $G_F = \Gal(\overline{F}/F)$ to denote to the absolute Galois group of the field $F$.) 

Theorem \ref{thm_main_intro_thm} is new even when $F = \bbQ$, and can be viewed as a special case of the Fontaine--Mazur conjecture \cite{Fon95} which falls outside the cases covered by \cite{Kis09a}. (Note that we need make no hypothesis of residual automorphy.) In \emph{loc. cit.}, Kisin proves the automorphy of suitable Galois representations under the assumption that the residual representation $\overline{\rho} : G_F \to \GL_2(\barfp)$ remains irreducible even after restriction to $G_{F(\zeta_p)}$. This so-called `Taylor--Wiles hypothesis' goes back to the pioneering work of Wiles \cite{Wil95} and Taylor--Wiles \cite{Tay95}, and the novelty of the present work is that we can prove a result without making this assumption. (If $\overline{\rho}$ is irreducible, but becomes reducible upon restriction to $G_{F(\zeta_p)}$, then it is necessarily induced from the unique index 2 subgroup of $G_F$ which contains $G_{F(\zeta_p)}$. We thus say in this case that $\rho$ is residually dihedral.)

\paragraph*{Methods.} Skinner and Wiles have successfully removed the Taylor--Wiles hypothesis on $\overline{\rho}$ in the case where $\rho$ is ordinary using Hida's theory of ordinary automorphic forms, cf. \cite{Ski99}, \cite{Ski01}. These automorphic forms fit into positive-dimensional $p$-adic families. Roughly speaking, by moving in a generic direction over weight space in these families one can replace $\overline{\rho}$ with a residual representation which does satisfy the Taylor--Wiles hypothesis. Similar techniques have recently been applied by Allen to prove automorphy theorems in the 2-adic residually dihedral case \cite{All13}. 

By contrast, the Galois representations that we consider are often (potentially) supersingular, and a good integral theory of variable weight $p$-adic automorphic forms is no longer available. We use a technique which is inspired by works of Khare \cite{Kha03}, \cite{Kha04}. In the article \cite{Kha04}, the author proves an automorphy lifting theorem by first proving an $R = \bbT$ theorem with some additional ramification conditions imposed. These conditions are not satisfied by $\rho$, but they are satisfied by $\rho$ modulo $p^{N_0}$ for some $N_0$. In fact, by choosing these conditions carefully, one can force $R = W(k)$ (i.e.\ the Witt vectors of a finite field) which immediately implies that $R = \bbT$. In particular, one shows $R = \bbT$ in this instance without using the Taylor--Wiles argument. One can then apply Mazur's principle (cf. \cite{Jar99}) to show that $\rho \text{ mod }p^{N_0}$ appears in the \'etale cohomology of a modular curve at a fixed level $N$, independent of $N_0$. Taking the limit $N_0 \to \infty$, one finds that $\rho$ itself appears in the \'etale cohomology of this modular curve, and hence itself is automorphic. 

Our method is what one obtains upon re-introducing the Taylor--Wiles argument into Khare's proof. The obstruction to applying the Taylor--Wiles argument when $\overline{\rho}|_{G_{F(\zeta_p)}}$ is reducible is the dual Selmer group of the adjoint representation $\ad^0 \overline{\rho}$. This is the group that we will denote $H^1_{\cS, T}(\ad^0 \overline{\rho}(1))$ below. We first show that by imposing auxiliary ramification conditions, one can kill off the troublesome part of the dual Selmer group. These conditions are again satisfied by $\rho \text{ mod }p^{N_0}$ but not by $\rho$ itself. (More precisely, we impose a Steinberg-type condition at places $v$ such that $q_v \equiv -1 \text{ mod }p^{N_0}$, and $\rho(\Frob_v)$ agrees modulo $p^{N_0}$ with the image under $\rho$ of complex conjugation.) This ramification condition has the desired effect on the dual Selmer group only when $c \in G_K$, where $K$ denotes the field from which $\overline{\rho}$ is induced. This is the main reason for imposing the condition that $K$ is totally real in the statement of Theorem \ref{thm_main_intro_thm}.

Having killed off the dual Selmer group, we can prove an $R = \bbT$ theorem using the techniques of Taylor, Wiles and Kisin, cf. \cite{Kis09}. This shows the automorphy of the representation $\rho \text{ mod }p^{N_0}$. We can then use a version of Mazur's principle to deduce from this the automorphy of the representation $\rho$. (In lowering the level, we also make essential use of an idea of Fujiwara \cite{Fuj06}.)

\paragraph*{Structure of this paper.} We now describe the organization of this paper. In \S \ref{sec_normalizations} below, we describe our conventions and normalizations. In \S \ref{sec_boston_lenstra_ribet}, we prove a slight extension of a result of Boston, Lenstra and Ribet about 2-dimensional Galois representations that plays an important role in our version of Mazur's principle. In \S \ref{sec_shimura_curves_and_varieties}, we study the cohomology of Shimura curves and their 0-dimensional analogues, and prove a result of `automorphy by successive approximation' (Corollary \ref{cor_automorphy_by_successive_approximation}).

In \S \ref{sec_galois_theory}, we recall some elements of the deformation theory of Galois representations. We also carry out our analysis of the dual Selmer groups of the residually dihedral Galois representations described above and make some remarks about ordinary automorphic forms and Galois representations. (We caution the reader that we use the word `ordinary' where some authors use `nearly ordinary'.) In \S \ref{sec_r_equals_t}, we prove our $R = \bbT$ type theorem. Finally, in \S \ref{sec_main_theorem} we combine everything to deduce Theorems \ref{thm_curve_intro_thm} and \ref{thm_main_intro_thm}.

\section{Notation and normalizations}\label{sec_normalizations}

A base number field $F$ having been fixed, we will also choose algebraic closures $\overline{F}$ of $F$ and $\overline{F}_v$ of $F_v$ for every finite place $v$ of $F$. The algebraic closure of $\bbR$ is $\bbC$. If $p$ is a prime, then we will also write $S_p$ for the set of places of $F$ above $p$, $\overline{\bbQ}_p$ for a fixed choice of algebraic closure of $\bbQ_p$, and $\val_p$ for the $p$-adic valuation on $\barqp$ normalized so that $\val_p(p) = 1$. These choices define the absolute Galois groups $G_F = \Gal(\overline{F}/F)$ and $G_{F_v} = \Gal(\overline{F}_v/F_v)$.  We write $I_{F_v} \subset G_{F_v}$ for the inertia subgroup. We also fix embeddings $\overline{F} \hookrightarrow \overline{F}_v$, extending the canonical embeddings $F \hookrightarrow F_v$. This determines for each place $v$ of $F$ an embedding $G_{F_v} \to G_{F}$. We write $\bbA_F$ for the adele ring of $F$, and $\bbA_F^\infty = \prod'_{v \nmid \infty} F_v$ for its finite part. If $v$ is a finite place of $F$, then we write $k(v)$ for the residue field at $v$, $\kappa(v)$ for the residue field of $\overline{F}_v$, and $q_v = \# k(v)$. If we need to fix a choice of uniformizer of $\cO_{F_v}$, then we will denote it $\varpi_v$.

If $S$ is a finite set of finite places of $F$, then we write $F_S$ for the maximal subfield of $\overline{F}$ unramified outside $S$, and $G_{F, S} = \Gal(F_S/F)$; this group is naturally a quotient of $G_F$. If $v \not\in S$ is a finite place of $F$, then the map $G_{F_v} \to G_{F, S}$ factors through the unramified quotient of $G_{F_v}$, and we write $\Frob_v \in G_{F, S}$ for the image of a \emph{geometric} Frobenius element. We write $\epsilon : G_F \to \bbZ_p^\times$ for the $p$-adic cyclotomic character; if $v$ is a finite place of $F$, not dividing $p$, then $\epsilon(\Frob_v) = q_v^{-1}$. If $\rho : G_F \to \GL_n(\barqp)$ is a continuous representation, we say that $\rho$ is de Rham if for each place $v | p$ of $F$, $\rho|_{G_{F_v}}$ is de Rham. In this case, we can associate to each embedding $\tau : F \hookrightarrow \barqp$ a multiset $\mathrm{HT}_\tau(\rho)$ of Hodge--Tate weights, which depends only on $\rho|_{G_{F_v}}$, where $v$ is the place of $F$ induced by $\tau$. This multiset has $n$ elements, counted with multiplicity. There are two natural normalizations for $\mathrm{HT}_\tau(\rho)$ which differ by a sign, and we choose the one with $\mathrm{HT}_\tau(\epsilon) = \{ -1 \}$ for every choice of $\tau$.

We use geometric conventions for the Galois representations associated to automorphic forms, which we now describe. First, we use the normalizations of the local and global Artin maps $\Art_{F_v} : F_v^\times \to W_{F_v}^\text{ab}$ and $\Art_F : \bbA_F^\times \to G_F^\text{ab}$ which send uniformizers to geometric Frobenius elements. If $v$ is a finite place of $F$, then we write $\rec_{F_v}$ for the local Langlands correspondence for $\GL_2(F_v)$, normalized as in Henniart \cite{Hen93} and Harris-Taylor \cite{Tay01} by a certain equality of $\epsilon$- and L-factors. We recall that $\rec_{F_v}$ is a bijection between the set of isomorphism classes of irreducible admissible representations $\pi$ of $\GL_2(F_v)$ over $\bbC$ and set of isomorphism classes of 2-dimensional Frobenius--semi-simple Weil--Deligne representations $(r, N)$ over $\bbC$. We define $\rec_{F_v}^T(\pi) = \rec_{F_v}(\pi \otimes | \cdot |^{-1/2})$. Then $\rec_{F_v}^T$ commutes with automorphisms of $\bbC$, and so makes sense over any field $\Omega$ which is abstractly isomorphic to $\bbC$ (e.g.\ $\barqp$).

If $v$ is a finite place of $F$ and $\chi : W_{F_v} \to \Omega^\times$ is a character with open kernel, then we write $\St_2(\chi \circ \Art_{F_v})$ for the inverse image under $\rec_{F_v}^T$ of the Weil--Deligne representation 
\begin{equation}\label{eqn_definition_of_steinberg_representation}
\left( \chi  \oplus  \chi | \cdot |^{-1} , \left(\begin{array}{cc} 0 & 1 \\ 0 & 0 \end{array}\right) \right).
\end{equation}
If $\Omega = \bbC$, then we call $\St_2 = \St_2(1)$ the Steinberg representation; it is the unique generic subquotient of the normalized induction $i_B^{\GL_2} |\cdot|^{1/2} \otimes |\cdot|^{-1/2}$. If $(r, N)$ is any Weil--Deligne representation, we write $(r, N)^\text{F-ss}$ for its Frobenius--semi-simplification. If $v$ is a finite place of $F$ and $\rho : G_{F_v} \to \GL_n(\barqp)$ is a continuous representation, which is de Rham if $v | p$, then we write $\mathrm{WD}(\rho)$ for the associated Weil--Deligne representation, which is uniquely determined, up to isomorphism. (If $v \nmid p$, then the representation $\mathrm{WD}(\rho)$ is defined in \cite[\S 4.2]{Tat79}. If $v | p$, then the definition of $\mathrm{WD}(\rho)$ is due to Fontaine, and is defined in e.g. \cite[\S 2.2]{Bre02}.)

In this paper, the only automorphic representations we consider are cuspidal automorphic representations $\pi = \otimes'_v \pi_v$ of $\GL_2(\bbA_F)$ such that for each $v | \infty$, $\pi_v$ is the lowest discrete series representation of $\GL_2(\bbR)$ of trivial central character. (The one exception is in the proof of Lemma \ref{lem_existence_of_ordinary_automorphic_lifts}.) In particular, any such $\pi$ is unitary. We will say that $\pi$ is a cuspidal automorphic representation of $\GL_2(\bbA_F)$ of weight 2. If $\pi$ is such a representation, then for every isomorphism $\iota : \barqp \to \bbC$, there is a continuous representation $r_\iota(\pi) : G_F \to \GL_2(\barqp)$ satisfying the following conditions:
\begin{enumerate}
\item The representation $r_\iota(\pi)$ is de Rham and for each embedding $\tau : F \hookrightarrow \barqp$, $\mathrm{HT}_\tau(\rho) = \{ 0, 1 \}$.
\item Let $v$ be a finite place of $F$. Then $\mathrm{WD}(r_\iota(\pi)|_{G_{F_v}})^\text{F-ss} \cong \rec^T_{F_v}(\iota^{-1} \pi_v)$.
\item Let $\omega_\pi$ denote the central character of $\pi$; it is a character $\omega_\pi : F^\times \backslash \bbA_F^\times \to \bbC^\times$ of finite order. Then $\det r_\iota(\pi) = \epsilon^{-1} \iota^{-1} (\omega_\pi \circ \Art_F^{-1})$.
\end{enumerate}
For concreteness, we spell out this local-global compatibility at the unramified places. Let $v \nmid p$ be a prime such that $\pi_v \cong i_B^{\GL_2} \chi_1 \otimes \chi_2$, where $\chi_1, \chi_2 : F_v^\times \to \bbC^\times$ are unramified characters and $i^{\GL_2}_B$ again denotes normalized induction from the upper-triangular Borel subgroup $B \subset \GL_2$. Then the representation $r_\iota(\pi)$ is unramified at $v$, and the characteristic polynomial of $\iota r_\iota(\pi)(\Frob_v)$ is $(X - q^{1/2} \chi_1(\varpi_v))(X - q^{1/2} \chi_2(\varpi_v))$. If $T_v$, $S_v$ are the unramified Hecke operators defined in \S \ref{sec_hecke_operators} below, and we write $t_v,$ $s_v$ for their respective eigenvalues on $\pi_v^{\GL_2(\cO_{F_v})}$, then we have
\begin{equation}\label{eqn_local_global_compatibility_at_unramified_places}
(X - q^{1/2} \chi_1(\varpi_v))(X - q^{1/2} \chi_2(\varpi_v)) = X^2 - t_v X + q_v s_v.
\end{equation}
With the above notations, the pair $(\pi, \omega_\pi)$ is a cuspidal regular algebraic polarized automorphic representation in the sense of \cite{Bar13}, and our representation $r_\iota(\pi)$ coincides with the one defined there. On the other hand, if $\sigma(\pi) : G_F \to \GL_2(\barqp)$ denotes the representation associated to $\pi$ by Carayol \cite{Car86}, then there is an isomorphism $\sigma(\pi) \cong r_\iota(\pi) \otimes (\iota^{-1} \omega_\pi \circ \Art_F^{-1})^{-1}.$

We will call a finite extension $E/\bbQ_p$ inside $\barqp$ a coefficient field. A coefficient field $E$ having been fixed, we will write $\cO$ or $\cO_E$ for its ring of integers, $k$ or $k_E$ for its residue field, and $\lambda$ or $\lambda_E$ for its maximal ideal. If $A$ is a complete Noetherian local $\cO$-algebra with residue field $k$, then we write $\ffrm_A \subset A$ for its maximal ideal, and $\CNL_A$ for the category of complete Noetherian local $A$-algebras with residue field $k$. We endow each object $R \in \CNL_A$ with its profinite ($\ffrm_R$-adic) topology. 

If $\Gamma$ is a profinite group and $\rho : \Gamma \to \GL_n(\barqp)$ is a continuous representation, then we can assume (after a change of basis) that $\rho$ takes values in $\GL_n(\cO)$, for some choice of coefficient field $E$. The semi-simplification of the composite representation $\Gamma \to \GL_n(\cO) \to \GL_n(k)$ is independent of choices, up to isomorphism, and we will write $\overline{\rho} : \Gamma \to \GL_n(\barfp)$ for this semi-simplification.

If $E$ is a coefficient field and $\overline{\rho} : \Gamma \to \GL_2(k)$ is a continuous representation, then we write $\ad \overline{\rho}$ for $\End_k(\overline{\rho})$, endowed with its structure of $k[\Gamma]$-module. We write $\ad^0 \overline{\rho} \subset \ad \overline{\rho}$ for the submodule of trace 0 endomorphisms, and (if $\Gamma = G_F$) $\ad^0 \overline{\rho}(1)$ for its twist by the cyclotomic character. If $M$ is a discrete $\bbZ[G_F]$-module (resp. $\bbZ[G_{F, S}]$-module), then we write $H^1(F, M)$ (resp. $H^1(F_S/F, M)$) for the continuous Galois cohomology group with coefficients in $M$. Similarly, if $M$ is a discrete $\bbZ[G_{F_v}]$-module, then we write $H^1(F_v, M)$ for the continuous Galois cohomology group with coefficients in $M$. If $M$ is a discrete $k[G_F]$-module (resp. $k[G_{F, S}$]-module, resp. $k[G_{F_v}]$-module), then $H^1(F, M)$ (resp. $H^1(F_S/F, M)$, resp. $H^1(F_v, M)$) is a $k$-vector space, and we write $h^1(F, M)$ (resp. $h^1(F_S/F, M)$, resp. $h^1(F_v, M)$) for the dimension of this $k$-vector space, provided that it is finite.

Suppose that a coefficient field $E$ has been fixed. If $M$ is a topological $\cO$-module (i.e.\ an $\cO$-module endowed with a topology making the natural map $\cO \times M \to M$ continuous), then we define $M^\vee = \Hom_\cO(M, E/\cO)$, the Pontryagin dual of $M$. If furthermore $R$ is an $\cO$-algebra and $M$ is an $R$-module, then $M^\vee$ is naturally an $R$-module: $r \in R$ acts by the transpose $r^\vee \in \End_\cO(M^\vee)$.

\section{A result on group representations of dimension 2}\label{sec_boston_lenstra_ribet}

Let $E$ be a coefficient field, and let $A \in \mathrm{CNL}_\cO$. We suppose given a group $G$ and a representation $\rho : G \to \Aut_A(V)$, where $V$ is a free $A$-module of rank 2, such that $\overline{\rho} = \rho \otimes_A k$ is absolutely irreducible. We extend $\rho$ to an algebra homomorphism $\rho : A[G] \to \Aut_A(V)$, and write $\tr : A[G] \to A$ and $\det : A[G] \to A$ for the maps obtained by composing $\rho$ with the usual trace and determinant.

\begin{theorem}\label{thm_eicher_shimura_ring_case}
Let $W$ be a finitely generated $A$-module, and let $\sigma : G \to \Aut_A(W)$ a representation such that for all $g \in G$, $\sigma(g)$ satisfies the characteristic polynomial of $\rho(g)$, i.e.\ the relation
\begin{equation}
 \sigma(g)^2 - (\tr g) \sigma(g) + \det g = 0
 \end{equation}
holds in $\End_A(W)$ for each $g \in G$. Then there exists an $A$-module $U_0$ and an isomorphism $\sigma \cong \rho \otimes_A U_0$ of $A[G]$-modules.
\end{theorem}
In the case $A = k$, Theorem \ref{thm_eicher_shimura_ring_case} is a well-known result of Boston, Lenstra and Ribet \cite{Bos91}. We give a proof that follows \cite{Bos91}. 
\begin{proof}[Proof of Theorem \ref{thm_eicher_shimura_ring_case}]
Let $J \subset A[G]$ denote the 2-sided ideal generated by the elements $g^2 - (\tr g) g + \det g$, $g \in G$, and $R = A[G]/J$; thus both $\sigma$ and $\rho$ factor through the quotient $A[G] \to R$. By Morita equivalence, to prove the theorem it suffices to show that the natural map $R \to \End_A(V)$ is an isomorphism.

Let $\ast : A[G] \to A[G]$ denote the $A$-linear anti-involution defined on group elements by $g^\ast = g^{-1} \det g$. Using the identities $g g^\ast = \det g$ and $g^2 - (\tr g) g + \det g \in J$, we deduce that $g + g^\ast \equiv \tr g \text{ mod }J$ for all $g \in G$, hence $x + x^\ast \equiv \tr x \text{ mod }J$ for all $x \in A[G]$. This implies that $J = J^\ast$, and hence $\ast$ descends to an anti-involution of $R$, which we denote by the same symbol.

It follows that $x + x^\ast = \tr x$ for all $x \in R$. We now show that $x x^\ast = \det x$ for all $x \in R$. Indeed, we have for any $x, y \in R$, $(x + y) (x + y)^\ast = x x^\ast + y y^\ast + \tr x y^\ast$, and the same identity holds for the images of $x, y$ in $\End_A(V)$. Consequently, the set $\{ x \in R \mid x x^\ast \in A \text{ and }x x^\ast = \det x \}$ is stable under addition. It contains every $A$-multiple of an element $g \in G$, so equals the whole of $R$. In particular, we see that an element $x \in R$ is a unit if and only if $\det x \in A^\times$.

We can now show that the natural map $R \to \End_A(V)$ is an isomorphism. It is surjective, as follows from the lemmas of Burnside and Nakayama, so it is enough to show that it is also injective. Choose an element $x \in R$ such that $\rho(x) = 0$. Then $x = -x^\ast$, hence $yx = -yx^\ast$ for all $y \in R$. We also have $x y^\ast + y x^\ast = \tr x y^\ast = 0$, hence $yx = x y^\ast$. Applying this last identity repeatedly, we find that for any $y, z \in R$, we have $yzx = yxz^\ast = x (zy)^\ast = zyx$, hence $(yz-zy)x = 0$. Consequently, $\Ann x = \{ r \in R \mid rx = 0 \}$ is a 2-sided ideal of $R$ which  contains the set $\{ yz - zy \mid y, z \in R \}$. 

We thus see that the image of $\Ann x$ in $\End_A(V)$ is non-trivial, and contains all the elements $ef - fe$ for $e, f \in \End_A(V)$. It follows from Nakayama's lemma that the image of $\Ann x$ in $\End_A(V)$ actually equals $\End_A(V)$, and hence there exists $w \in \Ann x$ such that $\rho(w) = 1$. In particular, $w$ is a unit and $wx = 0$. This implies that $x = 0$, and completes the proof.
\end{proof}

\section{Shimura curves and Hida varieties}\label{sec_shimura_curves_and_varieties}

Let $F$ be a totally real number field. We write $\tau_1, \dots, \tau_d$ for the distinct real embeddings of $F$, and assume that $d$ is even. In this section, we define Shimura curves and their 0-dimensional analogues which, following \cite{Fuj06}, we call Hida varieties.

Let $Q$ be a finite set of finite places of $F$. We fix for each such $Q$ a choice of quaternion algebra $B_Q$ over $F$, as follows:
\begin{itemize}
\item If $\# Q$ is odd, then $B_Q$ is ramified exactly at $Q \cup \{ \tau_2, \dots, \tau_d \}$.
\item If $\# Q$ is even, then $B_Q$ is ramified exactly at $Q \cup \{ \tau_1, \dots, \tau_d \}$.
\end{itemize}
Thus $B_Q$ is uniquely determined, up to isomorphism. We also fix for each such $Q$ a maximal order $\cO_Q \subset B_Q$ and an isomorphism $\cO_Q \otimes_{\cO_F} \prod_{v \nmid Q \infty} \cO_{F_v} \cong \prod_{v \nmid Q \infty} M_2(\cO_{F_v})$.  We write $G_Q$ for the associated reductive group over $\cO_F$; its functor of points is $G_Q(R) = (\cO_Q \otimes_{\cO_F} R)^\times$. We thus obtain for each finite place $v \not\in Q$ of $F$ an isomorphism $G_Q(\cO_{F_v}) \cong \GL_2(\cO_{F_v})$.

Let $v$ be a finite place of $F$. We define for each $n \geq 1$ a sequence of open compact subgroups of $\GL_2(\cO_{F_v})$
\begin{equation}\label{eqn_definition_of_open_compact_subgroups} U_0(v^n) \subset U_1(v^n) \subset U_1^1(v^n), 
\end{equation}
where:
\begin{itemize}
\item $U_0(v^n) = \left\{ \left( \begin{array}{cc} \ast & \ast \\ 0 & \ast \end{array} \right) \text{ mod } \varpi_v^n \right \}$;
\item $U_1(v^n) = \left\{ \left( \begin{array}{cc} a & \ast \\ 0 & d \end{array} \right) \text{ mod }\varpi_v^n \middle| ad^{-1} = 1 \right\}$;
\item and $U_1^1(v^n) = \left\{ \left( \begin{array}{cc} 1 & \ast \\ 0 & 1 \end{array} \right) \text{ mod }\varpi_v^n \right \}$.
\end{itemize}
If $n = 1$, then we omit it from the notation and write $U_0(v) \subset U_1(v) \subset U_1^1(v)$. By abuse of notation, when $v \not \in Q$ we will use the same symbols to denote the open compact subgroups of $G_Q(F_v)$ induced by our identifications $G_Q(\cO_{F_v}) \cong \GL_2(\cO_{F_v})$.

We now fix for the rest of \S \ref{sec_shimura_curves_and_varieties} a finite place $a$ of $F$ satisfying $q_a > 4^d$. In what follows, we will assume that our sets $Q$ are chosen so that $a \not\in Q$. This being the case, we now single out a convenient class of open compact subgroups of $G_Q(\bbA^\infty_F)$. We say that $U  \subset G_Q(\bbA^\infty_F)$ is good if it satisfies the following conditions.
\begin{itemize}
\item $U = \prod_v U_v$ for open compact subgroups $U_v \subset G_Q(F_v)$.
\item For each $v \in Q$, $U_v$ is the unique maximal compact subgroup of $G(F_v)$.
\item We have $U_a = U_1^1(a)$.
\end{itemize}
We will write $\cJ_Q$ for the set of good subgroups $U \subset G_Q(\bbA_F^\infty)$. (Note that this set depends on the choice of auxiliary place $a$, and not just on $Q$; since the choice of $a$ is supposed fixed, we omit it from the notation.)

\subsection{Hecke operators}\label{sec_hecke_operators}

We continue with the notation of the previous section. Let $Q$ be a finite set of finite places of $F$, with $a \not \in Q$, and let $U \in \cJ_Q$. If $v$ is a finite place of $F$, then we write $\cH(G_Q(F_v), U_v)$ for the $\bbZ$-algebra of compactly supported $U_v$-biinvariant functions $f : G(F_v) \to \bbZ$. A basis of this algebra as $\bbZ$-module is given by the characteristic functions $[ U_v g_v U_v ]$ of double cosets, and the identity is the function $[U_v]$. The multiplication law is given explicitly as follows: if $U_v g_v U_v = \sqcup_i h_i U_v$ and $U_v g'_v U_v = \sqcup_i h_i' U_v$, then $[U_v g_v U_v] \cdot [U_v g'_v U_v]$ is the characteristic function of the double coset $\sqcup_i h_i h_i' U_v$. (This algebra structure is the same as the convolution algebra structure that results from the choice of Haar measure on $G(F_v)$ that gives $U_v$ measure 1.) 

If $M$ is a smooth $\bbZ[G_Q(F_v)]$-module, then $M^{U_v}$ is canonically a $\cH(G_Q(F_v), U_v)$-module. The action of the element $[U_v g_v U_v]$ can be given explicitly as follows: write $U_v g_v U_v = \sqcup_i h_i U_v$. Then for any $x \in M^{U_v}$, $[U_v g_v U_v] \cdot x = \sum_i h_i \cdot x$. 

We now define some elements of these algebras that will be of particular interest. If $v \not \in Q$ and $U_v = \GL_2(\cO_{F_v})$, then we have $T_v, S_v \in  \cH(\GL_2(F_v), \GL_2(\cO_{F_v})) \cong \cH(G_Q(F_v), U_v)$, where
\[ T_v = [ \GL_2(\cO_{F_v}) \left( \begin{array}{cc} \varpi_v & 0 \\ 0 & 1 \end{array} \right) \GL_2(\cO_{F_v}) ]\text{ and }S_v = [ \GL_2(\cO_{F_v}) \left( \begin{array}{cc} \varpi_v & 0 \\ 0 & \varpi_v \end{array} \right) \GL_2(\cO_{F_v}) ], \]
and $\varpi_v \in \cO_{F_v}$ is a uniformizer. If $v \not\in Q$, and $U_1^1(v^n) \subset U_v \subset U_0(v^n)$, then we write
\[ \mathbf{U}_v = [ U_v \left( \begin{array}{cc} \varpi_v & 0 \\ 0 & 1 \end{array} \right) U_v ]. \]
It is an abuse of notation to denote elements of different algebras by the same symbol; however, if $M$ is a smooth $\bbZ[G_Q(F_v)]$-module, then the action of $\mathbf{U}_v$ commutes with all of the inclusions ($n \geq m \geq 1$):
\[ \xymatrix{M^{U_0(v^m)} \ar[r] \ar[d] & M^{U_v}  \ar[r] & \ar[d] M^{U_1^1(v^m)} \\ M^{U_0(v^n)} \ar[rr] &  & M^{U_1^1(v^n)}.} \]
(For the justification of this fact, see \cite[\S 3.1]{Clo08}.) If $v \in Q$, then $U_v$ is the unique maximal compact subgroup of $(B_Q \otimes_F F_v)^\times$ (since $U$ is a good subgroup, by assumption) and we have $\cH(G_Q(F_v), U_v) = \bbZ[\mathbf{U}_v, \mathbf{U}_v^{-1}]$, where
\[ \mathbf{U}_v = [ U_v \widetilde{\varpi}_v U_v ], \]
and $\widetilde{\varpi}_v \in \cO_Q \otimes_{\cO_F} \cO_{F_v}$ is a uniformizer. We note that $\widetilde{\varpi}_v$ normalizes $U_v$, so the action of $\mathbf{U}_v$ on $M^{U_v}$ is given simply by the action of $\widetilde{\varpi}_v$.

The use of the symbol $\mathbf{U}_v$ is again an abuse of notation, which is justified by the following lemma.
\begin{lemma}\label{lem_u_eigenvalue_on_twist_of_steinberg}
Let $v \in Q$, let $\chi : F_v^\times \to \bbC^\times$ be an unramified character, and let $\pi = \St_2(\chi)$, an irreducible admissible representation of $\GL_2(F_v)$. Let $\mathrm{JL}(\pi) = \chi \circ \det$ denote the 1-dimensional representation of $G_Q(F_v)$ associated to $\pi$ by the local Jacquet--Langlands correspondence. Then \[ \rec^T_{F_v}(\pi) = \left( \chi\oplus\chi | \cdot |^{-1}, \left(  \begin{array}{cc} 0 & 1 \\ 0 & 0 \end{array} \right) \right). \]
The $\mathbf{U}_v$-eigenvalues on $\pi^{U_0(v)}$ and $\mathrm{JL}(\pi)^{U_v}$ coincide, and are both equal to the eigenvalue of $\Frob_v$ on $\rec^T_{F_v}(\pi)^{N = 0}$.
\end{lemma}
\begin{proof}
This follows from the definition (\ref{eqn_definition_of_steinberg_representation}) of $\St_2(\chi)$ and a direct calculation: see the proof of \cite[Lemma 3.1.5]{Clo08} and note that, in the notation there, we have $\St_2(\chi) = \mathrm{Sp}_2(\chi|\cdot|^{-1/2})$.
\end{proof}
Let $\pi$ be a cuspidal automorphic representation of $\GL_2(\bbA_F)$ of weight 2. Let $p$ be a prime and let $\iota : \barqp \to \bbC$ be an isomorphism, and let $v$ be a place of $F$ dividing $p$. For each $n \geq 1$, the eigenvalues of the ${\mathbf{U}_v}$ operator on $\iota^{-1} \pi_v^{U_1^1(v^n)}$ lie in $\overline{\bbZ}_p$. We say that $\pi_v$ is $\iota$-ordinary if there exists $n \geq 1$ such that $\iota^{-1}\pi_v^{U_1^1(v^n)} \neq 0$ and ${\mathbf{U}_v}$ has an eigenvalue on this space which is a $p$-adic unit, i.e.\ lies in $\overline{\bbZ}_p^\times$. 
\subsection{Hida varieties}

We now suppose that $\# Q$ is even. For each $U \in \cJ_Q$, there is an associated finite double quotient: 
\[ X_Q(U) = G_Q(F) \backslash G_Q(\bbA_F^\infty) / U. \]
For any $g \in G_Q(\bbA_F^{a, \infty})$, $g^{-1} U g$ is also a good subgroup, and there is a map:
\[ X_Q(U) \to X_Q(g^{-1} U g) \]
induced by right multiplication on $G(\bbA_F^\infty)$. This gives rise to a right action of the group $G_Q(\bbA_F^{a, \infty})$ on the projective system of sets $\{ X_Q(U) \}_{U \in \cJ_Q}$.
\begin{lemma}\label{lem_triviality_of_isotropy_groups_definite_case}
Let $U \in \cJ_Q$. For each $g \in G_Q(\bbA_F^\infty)$, we have $g G_Q(F) g^{-1} \cap U \subset F^\times  \subset G_Q(\bbA_F^\infty)$.
\end{lemma}
\begin{proof}
We copy the proof of \cite[Lemma 12.1]{Jar99}. Let $\gamma \in G_Q(F)$, $u \in U$, and suppose that $g \gamma g^{-1} = u$. We define $\mu(\gamma) = (t(\gamma)^2 - 4 \nu(\gamma))/\nu(\gamma) \in F$ (where $t, \nu$ denote reduced trace and norm, respectively). We must show that $\mu(\gamma) = 0$. Since $G_Q(F \otimes_\bbQ \bbR)$ is compact modulo centre, we have $-4 \leq \tau_i(\mu(\gamma)) \leq 0$ for each $i = 1, \dots, d$, hence $| \bbN_{F/\bbQ}\mu(\gamma) | \leq 4^d$. 

On the other hand, we have $U_a = U_1^1(a)$, which implies that $\mu(\gamma) \equiv 0 \text{ mod } \varpi_a^2 \cO_{F_a}$. We then see that either $\mu(\gamma) = 0$, or $| \bbN_{F/\bbQ} \mu(\gamma) | \geq q_a > 4^d$. The latter possibility leads immediately to a contradiction. This completes the proof.
\end{proof}
\subsection{Shimura curves}\label{sec_shimura_curves}

We now suppose that $\# Q$ is odd. We remind the reader that there are two possible conventions regarding the definition of canonical models of the Shimura curves associated to the group $G_Q$; the difference between these is discussed carefully in \cite[\S 3.3.1]{Cor05}. In order that we can refer easily to \cite{Car86}, we follow the convention of Carayol \cite{Car86a}, which we now describe. We fix an isomorphism $B_Q \otimes_{F, \tau_1} \bbR \cong M_2(\bbR)$, and write $X$ for the $G_Q(F \otimes_\bbQ \bbR)$-conjugacy class of the homomorphism $h : \mathbb{S} = \Res_{\bbC/\bbR} \bbG_m \to (\Res_{F / \bbQ} G_Q)_\bbR$ which sends $z = x + iy \in \bbC^\times = \mathbb{S}(\bbR)$ to the element
\[ \left( \left( \begin{array}{cc} x & y \\ -y & x \end{array} \right)^{-1} , 1, \dots, 1 \right) \in \prod_{i=1}^d G_Q(F \otimes_{F, \tau_i} \bbR). \]
For each $U \in \cJ_Q$, there is an associated topological space:
\[ M_Q(U)(\bbC) = G_Q(F) \backslash ( G_Q(\bbA_F^\infty) / U \times X ). \]
According to the theory of Shimura varieties, there exists a canonical projective algebraic curve $M_Q(U)$ over $F$ such that $M_Q(U)(\bbC)$ is its set of complex points (for the embedding $\tau_1 : F \hookrightarrow \bbC$). There is a natural right action of the group $G_Q(\bbA_F^{a,\infty})$ on the projective system $\{ M_Q(U)(\bbC) \}_{U \in \cJ_Q}$, given by the system of isomorphisms
\[ M_Q(U)(\bbC) \to M_Q(g^{-1}U g)(\bbC). \]
induced by right multiplication in $G_Q(\bbA_F^\infty)$. This extends uniquely to a right action of $G_Q(\bbA_F^{a,\infty})$ on the projective system $\{ M_Q(U) \}_{U \in \cJ_Q}$ of curves over $F$. For each prime $p$, we obtain a left action of this group on the \'etale cohomology groups:
\[ \ilim_{U \in \cJ_Q} H^i(M_Q(U)_{\overline{F}}, \barqp). \]
For each $U \in \cJ_Q$ and for each isomorphism $\iota : \barqp \to \bbC$, we then have (cf. \cite[\S 2]{Car86}) an explicit decomposition of $\cH(G_Q(\bbA_F^{a, \infty}), U^a) \otimes \barqp[G_F]$-modules:
\begin{equation}\label{eqn_decomposition_of_shimura_curve_cohomology}
 H^1(M_Q(U)_{\overline{F}}, \barqp) \cong \bigoplus_{\pi \in A_Q(U)} \iota^{-1}\mathrm{JL}(\pi)^U \otimes ( r_\iota(\pi) \otimes (\iota^{-1} \omega_\pi \circ \Art_F^{-1})^{-1} ), 
\end{equation}
where $A_Q(U)$ denotes the set of cuspidal automorphic representations $\pi$ of $\GL_2(\bbA_F)$ of weight 2 such that for each $v \in Q$, $\pi_v$ is an unramified twist of the Steinberg representation, and $\mathrm{JL}(\pi)^U \neq 0$. (Here $\mathrm{JL}$ denotes the global Jacquet--Langlands correspondence.)

\subsection{Integral models, case $v \not\in Q$}
Let $v \neq a$ be a finite place of $F$, and let $\overline{Q}$ be a finite set of finite places of $F$, not containing $v$ or $a$, and of odd cardinality.
(The reason for using the notation $\overline{Q}$ is that we want to reserve the notation $Q$ for later use as the enlarged set $Q = \overline{Q} \cup \{ v \}$.)
\begin{theorem}\label{thm_existence_of_carayol_integral_models}
Let ${}_v\cJ_{\overline{Q}}$ denote the set of good subgroups $U = \prod_w U_w \subset G_{\overline{Q}}(\bbA_F^\infty)$ such that $U_{v} = \GL_2(\cO_{F_v})$. Let $U \in {}_v\cJ_{\overline{Q}}$, and let $U' = \prod_w U'_w$ be the subgroup defined by $U'_w = U_w$ if $w \neq v$ and $U'_v = U_0(v)$. Then the morphisms 
\begin{gather*}
M_{\overline{Q}}(U)_{F_v} \to \Spec F_v \\ 
M_{\overline{Q}}(U')_{F_v} \to \Spec F_v
\end{gather*}
extend canonically to flat projective morphisms 
\begin{gather*} {}_v \bbM_{\overline{Q}}(U) \to \Spec \cO_{F_v} \\ 
{}_v \bbM_{\overline{Q}}(U') \to \Spec \cO_{F_v}.
\end{gather*}
 Moreover, ${}_v \bbM_{\overline{Q}}(U)$ is smooth over $\cO_{F_v}$, and ${}_v \bbM_{\overline{Q}}(U')$ is regular and semi-stable over $\cO_{F_v}$. The actions of the group $G_{{\overline{Q}}}(\bbA_F^{a, v, \infty})$ on the projective systems $\{ M_{\overline{Q}}(U)_{F_v} \}_{U \in {}_v\cJ_{\overline{Q}}}$ and  $\{ M_{\overline{Q}}(U')_{F_v} \}_{U' \in {}_v\cJ_{\overline{Q}}}$ extend canonically to actions on the projective systems $\{ {}_v \bbM_{\overline{Q}}(U) \}_{U \in {}_v\cJ_{\overline{Q}}}$ and $\{ {}_v \bbM_{\overline{Q}}(U') \}_{U \in {}_v\cJ_{\overline{Q}}}$.
\end{theorem}
\begin{proof}
Let $U \in {}_v\cJ_{\overline{Q}}$. The integral models ${}_v \bbM_{\overline{Q}}(U)$ and ${}_v \bbM_{\overline{Q}}(U')$ were constructed by Carayol in the case that the prime-to-$v$ level $U^v$ is sufficiently small \cite{Car86a}. It has been shown by Jarvis \cite[Lemma 12.2]{Jar99} that a model with the desired properties can be constructed for any $U \in {}_v\cJ_{\overline{Q}}$ as follows: let $V \in {}_v\cJ_{\overline{Q}}$ be a normal subgroup to which Carayol's construction applies. Then ${}_v \bbM_{\overline{Q}}(U)$ is the quotient of the action of $U$ on ${}_v \bbM_{\overline{Q}}(V)$, and similarly for ${}_v \bbM_{\overline{Q}}(U')$. (This uses the fact that $U_a = U_1^1(a)$, in a similar way to Lemma \ref{lem_triviality_of_isotropy_groups_definite_case}.) To show that ${}_v \bbM_{\overline{Q}}(U')$ is semi-stable over $\cO_{F_v}$, we must check (\cite[2.16]{Jon96}) that the special fiber ${}_v \bbM_{\overline{Q}}(U')_{\kappa(v)}$ is reduced and that its irreducible components intersect transversely. This follows directly from the results of  \cite[\S 10, 12]{Jar99}.
\end{proof}
Let $Q = \overline{Q} \cup \{ v \}$ and let $\overline{U} = \prod_w \overline{U}_w \in {}_v \cJ_{\overline{Q}}$. We now define good subgroups $\overline{U}' = \prod_w \overline{U}_w \subset G_{\overline{Q}}(\bbA_F^\infty)$ and $U = \prod_w U_w \subset G_Q(\bbA_F^\infty)$ by the following formulae:
\begin{itemize}
\item If $w \neq v$, then $U_w = \overline{U}'_w = \overline{U}_w$.
\item If $w = v$, then $\overline{U}'_w = U_0(v)$ and $U_w$ is the unique maximal compact subgroup of $G_Q(F_v)$.
\end{itemize}
We now describe the geometry of the special fiber ${}_v \bbM_{\overline{Q}}(\overline{U}')_{\kappa(v)}$ of ${}_v \bbM_{\overline{Q}}(\overline{U}')$. (We recall, cf. \S \ref{sec_normalizations}, that $\kappa(v)$ is the residue field of the fixed algebraic closure $\overline{F}_v$.)  The special fiber ${}_v \bbM_{\overline{Q}}(\overline{U}')_{\kappa(v)}$ is reduced, because ${}_v \bbM_{\overline{Q}}(\overline{U}')$ is semi-stable over $\cO_{F_v}$. If $\overline{U} \in {}_v\cJ_{\overline{Q}}$, then we write ${}_v \widetilde{\bbM}_{\overline{Q}}(\overline{U}')_{\kappa(v)}$ for the normalization of this special fiber, and $\mathrm{Sing}_{\overline{Q}}(\overline{U})$ for the set of singular points of the special fiber.
\begin{proposition}\label{prop_special_fiber_of_carayol_integral_model} There is an isomorphism of projective systems of smooth $\kappa(v)$-schemes with right $G_{\overline{Q}}(\bbA_F^{a,v,\infty})$-action
\[ \{ {}_v \widetilde{\bbM}_{\overline{Q}}(\overline{U}')_{ \kappa(v)} \}_{\overline{U} \in {}_v\cJ_{\overline{Q}}} \cong \{ {}_v \bbM_{\overline{Q}}(\overline{U})_{ \kappa(v)} \sqcup {}_v \bbM_{\overline{Q}}(\overline{U})_{ \kappa(v)} \}_{\overline{U} \in {}_v \cJ_{\overline{Q}}}. \]
Similarly, there is an isomorphism of projective systems of sets with right $G_{\overline{Q}}(\bbA_F^{a,v,\infty})$-action 
\[ \{ \mathrm{Sing}_{\overline{Q}}(\overline{U}) \}_{\overline{U} \in {}_v \cJ_{\overline{Q}}} \cong \{ X_Q(U) \}_{\overline{U} \in {}_v\cJ_{\overline{Q}}}. \]
\end{proposition}
\begin{proof}
The first part follows from \cite[Theorem 10.2]{Jar99} and \cite[Lemma 12.2]{Jar99}. The second part follows from the determination of the set $\mathrm{Sing}_{\overline{Q}}(\overline{U})$ in terms of supersingular points; see \cite[\S 11.2]{Car86a}.
\end{proof}

\subsection{Integral models, case $v \in Q$}
Let $v \neq a$ be a finite place of $F$, and let $Q$ be a finite set of finite places of $F$ with $v \in Q$ but $a \not\in Q$, and of odd cardinality. In this case, integral models of the curves $M_Q(U)$ over $\cO_{F_v}$ can be constructed using their uniformization by the $v$-adic Drinfeld upper half plane, as we now describe. 

Let $\check{F}$ denote the completion of the maximal unramified extension of $F_v$ inside $\overline{F}_v$, and let $\check{\cO}$ denote its ring of integers. The $F_v$-adic Drinfeld upper half plane is a formal scheme $\Omega^2$ over $\Spf \cO_{F_v}$, formally locally of finite type, and with connected special fiber. It is equipped with a left action of the group $\PGL_2(F_v)$; see \cite{Bou91}. We write $\Omega^2_{\check{\cO}} \to \Spf \check{\cO}$ for the base extension to $\check{\cO}$, and define
\begin{equation} \cM = \Omega^2_{\check{\cO}} \times \left[ (B_Q \otimes_F F_v)^\times / (\cO_Q \otimes_{\cO_F} \cO_{F_v})^\times \right].
\end{equation}
(As a formal scheme over $\Spf \check{\cO}$, $\cM$ is thus isomorphic to a countable disjoint union of copies of $\Omega^2_{\check{\cO}}$.) Let $\overline{Q} = Q - \{ v \}$. We have already fixed isomorphisms $G_Q(\bbA_F^{Q, \infty}) \cong \GL_2(\bbA_F^{Q, \infty})$ and $G_{\overline{Q}}(\bbA_F^{\overline{Q}, \infty}) \cong \GL_2(\bbA_F^{\overline{Q}, \infty})$; we now fix an isomorphism $G_Q(\bbA_F^{v, \infty}) \cong G_{\overline{Q}}(\bbA_F^{v, \infty})$ compatible with these identifications.

The reduced norm gives an isomorphism $(B_Q \otimes_F F_v)^\times / (\cO_Q \otimes_{\cO_F} \cO_{F_v})^\times \cong F_v^\times / \cO_{F_v}^\times$, and we let $G_{\overline{Q}}(F_v) \cong \GL_2(F_v)$ act on $\cM$ on the left via its usual action via $\PGL_2(F_v)$ on the first factor, and by the determinant and this isomorphism on the second factor. We make $G_Q(F_v) = (B_Q \otimes_{F} F_v)^\times$ act on $\cM$ on the right via the trivial action on the first factor, and by multiplication on the second factor.

\begin{theorem}\label{thm_models_of_shimura_curves_fujiwara_case}
Let $U \in \cJ_Q$. \(Thus $U_v = (\cO_Q \otimes_{\cO_F} \cO_{F_v})^\times$.\) Then the morphism $M_Q(U)_{\check{F}} \to \Spec \check{F}$ extends canonically to a flat projective morphism ${}_v \bbM_Q(U) \to \check{\cO}$ with ${}_v\bbM_Q(U)$ regular and semi-stable over $\check{\cO}$, and there is an isomorphism of formal schemes over $\Spf \check{\cO}$:
\begin{equation}\label{eqn_uniformization_of_shimura_curves} {}_v \widehat{\bbM}_Q(U) \cong G_{\overline{Q}}(F) \backslash ( \cM \times G_Q(\bbA_F^{v,\infty}) / U^v). 
\end{equation}
\(On the left hand side, the hat denotes $\varpi_v$-adic completion.\) There is a canonical right $G_Q(\bbA_F^{a,  \infty})$-action on the projective system $\{ {}_v \bbM_Q(U) \}_{U \in \cJ_Q}$ of $\check{\cO}$-schemes, extending the action on $\{ M_Q(U)_{\check{F}} \}_{U \in \cJ_Q}$, and making the system of isomorphisms (\ref{eqn_uniformization_of_shimura_curves}) $ G_Q(\bbA_F^{a,  \infty})$-equivariant.
\end{theorem}
It follows from Lemma \ref{lem_triviality_of_isotropy_groups_definite_case} that the quotient  (\ref{eqn_uniformization_of_shimura_curves}) can be formed in the following na\"ive sense. It is a finite disjoint union of quotients of the form $\Gamma \backslash \Omega^2_{\check{\cO}}$, where $\Gamma \subset G_{\overline{Q}}(F)$ is a subgroup whose image $\overline{\Gamma}$ in $\PGL_2(F_v)$ has the following property: the formal scheme $\Omega^2_{\check{\cO}}$ has a covering by Zariski opens $\Omega_i$ such that for all $\gamma \in \overline{\Gamma}$, if $\gamma \Omega_i \cap \Omega_i \neq \emptyset$, then $\gamma = 1$. Compare \cite[Lemma 3.2]{Tho12a}.
\begin{proof}[Proof of Theorem \ref{thm_models_of_shimura_curves_fujiwara_case}]
See \cite[Theorem 3.1]{Bou95}, and the remark immediately afterwards. We caution the reader that this reference uses the opposite convention for the definition of the curves $M_Q(U)$, cf. the discussion at the beginning of \S \ref{sec_shimura_curves}. However, one can easily check using the results of \cite[\S 3.3.1]{Cor05} that the two projective systems of curves over $F$ with $G_Q(\bbA_F^\infty)$-action become isomorphic over $\check{F}$, so the results of \cite{Bou95} apply without modification. For this reason, we do not make any assertion about the compatibility of descent data on either side of the isomorphism (\ref{eqn_uniformization_of_shimura_curves}). The scheme ${}_v\bbM_Q(U)$ is semi-stable over $\check{\cO}$ because the formal scheme $\Omega^2_{\check{\cO}}$ has a reduced special fiber, with irreducible components which intersect transversely.
\end{proof}
We write $\mathrm{BT}$ for the Bruhat-Tits tree of the group $\PGL_2(F_v)$; it is an infinite tree with $\PGL_2(F_v)$-action, with vertex set $\mathrm{BT}(0)$ equal to the set of $\cO_{F_v}$-lattices $M \subset F_v^2$, taken up to $F_v^\times$-multiple. Two homothety classes $[M_0]$ and $[M_1]$ are joined by an edge if and only if we can choose representatives $M_0, M_1$ such that $M_0 \subset M_1$ and $[ M_0 : M_1 ] = q_v$. The irreducible components of the special fiber of $\Omega^2$ are in canonical bijection with the set $\mathrm{BT}(0)$. Consequently the irreducible components of the special fiber of $\cM$ are in bijection with the set $\mathrm{BT}(0) \times \GL_2(F_v) / \GL_2(F_v)^\circ$, where $\GL_2(F_v)^\circ \subset \GL_2(F_v)$ is the open subgroup consisting of matrices with determinant in $\cO_{F_v}^\times \subset F_v^\times$. Moreover, this bijection is $\GL_2(F_v)$-equivariant.

If $U \in \cJ_Q$, we write $\Irr_Q(U)$ for the set of irreducible components of the special fiber ${}_v {\bbM}_Q(U)_{\kappa(v)}$. We write $\overline{U} = \prod_w \overline{U}_w \subset G_{\overline{Q}}(\bbA_F^\infty)$ for the good subgroup with $\overline{U}_w = U_w$ if $w \neq v$ and $\overline{U}_v = \GL_2(\cO_{F_v})$.
\begin{corollary}\label{cor_special_fiber_of_drinfeld_model}
There is an isomorphism of projective systems of sets with right $G_{\overline{Q}}(\bbA_F^{a, v, \infty})$-action:
\[ \{  \Irr_Q(U) \}_{U \in \cJ_Q} \cong \{ X_{\overline{Q}}(\overline{U}) \sqcup X_{\overline{Q}}(\overline{U}) \}_{U \in \cJ_Q}. \]
\end{corollary}
\begin{proof}
By the above, there is a natural bijection
\[ \Irr_Q(U) \cong G_{\overline{Q}}(F) \backslash( \mathrm{BT}(0) \times \GL_2(F_v) / \GL_2(F_v)^\circ \times G_{Q}(\bbA_F^{v, \infty}) / U^v ). \]
The group $\GL_2(F_v)$ has two orbits on the set $\mathrm{BT}(0) \times \GL_2(F_v) / \GL_2(F_v)^\circ$, representatives being given by $(\cO_{F_v}^2, 1)$ and $(\cO_{F_v}^2, \diag(\varpi_v, 1))$. Both of these points have stabilizer $\GL_2(\cO_{F_v})$, so we obtain an isomorphism
\[ \Irr_Q(U) \cong G_{\overline{Q}}(F) \backslash( (\GL_2(F_v) / \overline{U}_v)^2 \times G_{\overline{Q}}(\bbA_F^{v, \infty}) / \overline{U}^v ). \]
It is easy to see that this system of isomorphisms is $G_{\overline{Q}}(\bbA_F^{a, v, \infty})$-equivariant, and we therefore recover the statement of the corollary.
\end{proof}

\subsection{Hecke algebras and modules}\label{sec_hecke_algebras_and_modules}

Let $Q$ be a finite set of finite places of $F$ not containing $a$. We now fix a prime $p$ and a coefficient field $E$ for the rest of \S \ref{sec_shimura_curves_and_varieties}, and define collections of Hecke algebras and modules for these algebras with $\cO$-coefficients. Let $S$ be a finite set of finite places of $F$ containing $Q$. We write $\bbT^{S,\text{univ}}$ for the polynomial algebra over $\cO$ in the infinitely many indeterminates $T_v, S_v$ for all finite places $v \not\in S$ of $F$. We write $\bbT_Q^{S,\text{univ}}$ for the polynomial algebra over $\bbT^{S,\text{univ}}$ in the indeterminates $\mathbf{U}_v,$ $v\in Q$. Thus $\bbT_\emptyset^{S, \text{univ}} = \bbT^{S, \text{univ}}$.

Let $U \in \cJ_Q$. If $\# Q$ is odd, then we define $H_Q(U) = H^1(M_Q(U)_{\overline{F}}, \cO)$. If $\# Q$ is even, then we define $H_Q(U) = H^0(X_Q(U), \cO)$. In either case, $H_Q(U)$ is a finite free $\cO$-module. Let $S$ be a finite set of places of $F$, containing $Q$, such that for all $v \not\in S$, $U_v = \GL_2(\cO_{F_v})$. Then the Hecke algebra $\bbT_Q^{S,\text{univ}}$ acts on $H_Q(U)$, each Hecke operator $T_v,$ $S_v$ ($v \not \in S$) or ${\mathbf{U}_v}$ ($v \in Q)$ acting by the element of the local Hecke algebra $\cH(G_Q(F_v), U_v)$ of the same name.

If $\# Q$ is odd, then this action commutes with the action of the group $G_F$, and the Eichler--Shimura relation holds: for all finite places $v \not\in S \cup S_p$ of $F$, the action of $G_{F_v}$ is unramified and we have the relation 
\begin{equation}\label{eqn_eichler_shimura_relation}
 \Frob_v^2 - S_v^{-1} T_v \Frob_v + q_v S_v^{-1} = 0
 \end{equation}
in $\End_\cO(H_Q(U))$. This can be deduced from local-global compatibility and the decomposition (\ref{eqn_decomposition_of_shimura_curve_cohomology}), but of course a more geometric version of this statement plays a key role in the proof of (\ref{eqn_decomposition_of_shimura_curve_cohomology}); see \cite[1.6.4]{Car86}. (Compare also the corresponding relation (\ref{eqn_local_global_compatibility_at_unramified_places}) for the representations $r_\iota(\pi)$.)

If $M$ is any $\bbT^{S, \text{univ}}$-module, then we write $\bbT^S(M)$ for the image of $\bbT^{S, \text{univ}}$ in $\End_\cO(M)$. If $M$ is a $\bbT^{S, \text{univ}}_Q$-module, then we define $\bbT^S_Q(M)$ similarly.  We note in particular that if $U \subset G_Q(\bbA_F^\infty)$ is a good subgroup, then (regardless of the parity of $Q$) the algebras $\bbT^S(H_Q(U))$ and $\bbT^S_Q(H_Q(U))$ are reduced and $\cO$-torsion free.

Let $\ffrm \subset \bbT^{S, \text{univ}}$ be the kernel of a homomorphism $\bbT^{S,\text{univ}} \to k'$, where $k'/k$ is a finite extension. If $\ffrm$ is in the support of $H_Q(U)$ for some $Q$ and some $U \in \cJ_Q$ (equivalently: if the image of $\ffrm$ in $\bbT^S(H_Q(U))$ is not the unit ideal) then there is an associated semi-simple Galois representation $\overline{\rho}_\ffrm : G_F \to \GL_2(\bbT^{S, \text{univ}}/\ffrm)$, uniquely determined up to isomorphism by the following relation: for all $v \not\in S \cup S_p$, $\overline{\rho}_\ffrm|_{G_{F_v}}$ is unramified, and $\overline{\rho}_\ffrm(\Frob_v)$ has characteristic polynomial $X^2 - T_v X + q_v S_v\in (\bbT^{S, \text{univ}}/\ffrm)[X]$. If $\overline{\rho}_\ffrm$ is absolutely reducible, we say that $\ffrm$ is Eisenstein; otherwise, we say that $\ffrm$ is non-Eisenstein. If $\ffrm$ is the kernel of a homomorphism $\bbT_Q^{S, \text{univ}} \to k'$, we say that $\ffrm$ is Eisenstein (resp. non-Eisenstein) if $\ffrm \cap \bbT^{S, \text{univ}}$ is Eisenstein (resp. non-Eisenstein). In particular, in either case there is a Galois representation $\overline{\rho}_\ffrm$ valued in $\GL_2(\bbT_Q^{S, \text{univ}}/\ffrm)$.

\begin{proposition}\label{prop_carayol_and_blr}
Let $\#Q$ be odd, and let $\ffrm \subset \bbT^S(H_Q(U))$ be a non-Eisenstein maximal ideal. Then we can find the following data.
\begin{enumerate} \item A continuous representation $\rho_\ffrm : G_F \to \GL_2(\bbT^S(H_Q(U))_\ffrm)$ lifting $\overline{\rho}_\ffrm$ and satisfying the following condition: for each finite place $v \not\in S \cup S_p$ of $F$, $\rho_\ffrm|_{G_{F_v}}$ is unramified, and $\rho_\ffrm(\Frob_v)$ has characteristic polynomial
\[ X^2 -  T_v X + q_v S_v. \]
\item A finite $\bbT^S(H_Q(U))_\ffrm$-module $M$, together with an isomorphism of $\bbT^S(H_Q(U))_\ffrm[G_F]$-modules 
\[ H_Q(U)_\ffrm \cong \rho_\ffrm \otimes_{\bbT^S(H_Q(U))_\ffrm} (\epsilon \det \rho_\ffrm)^{-1} \otimes_{\bbT^S(H_Q(U))_\ffrm} M \]
\end{enumerate}
\end{proposition}
\begin{proof}
The existence of $\rho_\ffrm$ follows from a well-known argument of Carayol; see \cite[\S 2]{Car94}. For each $v \not\in S \cup S_p$, $(\rho_\ffrm \otimes_{\bbT^S(H_Q(U))_\ffrm} (\epsilon \det \rho_\ffrm)^{-1})(\Frob_v)$ has characteristic polynomial $ X^2 -  S_v^{-1} T_v  X + q_v S^{-1}_v$. 
The second part of the proposition now follows from the Chebotarev density theorem, the Eichler--Shimura relation (\ref{eqn_eichler_shimura_relation}) and Theorem \ref{thm_eicher_shimura_ring_case}.
\end{proof}

\subsection{Relations between Hecke modules}

Let $Q$ be a finite set of finite places of $F$ not containing $a$. Let $v \in Q$ be prime to $p$, and define $\overline{Q} = Q - \{ v \}$. If $U = \prod_w U_w \in \cJ_Q$, then we define $\overline{U} = \prod_w \overline{U}_w,$ $\overline{U}' = \prod_w \overline{U}'_v \in \cJ_{\overline{Q}}$ by the following formulae.
\begin{itemize}
\item If $w \neq v$, then $\overline{U}_w=\overline{U}'_w= \overline{U}_w$.
\item $\overline{U}_v = \GL_2(\cO_{F_v})$ and $\overline{U}_v' = U_0(v)$.
\end{itemize}
Let $\ffrm$ be a non-Eisenstein maximal ideal of $\bbT_Q^{S, \text{univ}}$ which is in the support of $H_Q(U)$, and let $\overline{\ffrm} = \ffrm \cap \bbT_{\overline{Q}}^{S, \text{univ}}$. In this section, we discuss relations between the modules $H_Q(U)$ and $H_{\overline{Q}}(\overline{U})$. 
\begin{theorem}\label{thm_jarvis_exact_sequence}
Suppose that $\# Q$ is even. Then there is an exact sequence of $\bbT_{Q}^{S, \text{univ}}[G_{F_{v}}]$-modules, finite free as $\cO$-modules:
\begin{equation}\label{eqn_jarvis_exact_sequence} 0 \to H_Q(U)_\ffrm \to H_{\overline{Q}}(\overline{U}')_\ffrm^{I_{F_v}} \to H_{\overline{Q}}(\overline{U})_\ffrm^2 \to 0. 
\end{equation}
The action of $\Frob_v^{-1}$ on $H_Q(U)_\ffrm$ is equal to the action of ${\mathbf{U}_v}$; the action of ${\mathbf{U}_v} \in \bbT_Q^{S, \text{univ}}$ on the middle term in the sequence is by the element of $\cH(G_{\overline{Q}}(F_v), \overline{U}'_v)$ of the same name. Furthermore, for any $N \geq 1$ the natural map 
\[ H_{\overline{Q}}(\overline{U}')_\ffrm^{I_{F_v}} \otimes_\cO \cO/\lambda^N \to ( H_{\overline{Q}}(\overline{U}')_\ffrm \otimes_\cO \cO/\lambda^N)^{I_{F_v}} \]
is an isomorphism.
\end{theorem}
\begin{remark}
Before giving the proof of Theorem \ref{thm_jarvis_exact_sequence}, we clarify the meaning of the $\bbT_{Q}^{S, \text{univ}}[G_{F_{v}}]$-action on the terms of the exact sequence (\ref{eqn_jarvis_exact_sequence}). In the course of the proof, we construct an exact sequence of $\bbT_{\overline{Q}}^{S, \text{univ}}$-modules:
\begin{equation}\label{eqn_pre_jarvis_exact_sequence} 0 \to H_Q(U)_{\overline{\ffrm}} \to H_{\overline{Q}}(\overline{U}')_{\overline{\ffrm}}^{I_{F_v}} \to H_{\overline{Q}}(\overline{U})_{\overline{\ffrm}}^2 \to 0. 
\end{equation}
So far, we have defined an action of $\bbT_{Q}^{S, \text{univ}}$ on the first and second terms, and an action of $\bbT_{\overline{Q}}^{S, \text{univ}}[G_{F_v}]$ on the second and third terms. The second arrow is $G_{F_v}$-equivariant, so there is an induced $G_{F_v}$-action on the first term. Similarly, we will show that the first arrow is equivariant for the action of $\mathbf{U}_v$, so that there is an induced action $\bbT_{Q}^{S, \text{univ}}$ on the third term, and the sequence (\ref{eqn_pre_jarvis_exact_sequence}) becomes an exact sequence of $\bbT_{Q}^{S, \text{univ}}[G_{F_v}]$-modules. Localizing further at the maximal ideal $\ffrm$, we obtain the desired sequence (\ref{eqn_jarvis_exact_sequence}).
\end{remark}
\begin{proof}[Proof of Theorem \ref{thm_jarvis_exact_sequence}]
We use the theory of vanishing cycles, following the summary of \cite[\S 1]{Raj01}. Let $\Lambda = \cO$ or $\cO/\lambda^N$. Let $X \to \Spec \cO_{F_v}$ be a flat projective curve, with $X$ regular and semi-stable over $\cO_{F_v}$. Then $X_{\kappa(v)}$ is reduced; let $\Sigma$ denote its set of singular points. We write $R \Psi \Lambda$ and $R \Phi \Lambda$ for the complexes of nearby and vanishing cycles, respectively, on $X_{\kappa(v)}$. There are two exact sequences, the specialization exact sequence:
\begin{equation}\label{eqn_specialization_exact_sequence} \xymatrix@R-2pc{ 0 \ar[r] & H^1(X_{\kappa(v)}, \Lambda)(1) \ar[r] & H^1(X_{\overline{F}_v}, \Lambda)(1) \ar[r]^-\beta & \oplus_{x \in \Sigma} (R^1 \Phi \Lambda)_x(1) \\ \ar[r] & H^2(X_{{\kappa(v)}}, \Lambda)(1) \ar[r] & H^2(X_{\overline{F}_v}, \Lambda)(1) \ar[r] & 0; } 
\end{equation}
and the cospecialization exact sequence:
\begin{equation}\label{eqn_cospecialization_exact_sequence}
 \xymatrix@R-2pc{ 0 \ar[r] & H^0(\widetilde{X}_{{\kappa(v)}}, R \Psi \Lambda)) \ar[r] & H^0(\widetilde{X}_{{\kappa(v)}}, \Lambda) \\ \ar[r] & \oplus_{x \in \Sigma} H^1_x(X_{{\kappa(v)}}, R \Psi \Lambda) \ar[r]^-{\beta'} & H^1(X_{{\kappa(v)}}, R \Psi \Lambda) \ar[r] & H^1(X_{{\kappa(v)}}, \Lambda) \ar[r] & 0. } 
\end{equation}
(Here $\widetilde{X}_{{\kappa(v)}}$ denotes the normalization of $X_{{\kappa(v)}}$.) Moreover, there is a commutative diagram:
\begin{equation}\label{eqn_local_decomposition_of_monodromy}
\xymatrix{ H^1(X_{\overline{F}_v}, \Lambda)(1)\ar[d]^N \ar[r]^-c & \text{image }\beta \ar@{^{(}->}[r]\ar[d]^\mu & \oplus_{x \in \Sigma} R^1 \Phi(\Lambda)_x (1) \ar[d]^{\oplus N_x} \\ 
H^1(X_{\overline{F}_v}, \Lambda) & \text{coimage }\beta' \ar[l]^-{c'} & \oplus_{x \in \Sigma} H^1_x(X_{{\kappa(v)}}, R \Psi \Lambda) \ar@{->>}[l]. }
\end{equation}
(For the definitions of the various arrows here, we refer to \cite[\S 1]{Raj01}. The arrow we have denoted $\mu$ is denoted by $\lambda$ in this reference.) Taking $\Lambda = \cO$, we define the `component group' $\Phi = \text{ coker }\mu$, a finite group. There is a short exact sequence (cf. \cite[Proposition 4]{Raj01}):
\begin{equation}
\xymatrix@1{ 0 \ar[r] & H^1(X_{\kappa(v)}, \cO/\lambda^N)\ar[r] & H^1(X_{\overline{F}_v}, \cO/\lambda^N)^{I_{F_v}} \ar[r] & \Phi[\lambda^N](-1) \ar[r] & 0.}
\end{equation}
We now take $X = {}_v \bbM_{\overline{Q}}(\overline{U}')$, which is allowable by Theorem \ref{thm_existence_of_carayol_integral_models}. The sequences (\ref{eqn_specialization_exact_sequence}) and (\ref{eqn_cospecialization_exact_sequence}) become exact sequences of $\bbT_{\overline{Q}}^{S, \text{univ}}$-modules. Moreover, the only maximal ideals of $\bbT_{\overline{Q}}^{S, \text{univ}}$ in the support of the modules $H^2(X_{{\kappa(v)}}, \Lambda)$ and $H^0(\widetilde{X}_{\kappa(v)}, \Lambda)$ are Eisenstein (cf. \cite[\S 18]{Jar99}) and the operators $N_x$ in  (\ref{eqn_local_decomposition_of_monodromy}) are isomorphisms of free $\Lambda$-modules (as $X$ is regular, cf. \cite[\S 1.1]{Raj01}). In particular, the component group $\Phi$ vanishes after localization at ${\overline{\ffrm}}$ and the natural maps 
\[ H^1({}_v \bbM_{\overline{Q}}(\overline{U}')_{{\kappa(v)}}, \cO)_{\overline{\ffrm}} \otimes_\cO \cO/\lambda^N \to  H^1({}_v \bbM_{\overline{Q}}(\overline{U}')_{\kappa(v)}, \cO/\lambda^N)_{\overline{\ffrm}} \to H^1({}_v \bbM_{\overline{Q}}(\overline{U}')_{\overline{F}_v}, \cO/\lambda^N)_{\overline{\ffrm}}^{I_{F_v}} \]
are isomorphisms. This establishes the final assertion in the statement of the theorem.

Let $r : \widetilde{X}_{{\kappa(v)}} \to X_{{\kappa(v)}}$ denote the normalization map; it is an isomorphism away from the subset $\Sigma \subset X_{{\kappa(v)}}$. We consider the sheaf $T$ of free $\Lambda$-modules defined by the short exact sequence:
\begin{equation}
 \xymatrix@1{ 0 \ar[r] & \Lambda \ar[r] & r_\ast r^\ast \Lambda \ar[r] & T \ar[r] & 0. }
\end{equation}
Taking the long exact sequence in cohomology, we obtain:
\begin{equation}\label{eqn_exact_sequence_of_jarvis_skyscraper_sheaf}
\begin{aligned}
 \xymatrix@R-2pc{0 \ar[r] & H^0(X_{{\kappa(v)}}, \Lambda) \ar[r] & H^0(\widetilde{X}_{\kappa(v)}, \Lambda) \ar[r] & H^0(X_{\kappa(v)}, T) \\
\ar[r] & H^1(X_{\kappa(v)}, \Lambda) \ar[r] &  H^1(\widetilde{X}_{\kappa(v)}, \Lambda) \ar[r] & 0. }
\end{aligned}
\end{equation}
We have $H^0(X_{\kappa(v)}, T) \cong H_Q(U)$, by Proposition \ref{prop_special_fiber_of_carayol_integral_model}. Localizing the exact sequence (\ref{eqn_exact_sequence_of_jarvis_skyscraper_sheaf}) of $\bbT_{\overline{Q}}^{S, \text{univ}}$-modules at the maximal ideal ${\overline{\ffrm}}$, we obtain the short exact sequence:
\begin{equation}\label{eqn_proof_of_jarvis_exact_sequence}
\xymatrix@1{ 0 \ar[r] & H_Q(U)_{\overline{\ffrm}} \ar[r] & H_{\overline{Q}}(\overline{U}')_{\overline{\ffrm}}^{I_{F_v}} \ar[r] & H_{\overline{Q}}(\overline{U})^2_{\overline{\ffrm}}  \ar[r] & 0. }
\end{equation}
To complete the proof of the theorem, it remains to check that the first arrow in the sequence (\ref{eqn_proof_of_jarvis_exact_sequence}) is equivariant for the action of ${\mathbf{U}_v}$, and that $\Frob_v^{-1}$ acts as ${\mathbf{U}_v}$ on its image. Having done this, the exact sequence of the theorem will be obtained by localizing at the maximal ideal $\ffrm \subset \bbT_Q^{S, \text{univ}}$. (It is possible to write down explicitly the induced action of ${\mathbf{U}_v}$ on the third term of the sequence, but since we won't need this we don't do it. For a similar calculation, see \cite[Lemma 1]{Rib90}.) This check can be carried out after inverting $p$, and the desired result now follows from the calculations of \cite[\S 5--6]{Car86}. (Strictly speaking this reference assumes that the set $Q$ has at most 1 element, but the results are the same provided that $v \not\in Q$.) 
\end{proof}
We learned the following theorem from the unpublished work of Fujiwara \cite{Fuj06}.
\begin{theorem}\label{thm_fujiwara_exact_sequence}
Suppose that $\# Q$ is odd. Then for any $N \geq 1$, there is a map of $\bbT_Q^{S, \text{univ}}$-modules
\[ (H_Q(U)_\ffrm \otimes_\cO \cO/\lambda^N)^{I_{F_v}}\to H_{\overline{Q}}(\overline{U})_\ffrm^2 \otimes_\cO \cO/\lambda^N, \]
with kernel contained in the submodule
\[ H_Q(U)_\ffrm^{I_{F_v}} \otimes_\cO \cO/\lambda^N \subset (H_Q(U)_\ffrm \otimes_\cO \cO/\lambda^N)^{I_{F_v}}. \]
The action of $\Frob_v^{-1}$ on this kernel equals the action of ${\mathbf{U}_v}$.
\end{theorem}
The action of $\mathbf{U}_v$ on $H_{\overline{Q}}(\overline{U})_\ffrm^2$ will be defined in the proof.
\begin{proof}[Proof of Theorem \ref{thm_fujiwara_exact_sequence}]
Let $\check{F}$ denote the completion of the maximal unramified extension of $F_v$ inside $\overline{F}_v$, and let $\check{\cO}$ denote its ring of integers. Suppose that $f : X \to \check{\cO}$ is a flat projective morphism, with $X$ regular and semi-stable over $\check{\cO}$. Then there are exact sequences for $\Lambda = \cO$ or $\Lambda = \cO/\lambda^N$ (cf. \cite[Lemma 3.8]{Fuj06}):
\begin{equation}\label{eqn_regular_curve_1}
0 \to H^0(X_{\overline{{\check{F}}}}, \Lambda)(-1)_{I_{F_{v}}} \to H^1(X_{\check{F}}, \Lambda) \to H^1(X_{\overline{{\check{F}}}}, \Lambda)^{I_{F_{v}}} \to 0;
\end{equation}
and (if $J$ denotes the set of irreducible components of $X_{\kappa(v)}$):
\begin{equation}\label{eqn_regular_curve_2}
0 \to H^1(X, \Lambda) \to H^1(X_{\check{F}}, \Lambda) \to \prod_{Y \in J} H^0(Y^\text{reg}, \Lambda)(-1). 
\end{equation}
Indeed, the first exact sequence is part of the Hochschild-Serre spectral sequence for the covering $X_{\overline{\check{F}}} \to X_{\check{F}}$. For the second, it suffices to treat the case $\Lambda = \cO/\lambda^N$. For each $Y \in J$, let $\overline{\eta}_Y$ be a geometric point of $X$ above the generic point of $Y^\text{reg}$. Then there is a canonical isomorphism $H^1(\Spec \Frac \cO_{X,\overline{\eta}_Y}, \Lambda) \cong H^0(Y^\text{reg}, \Lambda)(-1)$.
 By the Zariski--Nagata purity theorem \cite[Theorem 41.1]{Nag62}, a $\Lambda$-torsor $T \to X_{\check{F}}$ extends over $X$ if and only if it extends over $\Spec \cO_{X, \overline{\eta}_Y}$ for each $Y \in J$, if and only if its image in $H^1(\Spec  \Frac \cO_{X,\overline{\eta}_Y}, \Lambda)$ is trivial; and if this extension exists, then it is unique. This implies the exactness of (\ref{eqn_regular_curve_2}).

We apply this with $X = {}_v \bbM_Q(U)$, the regular model of $M_Q(U)_{\check{F}}$ defined in Theorem \ref{thm_models_of_shimura_curves_fujiwara_case}. The sequences (\ref{eqn_regular_curve_1}) and (\ref{eqn_regular_curve_2}) then become exact sequences of $\bbT_{\overline{Q}}^{S, \text{univ}}$-modules. Moreover, we have 
\[ H^0(X_{\overline{{\check{F}}}}, \Lambda)_{\overline{\ffrm}} = 0 \]
(again by \cite[\S 18]{Jar99}) and 
\[ \prod_{Y \in J} H^0(Y^\text{reg}, \cO) \cong H_{\overline{Q}}(\overline{U})^2 \]
(by Corollary \ref{cor_special_fiber_of_drinfeld_model}). We obtain for any $N \geq 1$ a commutative diagram with exact rows:
\begin{equation}\label{eqn_regular_curve_3}
\begin{aligned}
\xymatrix{
0 \ar[r]& H^1({}_v \bbM_Q(U), \cO)_{\overline{\ffrm}} \ar[r]\ar[d]^\alpha& H_Q(U)_{\overline{\ffrm}}^{I_{F_v}} \ar[r]\ar[d]^\mu& H_{\overline{Q}}(\overline{U})^2_{\overline{\ffrm}} \ar[d] \\
0 \ar[r] & H^1({}_v \bbM_Q(U), \cO/\lambda^N)_{\overline{\ffrm}} \ar[r]^-\nu & (H_Q(U)_{\overline{\ffrm}} \otimes_\cO \cO/\lambda^N)^{I_{F_v}} \ar[r]& H_{\overline{Q}}(\overline{U})^2_{\overline{\ffrm}} \otimes_\cO \cO/\lambda^N}
\end{aligned}
\end{equation}
Since the operator ${\mathbf{U}_v}$ on $H_Q(U)$ is induced by an automorphism of ${}_v \bbM_Q(U)$, this is in fact a diagram of $\bbT_Q^{S, \text{univ}}$-modules, where the action of $\mathbf{U}_v$ on $H_{\overline{Q}}(\overline{U})_{\overline{\ffrm}}^2$ is defined via the isomorphism $\prod_{Y \in J} H^0(Y^\text{reg}, \cO) \cong H_{\overline{Q}}(\overline{U})^2$. (The exact sequence (\ref{eqn_regular_curve_2}) is functorial in automorphisms of $X$.)
The map of the theorem is obtained from the bottom row after localizing at the ideal $\ffrm$ of this algebra. 

To finish the proof of the theorem, we must show that $\text{image }\nu \subset \text{image }\mu$, or even that $\alpha$ is surjective. By the proper base change theorem, we have an isomorphism
\[ H^2({}_v \bbM_Q(U), \cO)_{\overline{\ffrm}} \cong H^2({}_v \bbM_Q(U)_{\kappa(v)}, \cO)_{\overline{\ffrm}}, \]
and this group is $\cO$-torsion free. This implies that $\alpha$ is surjective.
\end{proof}

\subsection{Level-raising}\label{sec_level_raising}
We now fix a finite set $R$ of finite places of $F$ of even cardinality, disjoint from $S_p \cup \{ a \}$, and a good subgroup $U = \prod_w U_w \in \cJ_R$. (We allow the possibility that $R$ may be empty.) We also fix another finite set $Q$ of finite places of $F$ satisfying the following conditions.
\begin{itemize}
\item $Q \cap ( S_p \cup \{ a \} \cup R ) = \emptyset$.
\item For each $w \in Q$, $q_w \not \equiv 1 \text{ mod }p$ and $U_w = \GL_2(\cO_{F_w})$.
\item $\# Q$ is even.
\end{itemize}
If $J \subset Q$ is a subset, then we define a subgroup $U_J  = \prod_w U_{J, w}\subset G_{R \cup J}(\bbA_F^\infty)$ by the following formulae.
\begin{itemize}
\item If $w \not\in J$, then $U_{J, w} = U_w$.
\item If $w \in J$, then $U_{J, w}$ is the unique maximal compact subgroup of $G_{R \cup J}(F_w)$.
\end{itemize}
Let $S$ be a finite set of finite places of $F$ containing $S_p \cup R \cup Q$ and the places such that $U_w \neq \GL_2(\cO_{F_w})$. Let $\ffrm = \ffrm_\emptyset \subset \bbT_\emptyset^{S, \text{univ}}$ be a non-Eisenstein maximal ideal which is in the support of $H_R(U)$. Thus $\overline{\rho}_\ffrm$ is absolutely irreducible, and for each $v \in Q$, $\overline{\rho}_\ffrm|_{G_{F_{v}}}$ is unramified. After enlarging the coefficient field $E$, we can assume that for all $v \in Q$, the eigenvalues $\alpha_v, \beta_v$ of $\overline{\rho}_{\ffrm}(\Frob_v)$ lie in $k$. We now make the following additional assumption:
\begin{itemize}
\item For each $v \in Q$, $\beta_v/\alpha_v = q_v$. 
\end{itemize}
In other words, $\overline{\rho}_\ffrm$ satisfies the well-known level-raising congruence at the place $v$. Because of our assumption $q_v \not\equiv 1 \text{ mod }p$, $\alpha_v$ and $\beta_v$ are distinct. We remark that if $q_v \equiv -1 \text{ mod }p$ then the roles of $\alpha_v$ and $\beta_v$ could be reversed. This will be the case in applications, and we emphasize that our labeling of $\alpha_v$ and $\beta_v$ represents a fixed choice in this case.
\begin{lemma}\label{lem_level_raising_with_prescribed_eigenvalue}
For each subset $J \subset Q$, let $\ffrm_J \subset \bbT_J^{S, \text{univ}}$ be the maximal ideal generated by $\ffrm_\emptyset$ and the elements ${\mathbf{U}_v} - \alpha_v,$ $v \in J$. \(This makes sense since $\lambda \subset \ffrm$.\) Then $\ffrm_J$ is in the support of $H_{R \cup J}(U_J)$.
\end{lemma}
\begin{proof}
Choose an isomorphism $\iota : \barqp \to \bbC$. The existence of $\ffrm$ implies the existence of a cuspidal automorphic representation $\pi$ of $\GL_2(\bbA_F)$ of weight 2, satisfying the following conditions.
\begin{itemize}
\item There is an isomorphism $\overline{r_\iota(\pi)} \cong \overline{\rho}_\ffrm$.
\item For each finite place $v \not\in R$ of $F$, $\pi_v^{U_v} \neq 0$. For each $v \in R$, $\pi_v$ is an unramified twist of the Steinberg representation.
\end{itemize}
To prove the lemma, it is enough to show the existence of a cuspidal automorphic representation $\pi$ of $\GL_2(\bbA_F)$ of weight 2, satisfying the following conditions.
\begin{itemize}
\item There is an isomorphism $\overline{r_\iota(\pi)} \cong \overline{\rho}_\ffrm$.
\item For each finite place $v \not\in R \cup J$ of $F$, $\pi_v^{U_v} \neq 0$. For each place $v \in R \cup J$, there is an unramified character $\chi_v : F_v^\times \to \bbC^\times$ such that $\pi_v \cong \St_2(\chi_v)$.
\item For each $v \in J$, $\iota^{-1} \chi(\varpi_v)$ is congruent to $\alpha_v$ modulo the maximal ideal of $\overline{\bbZ}_p$. (Compare Lemma \ref{lem_u_eigenvalue_on_twist_of_steinberg}.)
\end{itemize}
Let us say that a cuspidal automorphic representation $\pi$ satisfying these 3 conditions is $J$-good. We will prove by induction on $\# J$ that there exists a $J$-good $\pi$. More precisely, suppose that $\# J$ is even, and let $v_0, v_1 \in Q - J$ be distinct elements. We will establish the existence of a $J \cup \{ v_0 \}$-good representation $\pi_0$, and a $J \cup \{ v_0, v_1 \}$-good representation $\pi$.

We first show that there exists a representation $\pi_0$ which is $J \cup \{ v_0 \}$-good. Let $U_J' = \prod_v U_{J, v}'$ be the open compact subgroup of $G_{R \cup J}(\bbA_F^\infty)$ defined by $U_{J, v}' = U_{J, v}$ if $v \neq v_0$, and $U_{J, v_0}' = U_0(v_0)$. There is a natural degeneracy map
\begin{gather*}
i : H_{R \cup J}(U_J)_{\ffrm_J}^2 \to H_{R \cup J}(U_J')_{\ffrm_J} \\
(f, g) \mapsto f + \left( \begin{array}{cc} 1 & 0 \\ 0 & \varpi_{v_0} \end{array} \right) \cdot g.
\end{gather*}
The map $i$ is injective, even after tensoring with $- \otimes_\cO k$; this is the analogue of Ihara's lemma in this context, cf. \cite[Lemma 3.26]{Ski99}. The image of $i$ is preserved by the operator $\mathbf{U}_{v_0}$, which acts on this image by the matrix of unramified Hecke operators
\begin{equation}\label{eqn_operator_on_old_subspace}
\left( \begin{array}{cc} T_{v_0} & q_{v_0} S_{v_0} \\ -1 & 0 \end{array}\right). 
\end{equation}
Moreover, there are natural perfect pairings $\langle \cdot, \cdot \rangle_{U_J}$ and $\langle \cdot, \cdot \rangle_{U'_J}$ on $H_{R \cup J}(U_J)_{\ffrm_J}$ and $H_{R \cup J}(U_J')_{\ffrm_J}$; writing $i^\ast$ for the adjoint of $i$ with respect to these pairings, we have (cf. \cite[Lemma 2]{Tay89})
\begin{equation}\label{eqn_adjoint_of_degeneracy_operator} i^\ast \circ i = \left( \begin{array}{cc} q_{v_0} + 1 & S_{v_0}^{-1} T_{v_0} \\ T_{v_0} & q_{v_0} + 1 \end{array} \right). 
\end{equation}
The determinant of this matrix is $(q_{v_0} + 1)^2 - S_{v_0}^{-1} T_{v_0}^2$, which is topologically nilpotent on $H_J(U_J)_{\ffrm_J}$. It follows that $i^\ast \circ i$ is not surjective. Since the image of $i$ is saturated, this implies the existence of a cuspidal automorphic representation $\pi_0$ satisfying the first 2 points in the definition of $J \cup \{ v_0 \}$-good, and satisfying the third point for each $v \in J$. If $q_{v_0} \not\equiv -1 \text{ mod }p$, then it is immediate that the third point also holds for $v = v_0$. We therefore now assume that $q_{v_0} \equiv -1 \text{ mod }p$. After possibly enlarging $E$, we can assume the existence of a $\bbT_J^{S,\text{univ}}$-eigenvector $f \in H_{R \cup J}(U_J)_{\ffrm_J} - \lambda H_{R \cup J}(U_J)_{\ffrm_J}$. 

Let $\overline{f}$ denote the image of $f$ in $H_{R \cup J}(U_J) \otimes_\cO k$, and let $\overline{g} = i(\alpha_v \overline{f}, - \overline{f})$. Since $T_{v_0} \overline{f} = 0$, we have $i^\ast \overline{g} = 0$. On the other hand, it follows from the expression (\ref{eqn_operator_on_old_subspace}) that $\mathbf{U}_{v_0} \overline{g} = \alpha_v \overline{g}$. Writing $g \in \text{image }i$ for an arbitrary lift of $g$, it follows that $\langle g, x \rangle_{U_J'} \in \lambda$ for all $x \in \text{image }i$. Let $\pi \in \cO$ be a uniformizer, and let $h \in H_{R \cup J}(U_J)_{\ffrm_J}$ be an element such that $\langle h, x \rangle_{U_J'} = \pi^{-1} \langle g, x \rangle_{U_J'}$ for all $x \in \text{image }i$. Then $g - \pi h$ is in the orthogonal complement of the image of $i$ (in other words, the new subspace of $H_{R \cup J}(U_J')_{\ffrm_J}$). It follows from the Deligne-Serre lemma \cite[Lemma 6.11]{Del74} that after possibly enlarging $E$ once more, we can find a $\bbT_{J \cup \{ v_0 \}}^{S, \text{univ}}$-eigenvector $g'$ in the new subspace with Hecke eigenvalues lifting those of $\overline{g}$. The form $g'$ generates an automorphic representation of $G_{R \cup J}(\bbA_F)$ whose Jacquet-Langlands image on $\GL_2(\bbA_F)$ has the desired properties. This establishes the existence of $\pi_0$.

We now no longer make any assumption on $q_{v_0}$ (except for our running assumption in this section that $q_{v_0} \not\equiv 1$ mod $p$). The existence of a $J \cup \{ v_0, v_1 \}$-good $\pi$ now follows from the argument of \cite[Theorem 5]{Raj01}, applied with $\frp = v_0$ and $\frq = v_1$. We note that in the context of \emph{loc. cit.}, the set $J$ is empty, but this plays no role in the proof. The assumption that $\overline{\rho}_\ffrm(\Frob_{v_1})$ is conjugate to the image of a complex conjugation is also superfluous, as it is used only to ensure that the level-raising congruence is satisfied at the place $v_1$.
\end{proof}

\begin{proposition}\label{prop_growth_of_multiplicities}
With notation as above, we have: 
\[ 1 \leq \dim_k (H_{R \cup Q}(U_Q) \otimes_\cO k)[\ffrm_Q] \leq 4^{\# Q} \dim_k (H_R(U_\emptyset) \otimes_\cO k)[\ffrm_\emptyset]. \]
\end{proposition}
\begin{proof}
The lower bound follows immediately from Lemma \ref{lem_level_raising_with_prescribed_eigenvalue}, applied with $J = Q$. We now establish the upper bound. By induction, it suffices to establish the following claim. Let $J \subset Q$ be a non-empty subset, $v \in J$, and $\overline{J} = J - \{ v \}$. Then we have:
\[ \dim_k (H_{R \cup J}(U_J) \otimes_\cO k)[\ffrm_J] \leq 4 \dim_k (H_{R \cup \overline{J}}(U_{\overline{J}}) \otimes_\cO k)[\ffrm_{\overline{J}}]. \]
We prove the claim, splitting into cases according to the parity of $\# \overline{J}$. We suppose first that $\# \overline{J}$ is odd. Then there exists, by Theorem \ref{thm_jarvis_exact_sequence}, an exact sequence
\begin{equation}\label{eqn_jarvis_sequence}
 0 \to (H_{R \cup J}(U_J)_{\ffrm_J} \otimes_\cO k)[\ffrm_J] \to (H_{R \cup \overline{J}}(U'_{\overline{J}})_{\ffrm_J} \otimes_\cO k)^{I_{F_v}}[\ffrm_J] \to (H_{R \cup \overline{J}}(U_{\overline{J}})_{\ffrm_J}^2 \otimes_\cO k)[\ffrm_J],
\end{equation}
where $U'_{\overline{J}} = \prod_w U'_{\overline{J}, w}$ satisfies $U'_{\overline{J}, w} = U_{\overline{J}, w}$ unless $w = v$, in which case $U'_{\overline{J}, v} = U_0(v)$. By Theorem \ref{thm_eicher_shimura_ring_case} and the Eichler--Shimura relation (\ref{eqn_eichler_shimura_relation}) (cf. Proposition \ref{prop_carayol_and_blr}), $(H_{R \cup \overline{J}}(U'_{\overline{J}})_{\ffrm_J} \otimes_\cO k)[\ffrm_J]$ is isomorphic to a direct sum of copies of $\overline{\rho}_\ffrm \otimes (\epsilon \det \overline{\rho}_\ffrm)^{-1}$ as $k[G_F]$-module. Since $\overline{\rho}_\ffrm|_{G_{F_v}}$ is unramified, we can rewrite (\ref{eqn_jarvis_sequence}) as follows, replacing the third term by a possibly larger $k[G_F]$-module:
\begin{equation}\label{eqn_jarvis_sequence_modified}
 0 \to (H_{R \cup J}(U_J)_{\ffrm_J} \otimes_\cO k)[\ffrm_J] \to (H_{R \cup \overline{J}}(U'_{\overline{J}})_{\ffrm_J} \otimes_\cO k)[\ffrm_J] \to (H_{R \cup \overline{J}}(U_{\overline{J}})_{\ffrm_{\overline{J}}}^2 \otimes_\cO k)[\ffrm_{\overline{J}}],
\end{equation}
where $\Frob_v^{-1}$ acts as the scalar $\alpha_v$ on the first term. Comparing the dimensions of the eigenspaces of $\Frob_v^{-1}$ in the terms of the exact sequence (\ref{eqn_jarvis_sequence_modified}), we obtain 
\[\begin{split} \dim_k (H_{R \cup J}(U_J)_{\ffrm_J} \otimes_\cO k)[\ffrm_J] & \leq \dim_k (H_{R \cup \overline{J}}(U'_{\overline{J}})_{\ffrm_J} \otimes_\cO k)[\ffrm_J]^{\Frob_v^{-1} = \alpha_v}\\ &  = \dim_k (H_{R \cup \overline{J}}(U'_{\overline{J}})_{\ffrm_J} \otimes_\cO k)[\ffrm_J]^{\Frob_v^{-1} = q_v^{-1} \alpha_v} \\ &\leq \dim_k (H_{R \cup \overline{J}}(U_{\overline{J}})_{\ffrm_{\overline{J}}}^2 \otimes_\cO k)[\ffrm_{\overline{J}}]^{\Frob_v^{-1} = q_v^{-1} \alpha_v}\\& = \frac{1}{2} \dim_k (H_{R \cup \overline{J}}(U_{\overline{J}})_{\ffrm_{\overline{J}}}^2 \otimes_\cO k)[\ffrm_{\overline{J}}] \\
& =  \dim_k (H_{R \cup \overline{J}}(U_{\overline{J}})_{\ffrm_{\overline{J}}} \otimes_\cO k)[\ffrm_{\overline{J}}],
\end{split}\]
which proves the claim in this case.

We now suppose that $\# \overline{J}$ is even. By Theorem \ref{thm_fujiwara_exact_sequence}, there is a map
\begin{equation}\label{eqn_fujiwara_map} ( H_{R \cup J}(U_J)_{\ffrm_J} \otimes_\cO k)^{I_{F_v}}[\ffrm_J] \to (H_{R \cup \overline{J}}(U_{\overline{J}})_{\ffrm_J} \otimes_\cO k)^2[\ffrm_J], 
\end{equation}
with kernel contained in the submodule
\[ (H_{R \cup J}(U_J)_{\ffrm_J}^{I_{F_v}} \otimes_\cO k)[\ffrm_J] \subset (H_{R \cup J}(U_J)_{\ffrm_J} \otimes_\cO k)^{I_{F_v}}[\ffrm_J], \]
on which $\Frob_v^{-1}$ acts by the scalar ${\mathbf{U}_v} \text{ mod }{\ffrm_J} = \alpha_v$. By assumption, $\overline{\rho}_\ffrm|_{G_{F_v}}$ is unramified. By Theorem \ref{thm_eicher_shimura_ring_case}, we thus have
\[  ( H_{R \cup J}(U_J)_{\ffrm_J} \otimes_\cO k)^{I_{F_v}}[{\ffrm_J}] = ( H_{R \cup J}(U_J)_{\ffrm_J} \otimes_\cO k)[{\ffrm_J}]. \]
We see that the $q_v^{-1} \alpha_v$-eigenspace of $\Frob_v^{-1}$ in $( H_{R \cup J}(U_J)_{\ffrm_J} \otimes_\cO k)[{\ffrm_J}]$ injects into $(H_{R \cup \overline{J}}(U_{\overline{J}})_{\ffrm_J} \otimes_\cO k)^2[\ffrm_J]$, and thus
\[ \begin{split}
\dim_k ( H_{R \cup J}(U_J)_{\ffrm_J} \otimes_\cO k)[\ffrm_J] & = 2 \dim_k ( H_{R \cup J}(U_J)_{\ffrm_J} \otimes_\cO k)[\ffrm_J]^{\Frob_v^{-1} = q_v^{-1} \alpha_v} \\
& \leq 2 \dim_k (H_{R \cup \overline{J}}(U_{\overline{J}})_{\ffrm_J} \otimes_\cO k)^2[\ffrm_J] \\
& \leq 2 \dim_k (H_{R \cup \overline{J}}(U_{\overline{J}})_{\ffrm_{\overline{J}}} \otimes_\cO k)^2[\ffrm_{\overline{J}}] \\
& = 4 \dim_k (H_{R \cup \overline{J}}(U_{\overline{J}})_{\ffrm_{\overline{J}}} \otimes_\cO k)[\ffrm_{\overline{J}}].
\end{split} \]
This establishes the claim in this case, and completes the proof.
\end{proof}
The method of proof of Lemma \ref{lem_level_raising_with_prescribed_eigenvalue} easily yields the following variant. We omit the proof.
\begin{lemma}\label{lem_level_raising_while_preserving_ordinarity}
Let $\sigma \subset S_p$, let $\iota : \barqp \to \bbC$ be an isomorphism, and let $\pi$ be a cuspidal automorphic representation of $\GL_2(\bbA_F)$ of weight 2 which satisfies the following conditions.
\begin{itemize}
\item The residual representation $\overline{r_\iota(\pi)}$ is irreducible.
\item If $v \in \sigma$, then $\pi_v$ is $\iota$-ordinary and $\pi_v^{U_0(v)} \neq 0$.
\item If $v \in S_p - \sigma$, then $\pi_v$ is not $\iota$-ordinary and $\pi_v$ is unramified.
\item If $v \in R$, then $\pi_v$ is an unramified twist of the Steinberg representation.
\item If $v = a$, then $\pi_a^{U_1^1(a)} \neq 0$.
\item If $v \not\in S_p \cup R \cup \{ a \}$ is a finite place of $F$, then $\pi_v$ is unramified. If $v \in Q$, then the eigenvalues $\alpha_v, \beta_v \in \barfp$ of $\overline{r_\iota(\pi)}(\Frob_v)$ satisfy $\beta_v/\alpha_v = q_v$.
\end{itemize}
Then there exists a cuspidal automorphic representation $\pi'$ of $\GL_2(\bbA_F)$ of weight 2 which satisfies the following conditions.
\begin{itemize}
\item There is an isomorphism $\overline{r_\iota(\pi)} \cong \overline{r_\iota(\pi')}$.
\item If $v \in \sigma$, then $\pi'_v$ is $\iota$-ordinary and $(\pi'_v)^{U_0(v)} \neq 0$.
\item If $v \in S_p - \sigma$, then $\pi'_v$ is not $\iota$-ordinary and $\pi'_v$ is unramified.
\item If $v \in R \cup Q$, then $\pi'_v \cong \St_2(\chi_v)$ for some unramified character $\chi_v : F_v^\times \to \bbC^\times$. If $v \in Q$, then $\iota^{-1}\chi_v(\varpi_v)$ is congruent to $\alpha_v$ modulo the maximal ideal of $\overline{\bbZ}_p$.
\item If $v = a$, then $(\pi'_a)^{U_1^1(a)} \neq 0$.
\item If $v \not\in S_p \cup R \cup Q \cup \{ a \}$ is a finite place of $F$, then $\pi'_v$ is unramified.
\end{itemize}
\end{lemma}
\subsection{Level-lowering mod $p^N$}
We now fix again a finite set $R$ of finite places of $F$ of even cardinality, disjoint from $S_p \cup \{ a \}$, and a good subgroup $U = \prod_w U_w  \in \cJ_R$.  If $Q$ is any finite set of finite places of $F$, disjoint from $S_p \cup \{ a \} \cup R$, and such that for each $v \in Q$, $U_v = \GL_2(\cO_{F_v})$, then we define a good open compact subgroup $U_Q = \prod_v U_{Q, v} \in \cJ_{R \cup Q}$ by the following conditions.
\begin{itemize}
\item If $v \not\in Q$, then $U_{Q, v} = U_v$.
\item If $v \in Q$, then $U_{Q, v}$ is the unique maximal compact subgroup of $G_{R \cup Q}(F_v)$. 
\end{itemize}
Let $S$ be a finite set of finite places of $F$, containing $S_p$, and such that for all $v \not\in S$, $U_v = \GL_2(\cO_{F_v})$. We suppose given a non-Eisenstein maximal ideal $\ffrm \subset \bbT^{S, \text{univ}}$ which is in the support of $H_R(U)$.
\begin{theorem}\label{thm_main_level_lowering} Fix an integer $N \geq 1$, and a lifting of $\overline{\rho}_\ffrm$ to a continuous representation $\rho : G_F \to \GL_2(\cO / \lambda^N)$. We assume that $\rho$ satisfies the following properties:
\begin{enumerate}
\item $\rho$ is unramified outside $S$.
\item There exists a set $Q$ as above, of even cardinality, and a homomorphism $f : \bbT_Q^{S \cup Q}(H_{R \cup Q}(U_Q)) \to \cO/\lambda^N$ satisfying:
\begin{enumerate}
\item For each $v \in Q$, $q_v \not\equiv 1 \text{ mod }p$.
\item For every finite place $v \not\in S \cup Q$ of $F$, we have $f(T_v) = \tr \rho(\Frob_v)$.
\item Let $I = \ker f$. Then $(H_{R \cup Q}(U_Q) \otimes_\cO \cO/\lambda^N)[I]$ contains a submodule isomorphic to $\cO/\lambda^N$.
\end{enumerate} 
\end{enumerate} 
Then there exists a homomorphism $f' : \bbT^{S \cup Q}(H_R(U)) \to \cO/\lambda^N$ such that for all $v \not\in S \cup Q$, we have $f(T_v) = \tr \rho(\Frob_v)$.
\end{theorem}

\begin{corollary}\label{cor_automorphy_by_successive_approximation}
Let $\rho : G_F \to \GL_2(\cO)$ be a continuous lifting of $\overline{\rho}_\ffrm$, unramified outside $S$. Suppose that for every integer $N \geq 1$, there exists a set $Q$ as above, of even cardinality, and a homomorphism $f : \bbT^{S \cup Q}_Q(H_{R \cup Q}(U)) \to \cO/\lambda^N$ satisfying:
\begin{enumerate}
\item For each $v \in Q$, $q_v \not\equiv 1 \text{ mod }p$.
\item For every finite place $v \not\in S \cup Q$ of $F$, we have $f(T_v) = \tr \rho(\Frob_v)$.
\item Let $I = \ker f$. Then $(H_{R \cup Q}(U_Q) \otimes \cO/\lambda^N)[I]$ contains a submodule isomorphic to $\cO/\lambda^N$.
\end{enumerate}
Then $\rho$ is automorphic: there exists a cuspidal automorphic representation $\pi$ of $\GL_2(\bbA_F)$ of weight 2, and an isomorphism $\rho \otimes_\cO \barqp \cong r_\iota(\pi)$.
\end{corollary}
\begin{proof}[Proof of Corollary \ref{cor_automorphy_by_successive_approximation}]
Take $N, Q$ satisfying these conditions. Then the natural map $\bbT^{S \cup Q}(H_R(U)) \to \bbT^{S}(H_R(U))_\ffrm$ is surjective. Indeed, there exists a lifting of $\overline{\rho}_\ffrm$ to a continuous representation
\[ \rho_\ffrm : G_{F, S} \to \GL_2(\bbT^{S}(H_R(U))_\ffrm) \]
with the property that $\tr \rho_\ffrm(\Frob_v) = T_v$ and $\det \rho_\ffrm(\Frob_v) = q_v S_v$ for all $v \not\in S$ (cf. Proposition \ref{prop_carayol_and_blr}). It follows from the Chebotarev density theorem and a result of Carayol (\cite[Th\'eor\`eme 2]{Car94}) that $\bbT^S(H_R(U))_\ffrm$ can be generated by the unramified Hecke operators $T_v$, $v \not\in S \cup Q$, which proves the claim.

Applying Theorem \ref{thm_main_level_lowering} to the map $f$, we obtain a homomorphism $f' : \bbT^{S \cup Q}(H_R(U)) \to \cO/\lambda^N$ with the property that $f'(T_v) = \tr \rho(\Frob_v) \text{ mod }\lambda^N$ for all $v\not \in S \cup Q$. It follows from the previous paragraph that this map factors
\[ \xymatrix@1{ \bbT^{S \cup Q}(H_R(U)) \ar[r]^-{a_N} & \bbT^{S}(H_R(U))_{\ffrm} \ar[r]^-{b_N} & \cO/\lambda^N}, \]
where $a_N$ is surjective. Moreover, the map $b_N$ satisfies the condition $b_N(T_v) = \tr \rho(\Frob_v)  \text{ mod }\lambda^N$ for all $v \not\in S$. Indeed, it satisfies this condition for all $v \not\in S \cup Q$, by construction. Applying \cite[Th\'eor\`eme 1]{Car94}, we see that $b_N \circ \rho_\ffrm$ and $\rho \text{ mod }\lambda^N$ are equivalent, which implies the relation
\[ \tr (b_N \circ \rho_\ffrm)(\Frob_v) = b_N(T_v) = \tr \rho(\Frob_v) \text{ mod }\lambda^N \]
for all $v \not\in S$. Letting $N$ go to infinity, we obtain a homomorphism $b_\infty : \bbT^{S}(H_R(U))_{\ffrm} \to \cO$ such that for all finite places $v \not\in S$ of $F$, $b_\infty(T_v) = \tr \rho(\Frob_v)$. It follows that there exists a cuspidal automorphic representation $\pi$ of $\GL_2(\bbA_F)$ of weight 2 such that $r_\iota(\pi)$ and $\rho$ have, by the Chebotarev density theorem, the same character; in particular, they become isomorphic after extending scalars to $\barqp$.
\end{proof}
The remainder of this section is devoted to the proof of Theorem \ref{thm_main_level_lowering}. By induction, it will suffice to prove the following result:
\begin{proposition}\label{prop_level_lowering_induction_step}
Fix an integer $N \geq 1$, and a lifting of $\overline{\rho}_\ffrm$ to a continuous representation $\rho : G_F \to \GL_2(\cO / \lambda^N)$. We assume that $\rho$ satisfies the following properties:
\begin{enumerate}
\item $\rho$ is unramified outside $S$.
\item There exists a set $Q$ as above and a homomorphism $f : \bbT_Q^{S \cup Q}(H_{R \cup Q}(U_Q)) \to \cO/\lambda^N$ satisfying:
\begin{enumerate}
\item For each $v \in Q$, $q_v \not\equiv 1 \text{ mod }p$.
\item For every finite place $v \not\in S \cup Q$ of $F$, we have $f(T_v) = \tr \rho(\Frob_v)$.
\item Let $I = \ker f$. Then $(H_{R \cup Q}(U_Q) \otimes \cO/\lambda^N)[I]$ contains an $\cO$-submodule isomorphic to $\cO/\lambda^N$.
\end{enumerate} 
\end{enumerate} 
Choose $v \in Q$, and let $\overline{Q} = Q - \{ v \}$. Then there exists a homomorphism $f' : \bbT_{\overline{Q}}^{S \cup Q}(H_{R \cup \overline{Q}}(U_{\overline{Q}})) \to \cO/\lambda^N$ such that for all $v \not\in Q \cup S$, $f'(T_v) = \tr \rho(\Frob_v)$. Moreover, writing $I' = \ker f'$,  $(H_{R \cup \overline{Q}}(U_{\overline{Q}}) \otimes \cO/\lambda^N)[I']$ contains an $\cO$-submodule isomorphic to $\cO/\lambda^N$.
\end{proposition}
\begin{proof}
We first assume that $\# Q$ is even. Let $\ffrm_Q \subset \bbT_Q^{S \cup Q, \text{univ}}$ denote the pullback of the maximal ideal $(I, \lambda) \subset  \bbT_Q^{S \cup Q}(H_{R \cup Q}(U_Q))$. Taking the exact sequence of Theorem \ref{thm_jarvis_exact_sequence} and passing to $I$-torsion, we obtain an exact sequence:
\[ 0 \to (H_{R \cup Q}(U_Q)_{\ffrm_Q} \otimes_\cO \cO/\lambda^N)[I] \to (H_{R \cup \overline{Q}}(U_{\overline{Q}}')_{\ffrm_Q} \otimes_\cO \cO/\lambda^N)^{I_{F_{v}}}[I] \to (H_{R \cup \overline{Q}}(U_{\overline{Q}})^2_{\ffrm_Q} \otimes_\cO \cO/\lambda^N)[I]. \]
By hypothesis, $(H_{R \cup Q}(U_Q) \otimes_\cO \cO/\lambda^N)[I]$ contains an $\cO$-submodule isomorphic to $\cO/\lambda^N$; hence the same is true of $(H_{R \cup \overline{Q}}(U_{\overline{Q}}')_{\ffrm_Q} \otimes_\cO \cO/\lambda^N)^{I_{F_{v}}}[I]$. By Theorem \ref{thm_eicher_shimura_ring_case}, there is an $\cO/\lambda^N$-module $U_0$ and an isomorphism 
\[ (H_{R \cup \overline{Q}}(U_{\overline{Q}}')_{\ffrm_Q} \otimes_\cO \cO/\lambda^N)[I] \cong U_0 \otimes_\cO \rho. \]
We see that $U_0$ contains an $\cO$-submodule isomorphic to $\cO/\lambda^N$. Since $\rho|_{G_{F_{v}}}$ is unramified, we obtain an inclusion
\[ \rho|_{G_{F_{v}}} \subset (H_{R \cup \overline{Q}}(U'_{\overline{Q}})_{\ffrm_Q} \otimes_\cO \cO/\lambda^N)^{I_{F_{v}}}[I] \]
of $\cO[G_{F_{v}}]$-modules. We now observe that $\overline{\rho}_\ffrm(\Frob_v)$ has 2 distinct eigenvalues, hence the same is true for $\rho$. On the other hand, $\Frob_{v}^{-1}$ acts as the scalar $f(\mathbf{U}_v)$ on $(H_{R \cup Q}(U_Q) \otimes_\cO \cO/\lambda^N)[I]$. It follows that the image of $ \rho|_{G_{F_{v}}}$ in $(H_{R \cup \overline{Q}}(U_{\overline{Q}})^2 \otimes_\cO \cO/\lambda^N)[I]$ contains an $\cO$-submodule isomorphic to $\cO/\lambda^N$. Writing 
\[ f' : \bbT^{S \cup Q}_{\overline{Q}}(H_{R \cup \overline{Q}}(U_{\overline{Q}})) \to \bbT^{S \cup Q}_{\overline{Q}}((H_{R \cup \overline{Q}}(U_{\overline{Q}})^2 \otimes_\cO\cO/\lambda^N)[I]) \cong \cO/\lambda^N \]
for the induced homomorphism proves the proposition in this case.

We now assume that $\# Q$ is odd. Let $\ffrm_Q \subset \bbT_Q^{S \cup Q, \text{univ}}$ denote the pullback of the maximal ideal $(I, \lambda) \subset \bbT_Q^{S \cup Q}(H_{R \cup Q}(U_Q))$. Taking the morphism of Theorem \ref{thm_fujiwara_exact_sequence}  and passing to $I$-torsion, we obtain a morphism of $\bbT_Q^{S \cup Q, \text{univ}}$-modules
\begin{equation}\label{eqn_level_lowering_fujiwara_map} (H_{R \cup Q}(U_Q) \otimes_\cO \cO/\lambda^N)_{\ffrm_Q}^{I_{F_{v}}}[I] \to (H_{R \cup \overline{Q}}(U_{\overline{Q}}) \otimes_\cO \cO/\lambda^N)_{\ffrm_Q}^2[I], 
\end{equation}
with kernel contained in the submodule
\begin{equation} (H_{R \cup Q}(U_Q)^{I_{F_{v}}} \otimes_\cO \cO/\lambda^N)_{\ffrm_Q}[I] \subset (H_{R \cup Q}(U_Q) \otimes_\cO \cO/\lambda^N)_{\ffrm_Q}^{I_{F_{v}}}[I], 
\end{equation}
on which $\Frob_v^{-1}$ acts by $f(\mathbf{U}_v) \in \cO/\lambda^N$. By hypothesis, $(H_{R \cup Q}(U_Q) \otimes_\cO \cO/\lambda^N)[I]$ contains an $\cO$-submodule isomorphic to $\cO/\lambda^N$. By Theorem \ref{thm_eicher_shimura_ring_case}, there is an $\cO/\lambda^N$-module $U_0$ and an isomorphism of $\bbT_Q^{S \cup Q}[G_{F}]$-modules:
\[(H_{R \cup Q}(U_Q) \otimes_\cO \cO/\lambda^N)[I] \cong U_0 \otimes_\cO \rho. \]
Since $\overline{\rho}_\ffrm(\Frob_v)$ has 2 distinct eigenvalues, the image of the map (\ref{eqn_level_lowering_fujiwara_map}) must contain an $\cO$-submodule isomorphic to $\cO/\lambda^N$. We now take
\[  f' : \bbT^{S \cup Q}_{\overline{Q}}(H_{R \cup \overline{Q}}(U_{\overline{Q}})) \to \bbT^{S \cup Q}_{\overline{Q}}((H_{R \cup\overline{Q}}(U_{\overline{Q}})^2 \otimes_\cO\cO/\lambda^N)_{\ffrm_Q}[I]) \cong \cO/\lambda^N \]
to complete the proof of the proposition.
\end{proof}

\section{Galois theory}\label{sec_galois_theory}

In this section, we recall some of the basics of Galois deformation theory, and make our study of some residually dihedral Galois representations. Since we want to allow deformations which have variable Hodge--Tate weights at some places above $p$ and fixed Hodge--Tate weights at others, our definition of a Galois deformation problem includes the data of an additional coefficient ring (denoted $\Lambda_v$ below) at places where ramification is allowed. We begin by making a study of ordinary Galois representations.

\subsection{Ordinary Galois representations}\label{sec_ordinary_representations}

Let $p$ be an odd prime, and let $K$ be a finite extension of $\bbQ_p$. Let $\rho : G_K \to \GL_2(\overline{\bbQ}_p)$ be a de Rham representation such that for each embedding $\tau : K \hookrightarrow \barqp$, we have $\mathrm{HT}_\tau(\rho) = \{ 0, 1 \}$. In this paper, we say that $\rho$ is ordinary if it is isomorphic to a representation of the form
\begin{equation}\label{eqn_definition_of_ordinary}
\rho \sim \left(\begin{array}{cc} \psi_1 & \ast \\ 0 & \psi_2 \epsilon^{-1} \end{array}\right), 
\end{equation}
where $\psi_1, \psi_2 : G_{K} \to \overline{\bbQ}_p^\times$ are continuous characters with restriction to inertia of finite order. Otherwise, we say that $\rho$ is non-ordinary. The condition of being ordinary depends only on $\mathrm{WD}(\rho)$:
\begin{lemma}\label{lem_ordinary_of_galois_and_weil_deligne}
Let $\rho : G_{K} \to \GL_2(\overline{\bbQ}_p)$ be a continuous de Rham representation such that for each embedding $\tau : K \hookrightarrow \barqp$, we have $\mathrm{HT}_\tau(\rho) = \{ 0, 1 \}$. Then one of the following is true:
\begin{itemize}
\item $\mathrm{WD}(\rho)^{\text{F-ss}}$ is irreducible. In this case, $\rho$ is non-ordinary.
\item $\mathrm{WD}(\rho)^{\text{F-ss}}$ is indecomposable. In this case, $\rho$ is ordinary.
\item $\mathrm{WD}(\rho)^{\text{F-ss}} = \chi_1 \oplus \chi_2$ is decomposable; $\chi_1, \chi_2 : W_{K} \to \overline{\bbQ}_p^\times$ are smooth characters. Let $\Frob_K \in W_K$ be a geometric Frobenius element, and assume that $\val_p(\chi_1(\Frob_K)) \leq \val_p(\chi_2(\Frob_K))$. Then 
\[ \val_p(\chi_1(\Frob_K)) +  \val_p(\chi_2(\Frob_K)) = [K_0 : \bbQ_p]\]
 and 
\[ \val_p(\chi_1(\Frob_K)) \geq 0, \]
and $\rho$ is ordinary if and only if equality holds. 
\end{itemize}
\end{lemma}
\begin{proof}
This follows from an easy calculation with filtered $(\phi, N, G_K)$--modules; cf. the proof of \cite[Theorem 2.4]{Tho13}.
\end{proof}
We now fix a totally real field $F$ and an isomorphism $\iota : \barqp \to \bbC$. We recall that if $\pi$ is a cuspidal automorphic representation of $\GL_2(\bbA_F)$ of weight 2 and $v \in S_p$, then we have defined in \S \ref{sec_hecke_operators} what it means for the local component $\pi_v$ to be $\iota$-ordinary.
\begin{lemma}\label{lem_ordinarity_and_satake_parameters}
 Let $\pi$ be a cuspidal automorphic representation of $\GL_2(\bbA_F)$ of weight 2, and let $v \in S_p$. Then exactly one of the following is true.
\begin{enumerate}
\item $\pi_v$ is supercuspidal. In this case $\pi_v$ is not $\iota$-ordinary.
\item There is a character $\chi : F_v^\times \to \overline{\bbQ}_p^\times$ of finite order and an isomorphism
\[ \pi_v \cong \St_2(\iota\chi). \]
\(In this case, $\pi_v$ is $\iota$-ordinary and we have an equivalence
\[ r_\iota(\pi)|_{G_{F_v}} \sim \left(\begin{array}{cc} \psi_1 & \ast \\ 0 & \psi_2 \epsilon^{-1} \end{array}\right), \]
with $\psi_1|_{I_{F_v}} = \psi_2|_{I_{F_v}} = (\chi \circ \Art_{F_v}^{-1})|_{I_{F_v}}$.\)
\item There are characters $\chi_1, \chi_2 : F_v^\times \to \overline{\bbQ}_p^\times$ with open kernel and an isomorphism
\[ \pi_v \cong i_B^{\GL_2} \iota\chi_1 \otimes \iota\chi_2. \]
Suppose that $\val_p \chi_1(\varpi_v) \leq \val_p \chi_2(\varpi_v)$. Then $-\val_p(q_v)/2 \leq \val_p \chi_1(\varpi_v)$, and $\pi_v$ is $\iota$-ordinary if and only if equality holds. \(If $\pi_v$ is $\iota$-ordinary, then  we have an equivalence
\[ r_\iota(\pi)|_{G_{F_v}} \sim \left(\begin{array}{cc} \psi_1 & \ast \\ 0 & \psi_2 \epsilon^{-1} \end{array}\right), \]
with $\psi_1|_{I_{F_v}} =  (\chi_1 \circ \Art_{F_v}^{-1})|_{I_{F_v}}$ and $\psi_2|_{I_{F_v}} =  (\chi_2 \circ \Art_{F_v}^{-1})|_{I_{F_v}}$.\)
\end{enumerate}
\end{lemma}
\begin{proof}
This follows Lemma \ref{lem_ordinary_of_galois_and_weil_deligne} and local-global compatibility; see again the proof of \cite[Theorem 2.4]{Tho13}.
\end{proof}

\begin{lemma}\label{lem_ordinarity_supercuspidality_and_base_change}
Let $\pi$ be a cuspidal automorphic representation of $\GL_2(\bbA_F)$ of weight 2, and let $v \in S_p$.
\begin{enumerate}
\item The representation $r_\iota(\pi)|_{G_{F_v}}$ is ordinary if and only if $\pi_v$ is $\iota$-ordinary.
\item Suppose that $\pi_v$ is supercuspidal. Let $K/F_v$ be a finite extension inside $\overline{F}_v$ such that $\rec^T_{F_v}(\pi_v)|_{W_K}$ is unramified. Then $r_\iota(\pi)|_{G_{K}}$ is crystalline and non-ordinary.
\end{enumerate}
\end{lemma}
\begin{proof}
The first part is an immediate consequence of Lemma \ref{lem_ordinary_of_galois_and_weil_deligne} and Lemma \ref{lem_ordinarity_and_satake_parameters}. The second part follows immediately from the observation that if $r : W_{F_v} \to \GL_2(\barqp)$ is an irreducible representation with open kernel, and $r|_{W_K}$ is unramified, then all of the eigenvalues of $r(\Frob_K)$ have the same $p$-adic valuation.
\end{proof}

\begin{lemma}\label{lem_convert_ordinary_to_supercuspidal}
Suppose that $[F : \bbQ]$ is even, and let $\pi$ be a cuspidal automorphic representation of $\GL_2(\bbA_F)$ of weight 2. Suppose that for each finite place $v \not \in S_p$ of $F$, either $\pi_v$ is unramified or $q_v \equiv 1 \text{ mod }p$ and $\pi_v$ is an unramified twist of the Steinberg representation, while for every $v \in S_p$, $\pi_v$ is $\iota$-ordinary and $\pi_v^{U_0(v)} \neq 0$. Suppose furthermore that $\overline{r_\iota(\pi)}$ is irreducible and $[F(\zeta_p) : F] \geq 4$. Let $\sigma \subset S_p$ be a \(possibly empty\) subset. Then there exists a cuspidal automorphic representation $\pi'$ of $\GL_2(\bbA_F)$ of weight 2, satisfying the following conditions:
\begin{itemize}
\item There is an isomorphism of residual representations $\overline{r_\iota(\pi')} \cong \overline{r_\iota(\pi)}$, and $\pi$ and $\pi'$ have the same central character.
\item If $v \in \sigma$, then $\pi'_{v}$ is $\iota$-ordinary. If $v \in S_p - \sigma$, then $\pi'_{v}$ is supercuspidal.
\item If $v \nmid p\infty$ is a place of $F$ and $\pi_{v}$ is unramified, then $\pi'_{v}$ is unramified. If $\pi_{v}$ is ramified, then $\pi'_{v}$ is a ramified principal series representation.
\end{itemize}
\end{lemma}
\begin{proof}
We omit the proof, which is an easy consequence of the theory of types; cf. \cite[Lemma 3.1.5]{Kis09}, \cite[Theorem 1.1]{Gee09a} and \cite[Lemma 4.4.1]{Clo08}. The assumptions that $\overline{r_\iota(\pi)}$ is irreducible and $[F(\zeta_p) : F] \geq 4$ are imposed so that one can apply \cite[Lemma 12.3]{Jar99} (existence of auxiliary primes where $\overline{r_\iota(\pi)}$  satisfies no level-raising congruences, cf. Lemma \ref{lem_auxiliary_place_a} below).
\end{proof}

\subsection{Galois deformation theory}\label{sec_deformation_theory}

We now establish notation that will remain in effect until the end of \S \ref{sec_galois_theory}. Let $p$ be an odd prime, and let $E$ be a coefficient field. We fix a continuous, absolutely irreducible representation $\overline{\rho} : G_F \to \GL_2(k)$ and a continuous character $\mu : G_F \to \cO^\times$ which lifts $\det \overline{\rho}$. We will assume that $k$ contains the eigenvalues of all elements in the image of $\overline{\rho}$. We also fix a finite set $S$ of finite places of $F$, containing the set $S_p$ of places dividing $p$, and the places at which $\overline{\rho}$ and $\mu$ are ramified. For each $v \in S$, we fix a ring $\Lambda_v \in \CNL_\cO$ and we define $\Lambda = \widehat{\otimes}_{v \in S} \Lambda_v$, the completed tensor product being over $\cO$. Then $\Lambda \in \CNL_\cO$.

Let $v \in S$. We write $\cD_v^\square : \CNL_{\Lambda_v} \to \mathrm{Sets}$ for the functor that associates to $R \in \mathrm{CNL}_{\Lambda_v}$ the set of all continuous homomorphisms $r : G_{F_v} \to \GL_2(R)$ such that $r \mod \ffrm_R = \overline{\rho}|_{G_{F_v}}$ and $\det r$ agrees with the composite $G_{F_v} \to \cO^\times \to R^\times$ given by $\mu|_{G_{F_v}}$ and the structural homomorphism $\cO \to R$. It is easy to see that $\cD_v^\square$ is represented by an object $R_v^{\square} \in \CNL_{\Lambda_v}$.
\begin{definition}
Let $v \in S$. A local deformation problem for $\overline{\rho}|_{G_{F_v}}$ is a subfunctor $\cD_v \subset \cD_v^\square$ satisfying the following conditions:
\begin{itemize}
\item $\cD_v$ is represented by a quotient $R_v$ of $R_v^\square$.
\item For all $R \in \CNL_{\Lambda_v}$, $a \in \ker(\GL_2(R) \to \GL_2(k))$ and $r \in \cD_v(R)$, we have $a r a^{-1} \in \cD_v(R)$.
\end{itemize}
\end{definition}
We will write $\rho_v^\square : G_{F_v} \to \GL_2(R_v^\square)$ for the universal lifting. If a quotient $R_v$ of $R_v^\square$ corresponding to a local deformation problem $\cD_v$ has been fixed, we will write $\rho_v : G_{F_v} \to \GL_2(R_v)$ for the universal lifting of type $\cD_v$.
\begin{definition}
A global deformation problem is a tuple
\[ \cS = (\overline{\rho}, \mu, S, \{ \Lambda_v \}_{v \in S}, \{ \cD_v \}_{v \in S}), \]
where:
\begin{itemize}
\item The objects $\overline{\rho} : G_F \to \GL_2(k)$, $\mu : G_F \to k^\times$, $S$ and $\{ \Lambda_v \}_{v \in S}$ are as at the beginning of this section.
\item For each $v \in S$, $\cD_v$ is a local deformation problem for $\overline{\rho}|_{G_{F_v}}$.
\end{itemize}
\end{definition}
\begin{definition}
Let $\cS = (\overline{\rho}, \mu, S, \{ \Lambda_v \}_{v \in S}, \{ \cD_v \}_{v \in S})$ be a global deformation problem. Let $R \in \CNL_{\Lambda}$, and let $\rho : G_F \to \GL_2(R)$ be a continuous lifting of $\overline{\rho}$. We say that $\rho$ is of type $\cS$ if it satisfies the following conditions:
\begin{itemize}
\item $\rho$ is unramified outside $S$.
\item $\det \rho = \mu$. More precisely, the homomorphism $\det \rho : G_F \to R^\times$ agrees with the composite $G_F \to \cO^\times \to R^\times$ induced by $\mu$ and the structural homomorphism $\cO \to R$.
\item For each $v \in S$, the restriction $\rho|_{G_{F_v}}$ lies in $\cD_v(R)$, where we give $R$ the natural $\Lambda_v$-algebra structure arising from the homomorphism $\Lambda_v \to \Lambda$.
\end{itemize}
We say that two liftings $\rho_1, \rho_2 : G_F \to \GL_2(R)$ are strictly equivalent if there exists a matrix $a \in \ker(\GL_2(R) \to \GL_2(k))$ such that $\rho_2 = a \rho_1 a^{-1}$. 
\end{definition}
It is easy to see that strict equivalence preserves the property of being of type $\cS$. We write $\cD^\square_\cS$ for the functor $\CNL_{\Lambda} \to \Sets$ which associates to $R \in \CNL_{\Lambda}$ the set of liftings $\rho : G_F \to \GL_2(R)$ which are of type $\cS$. We write $\cD_\cS$ for the functor $\CNL_{\Lambda} \to \Sets$ which associates to $R \in \CNL_{\Lambda}$ the set of strict equivalence classes of liftings of type $\cS$. 

\begin{definition} If $T \subset S$ and $R \in \mathrm{CNL}_\Lambda$, then we define a $T$-framed lifting of $\overline{\rho}$ to $R$ to be a tuple $(\rho, \{ \alpha_v \}_{v \in T})$, where $\rho : G_F \to \GL_2(R)$ is a lifting and for each $ v\in T$, $\alpha_v$ is an element of $\ker(\GL_2(R) \to \GL_2(k))$. Two $T$-framed liftings $(\rho_1, \{ \alpha_v \}_{v \in T})$ and $(\rho_2, \{ \beta_v \}_{v \in T})$ are said to be strictly equivalent if there is an element $a \in \ker(\GL_2(R) \to \GL_2(k))$ such that $\rho_2 = a \rho_1 a^{-1}$ and $\beta_v = a \alpha_v$ for each $v \in T$.
\end{definition}
We write $\cD^T_\cS$ for the functor $\CNL_{\Lambda} \to \Sets$ which associates to $R \in \CNL_{\Lambda}$ the set of strict equivalence classes of $T$-framed liftings $(\rho, \{ \alpha_v \}_{v \in T})$ to $R$ such that $\rho$ is of type $\cS$. 

\begin{theorem}
Let $\cS = (\overline{\rho}, \mu, S, \{ \Lambda_v \}_{v \in S}, \{ \cD_v \}_{v \in S})$ be a global deformation problem. Then the functors $\cD_\cS$, $\cD_\cS^\square$ and $\cD_\cS^T$ are represented by objects $R_\cS$, $R_\cS^\square$ and $R_\cS^T$, respectively, of $\CNL_{\Lambda}$. 
\end{theorem}
\begin{proof}
This is well-known; see \cite[Appendix 1]{Gou01} for a proof that $\cD_\cS$ is representable. The representability of the functors $\cD_\cS^\square$ and $\cD_\cS^T$ can be deduced easily from this.
\end{proof}
We will generally write $\rho_\cS : G_F \to \GL_2(R_\cS)$ for a choice of representative of the universal deformation $[\rho_\cS]$ of type $\cS$.

Let $\cS = (\overline{\rho}, \mu, S, \{ \Lambda_v \}_{v \in S}, \{ \cD_v \}_{v \in S})$ be a global deformation problem, and for each $v \in S$ let $R_v \in \CNL_{\Lambda_v}$ denote the representing object of $\cD_v$. We write $A_\cS^T = \widehat{\otimes}_{v \in T} R_v$ for the completed tensor product, taken over $\cO$, of the rings $R_v$. The ring $A^T_\cS$ has a canonical $\Lambda_T$-algebra structure, where $\Lambda_T = \widehat{\otimes}_{v \in T} \Lambda_v$; it is easy to see that $A^T_\cS$ represents the functor $\mathrm{CNL}_{\Lambda_T} \to \mathrm{Sets}$ which associates to a $\Lambda_T$-algebra $R$ the set of tuples $(\rho_v)_{v \in T}$, where for each $v \in T$, $\rho_v : G_{F_v} \to \GL_2(R)$ is a lifting of $\overline{\rho}|_{G_{F_v}}$ such that $\rho_v \in \cD_v(R)$ when we give $R$ the $\Lambda_v$-algebra structure arising from the homomorphism $\Lambda_v \to \Lambda_T \to R$.

The natural transformation $(\rho, \{ \alpha_v \}_{v \in T}) \mapsto ( \alpha_v^{-1} \rho|_{G_{F_v}} \alpha_v)_{v \in T}$ induces a canonical map $A^T_\cS \to R_\cS^T$, which is a homomorphism of $\Lambda_T$-algebras. We will generally use this construction only when $\Lambda_v = \cO$ for each $v \in S - T$, in which case there is a canonical isomorphism $\Lambda_T \cong \Lambda$.

\subsection{Galois cohomology}

Let $\cS = (\overline{\rho}, \mu, S, \{ \Lambda_v \}_{v \in S}, \{ \cD_v \}_{v \in S})$ be a global deformation problem. For each $v \in S$, let $R_v$ denote the representing object of $\cD_v$. There are canonical isomorphisms
\begin{equation}\label{eqn_defn_of_local_cocycle_space}
Z^1(F_v, \ad^0 \overline{\rho}) \cong \Hom_k(\ffrm_{R_v^\square}/(\ffrm_{R_v^\square}^2, \ffrm_{\Lambda_v}), k) \cong \Hom_{\CNL_{\Lambda_v}}(R^\square_v, k[\epsilon]/(\epsilon^2)),
\end{equation}
where we write $Z^1(F_v, \ad^0 \overline{\rho})$ for the space of continuous 1-coycles $\phi : G_{F_v} \to \ad^0 \overline{\rho}$. The isomorphism between the first and third terms associates to a cocycle $\phi \in Z^1(F_v, \ad^0 \overline{\rho})$ the classifying homomorphism of the lifting $(1 + \epsilon \phi)\overline{\rho}|_{G_{F_v}}$; the isomorphism between the third and second terms is given by restriction to $\ffrm_{R_v^\square}$. We write $\cL^1_v \subset Z^1(F_v, \ad^0 \overline{\rho})$ for the pre-image of the subspace $\Hom_k(\ffrm_{R_v}/(\ffrm_{R_v}^2, \ffrm_{\Lambda_v}), k)$ under the isomorphism (\ref{eqn_defn_of_local_cocycle_space}). It follows from the definitions that $\cL^1_v$ is also the pre-image of a subspace $\cL_v \subset H^1(F_v, \ad^0 \overline{\rho})$.

Now let $T \subset S$ be a non-empty subset, and suppose that $\Lambda_v = \cO$ for each $v \in S - T$. Then there is a canonical isomorphism $\Lambda_T \cong \Lambda$, and the map $A_\cS^T \to R_\cS^T$ is a morphism of $\Lambda$-algebras. In this case, we define (following \cite[\S 2]{Clo08}) a complex $C_{\cS, T}^\bullet(\ad^0 \overline{\rho})$ by the formula
\[ C_{\cS, T}^i(\ad^0 \overline{\rho}) = \left\{ \begin{array}{ll} C^0(F_S/F, \ad \overline{\rho}) & i = 0\\
C^1(F_S/F, \ad^0 \overline{\rho}) \oplus_{v \in T} C^0(F_v, \ad \overline{\rho}) & i = 1\\
C^2(F_S/F, \ad^0 \overline{\rho}) \oplus_{v \in S-T} C^1(F_v, \ad^0 \overline{\rho})/\cL^1_v & i =2\\
C^i(F_S/F, \ad^0 \overline{\rho}) \oplus_{v \in S} C^{i-1}(F_v, \ad^0 \overline{\rho}) & \text{otherwise. }
\end{array}\right. \]
(Here, for example, $C^\bullet(F_S/F, \ad^0 \overline{\rho})$ denotes the usual complex of continuous inhomogeneous cochains $G_{F, S}^\bullet \to \ad^0 \overline{\rho}$, which calculates the continuous group cohomology of the discrete $\bbZ[G_{F, S}]$-module $\ad^0 \overline{\rho}$; cf. \cite[Ch. 1, \S 2.2]{Ser94}.) The boundary map is given by the formula
\begin{gather*} C^i_{\cS, T}(\ad^0 \overline{\rho}) \to C^{i+1}_{\cS, T}(\ad^0 \overline{\rho}) \\ 
(\phi, (\psi_v)_v) \mapsto (\partial \phi, (\phi|_{G_{F_v}} - \partial \psi_v)_v). 
\end{gather*} 
There is a long exact sequence of cohomology groups
\begin{equation}\label{eqn_long_exact_sequence_in_taylor_cohomology}
\xymatrix@R-2pc{ 
0 \ar[r] & H^0_{\cS, T}(\ad^0 \overline{\rho}) \ar[r] & H^0(F_S/F, \ad \overline{\rho}) \ar[r] & \oplus_{v \in T} H^0(F_v, \ad \overline{\rho}) \\
\ar[r] &  H^1_{\cS, T}(\ad^0 \overline{\rho}) \ar[r] & H^1(F_S/F, \ad^0 \overline{\rho}) \ar[r] & \oplus_{v \in T} H^1(F_v, \ad^0 \overline{\rho}) \oplus_{v \in S-T} H^1(F_v, \ad^0 \overline{\rho})/\cL_v \\
 \ar[r] &  H^2_{\cS, T}(\ad^0 \overline{\rho}) \ar[r] & H^2(F_S/F, \ad^0 \overline{\rho}) \ar[r] & \oplus_{v \in S} H^2(F_v, \ad^0 \overline{\rho}) \\
\ar[r] &  \dots}
\end{equation}
and consequently an equation
\begin{equation}\label{eqn_equality_of_euler_characteristics}
\chi_{\cS, T}(\ad^0 \overline{\rho}) = \chi(F_S/F, \ad^0 \overline{\rho}) - \sum_{v \in S} \chi(F_v, \ad^0 \overline{\rho}) - \sum_{v \in S - T} (\dim_k \cL_v - h^0(F_v, \ad^0 \overline{\rho})) + 1 - \# T 
\end{equation}
relating the Euler characteristics of these complexes (which are all finite; see \cite[Ch. 1, Corollary 2.3]{Mil06} and \cite[Ch. 1, Corollary 4.15]{Mil06}). We also define a group that plays the role of the dual Selmer group in this setting. Since $p$ is odd, there is a perfect duality of Galois modules
\begin{gather}
\begin{aligned}\label{eqn_duality_of_adjoint_representation} \ad^0 \overline{\rho} \times \ad^0 \overline{\rho}(1) \to k(\epsilon) \\
(X, Y) \mapsto \tr X Y. \
\end{aligned}
\end{gather}
In particular, this induces for each finite place $v$ of $F$ a perfect duality between the groups $H^1(F_v, \ad^0 \overline{\rho})$ and $H^1(F_v, \ad^0 \overline{\rho}(1))$ (by Tate duality; see \cite[Ch. 1, Corollary 2.3]{Mil06} again). We write $\cL_v^\perp \subset H^1(F_v, \ad^0 \overline{\rho}(1))$ for the annihilator under this pairing of $\cL_v$, and we define
\begin{equation} H^1_{\cS, T}(\ad^0 \overline{\rho}(1)) = \ker \left( H^1(F_S/F, \ad^0 \overline{\rho}(1)) \to \prod_{v \in S - T} H^1(F_v, \ad^0 \overline{\rho}(1)) / \cL_v^\perp \right). 
\end{equation}
\begin{proposition}\label{prop_presenting_global_deformation_ring} Let $\cS = (\overline{\rho}, \mu, S, \{ \Lambda_v \}_{v \in S}, \{ \cD_v \}_{v \in S})$ be a global deformation problem, and let $T \subset S$ be a non-empty subset. Suppose that $\Lambda_v = \cO$ for each $v \in S - T$. 
\begin{enumerate} \item The ring $R^T_\cS$ is a quotient of a power series ring over $A_\cS^T$ in $r$ variables, where $r =  h^1_{\cS, T}(\ad^0 \overline{\rho})$.
\item If $v \in S$, let $\ell_v = \dim_k \cL_v$. There is an equality
\[ h^1_{\cS, T}(\ad^0 \overline{\rho}) = h^1_{\cS, T}(\ad^0 \overline{\rho}(1))+\sum_{v \in S - T} \left( \ell_v - h^0(F_v, \ad^0 \overline{\rho}) \right) - h^0(F, \ad^0 \overline{\rho}(1)) - \sum_{v | \infty} h^0(F_v, \ad^0 \overline{\rho}) - 1 + \# T. \]
\end{enumerate}
\end{proposition}
\begin{proof}
The global analogue of (\ref{eqn_defn_of_local_cocycle_space}) is the chain of isomorphisms
\begin{equation}
H^1_{\cS, T}(\ad^0 \overline{\rho}) \cong \Hom_k(\ffrm_{R_\cS^T}/(\ffrm_{R_\cS^T}^2, \ffrm_{A_{\cS}^T}), k) \cong \Hom_{\mathrm{CNL}_{\Lambda}}(R_{\cS}^T/(\ffrm_{A_{\cS}^T}), k[\epsilon]/(\epsilon^2)).
\end{equation}
We explain the isomorphism between the first and third terms. A $T$-framed lifting of $\overline{\rho}$ to $k[\epsilon]/(\epsilon^2)$ can be written in the form $((1+\epsilon \phi)\overline{\rho}, (1 + \epsilon \alpha_v)_{v \in T})$, with $\phi \in Z^1(F_S/F, \ad^0 \overline{\rho})$. The condition that it be of type $\cS$ is equivalent to the condition $\phi|_{G_{F_v}} \in \cL^1_v$ for $v \in S$. The condition that it give the trivial lifting at $v \in T$ is equivalent to the condition
\[ (1 - \epsilon \alpha_v) (1 + \epsilon \phi|_{G_{F_v}})\overline{\rho}|_{G_{F_v}} (1 + \epsilon \alpha_v) = \overline{\rho}|_{G_{F_v}}. \]
Two pairs $((1+\epsilon \phi)\overline{\rho}, (1 + \epsilon \alpha_v)_{v \in T})$ and $((1+\epsilon \phi')\overline{\rho}, (1 + \epsilon \beta_v)_{v \in T})$ give rise to strictly equivalent $T$-framed liftings if and only if there exists $b \in \ad^0 \overline{\rho}$ satisfying
\begin{gather*}
\phi'(\sigma) = \phi(\sigma) + (1 - \ad^0 \overline{\rho}(\sigma))b, \\
\beta_v = \alpha_v + b 
\end{gather*}
for each $\sigma \in G_F$, $v \in T$. This is exactly the equivalence relation imposed on cocycles in the definition of the group $H^1_{\cS, T}(\ad^0 \overline{\rho})$, and this proves the first part of the proposition. For the second part, we recall that $H^i(F_S/F, \ad^0 \overline{\rho}) = 0$ if $i \geq 3$ (by \cite[Ch. 1, Theorem 4.10]{Mil06}, and since $p$ is odd), while Tate's local and global Euler characteristic formulae (see \cite[Ch. 1, Theorem 2.8]{Mil06} and \cite[Ch. 1, Theorem 5.1]{Mil06}, respectively) give 
\begin{gather*} \sum_{v \in S}\chi(F_v, \ad^0 \overline{\rho}) = - 3 [F : \bbQ], \\
\chi(F_S/F, \ad^0 \overline{\rho}) = \sum_{v | \infty} h^0(F_v, \ad^0 \overline{\rho}) - 3 [ F : \bbQ] ,
\end{gather*}
hence 
\begin{equation}\label{eqn_taylor_euler_characteristic}
\chi_{\cS, T}(\ad^0 \overline{\rho}) = \sum_{v | \infty} h^0(F_v, \ad^0 \overline{\rho}) - \sum_{v \in S - T} (\ell_v - h^0(F_v, \ad^0 \overline{\rho})) + 1 - \# T
\end{equation}
(use (\ref{eqn_equality_of_euler_characteristics})).
We now observe that there are exact sequences
\begin{equation*}
\xymatrix@R-2pc{& & H^1(F_S/F, \ad^0 \overline{\rho}) \ar[r] & \oplus_{v \in T} H^1(F_v, \ad^0 \overline{\rho}) \oplus_{v \in S-T} H^1(F_v, \ad^0 \overline{\rho})/\cL_v \\
 \ar[r] &  H^2_{\cS, T}(\ad^0 \overline{\rho}) \ar[r] & H^2(F_S/F, \ad^0 \overline{\rho}) \ar[r] & \oplus_{v \in S} H^2(F_v, \ad^0 \overline{\rho})  \\
\ar[r] &  H^3_{\cS, T}(\ad^0 \overline{\rho}) \ar[r] & 0}
\end{equation*}
and
\begin{equation*}
\xymatrix@R-2pc{ & & H^1(F_S/F, \ad^0 \overline{\rho}) \ar[r] & \oplus_{v \in T} H^1(F_v, \ad^0 \overline{\rho}) \oplus_{v \in S - T} H^1(F_v, \ad^0 \overline{\rho})/\cL_v \\
\ar[r] & H^1_{\cS, T}(\ad^0 \overline{\rho}(1))^\vee \ar[r] & H^2(F_S/F, \ad \overline{\rho}) \ar[r] & \oplus_{v \in S} H^2(F_v, \ad^0 \overline{\rho}) \\
\ar[r] & H^0(F_S/F, \ad^0 \overline{\rho}(1))^\vee\ar[r] & 0.}
\end{equation*}
(The first sequence is part of (\ref{eqn_long_exact_sequence_in_taylor_cohomology}), while the second is part of the Poitou-Tate exact sequence (see \cite[Ch. 1, Theorem 4.10]{Mil06}).) Comparing these two exact sequences, we obtain
\begin{gather*}
 h^2_{\cS, T}(\ad^0 \overline{\rho}) = h^1_{\cS, T}(\ad^0 \overline{\rho}(1)),\\
 h^3_{\cS, T}(\ad^0 \overline{\rho}) = h^0(F_S/F, \ad^0 \overline{\rho}(1)),
\end{gather*}
and so (\ref{eqn_taylor_euler_characteristic}) gives
\begin{equation*}
h^1_{\cS, T}(\ad^0 \overline{\rho}) = h^1_{\cS, T}(\ad^0 \overline{\rho}(1)) - h^0(F_S/F, \ad^0 \overline{\rho}(1)) - \sum_{v | \infty} h^0(F_v, \ad^0 \overline{\rho}) + \sum_{v \in S - T} (\ell_v - h^0(F_v, \ad^0 \overline{\rho})) - 1 + \# T
\end{equation*}
(we have $h^0_{\cS, T}(\ad^0 \overline{\rho}) = 0$, since $T$ is assumed to be non-empty). Re-arranging this equation completes the proof.
\end{proof}
\begin{corollary}
Suppose further that $F$ is totally real, $\overline{\rho}$ is totally odd, $[F(\zeta_p) : F)] > 2$, and $R_v$ is formally smooth over $\cO$ of dimension 4 for each $v \in S - T$. Then $R_\cS^T$ is a quotient of a power series ring over $A_\cS^T$ in $h^1_{\cS, T}(\ad^0 \overline{\rho}(1)) - [F : \bbQ] - 1 + \# T$ variables.
\end{corollary}
\begin{proof}
If $[F(\zeta_p) : F] > 2$, then $h^0(F_S/F, \ad^0 \overline{\rho}(1)) = 0$. If $R_v$ is formally smooth over $\cO$ of dimension 4, then we have $\ell_v - h^0(F_v, \ad^0 \overline{\rho}) = \ell_v^1 - 3 = \dim R_v - 4 = 0$. If $F$ is totally real and $\overline{\rho}$ is totally odd, then $\sum_{v | \infty} h^0(F_v, \ad^0 \overline{\rho}) = [F : \bbQ]$. The result now follows immediately from Proposition \ref{prop_presenting_global_deformation_ring}. 
\end{proof}

\subsection{Local deformation problems}

We continue with the notation of \S \ref{sec_deformation_theory}, and now define some local deformation problems. The following lemma is often useful.
\begin{lemma}\label{lem_local_deformation_problem}
Let $R_v \in \mathrm{CNL}_{\Lambda_v}$ be a quotient of $R_v^\square$ satisfying the following conditions:
\begin{enumerate}
\item The ring $R_v$ is reduced, and not isomorphic to $k$.
\item Let $r : G_{F_v} \to \GL_2(R_v)$ denote the specialization of the universal lifting, and let $a \in \ker(\GL_2(R_v) \to \GL_2(k))$. Then the homomorphism $R_v^\square \to R_v$ associated to the representation $a r a^{-1}$ by universality factors through the canonical projection $R_v^\square \to R_v$.
\end{enumerate}
Then the subfunctor of $\cD_v^\square$ defined by $R_v$ is a local deformation problem.
\end{lemma}
\begin{proof}
The proof is essentially the same as the proof of \cite[Lemma 3.2]{Bar11}. 
\end{proof}

\subsubsection{Ordinary deformations}

Let $v \in S_p$, and suppose that $\overline{\rho}|_{G_{F_v}}$ is trivial. We assume that $E$ contains the image of all embeddings $F_v \hookrightarrow \barqp$. We write $\cO_{F_v}^\times(p) = \ker( \cO_{F_v}^\times \to k(v)^\times)$, the maximal pro-$p$ subgroup of $\cO_{F_v}^\times$, and set $\Lambda_v = \cO \llbracket \cO_{F_v}^\times(p) \rrbracket$. We write $\eta^\text{univ} : \cO_{F_v}^\times(p) \to \Lambda_v^\times$ for the universal character. We also write $I_{F_v}^\text{ab}(p)$ for the maximal pro-$p$ subgroup of the inertia subgroup of the Galois group of the maximal abelian extension of $F_v$; then $\Art_{F_v}$ restricts to an isomorphism $\cO_{F_v}^\times(p) \cong I_{F_v}^\text{ab}(p)$. We now define a deformation problem $\cD_v^\text{ord}$ in terms of its corresponding local lifting ring $R_v^\text{ord}$. The rings we consider were first defined by Geraghty \cite{Ger09} for liftings valued in $\GL_n$; we follow here the presentation of Allen \cite{All13} for liftings valued in $\GL_2$. 

We define $\cL$ as the closed subscheme of $\bbP^1_{R_v^\square}$ whose $R$-points, $R$ an $R_v^\square$-algebra, consist of a free $A$-direct summand $L \subset A^2$ of rank 1 on which $I_{F_v}^\text{ab}(p)$ acts by the character $\eta^\text{univ} \circ \Art_{F_v}^{-1}$. We define $R_v^\text{ord}$ to be the maximal reduced, $\cO$-torsion free  quotient of the image of the map $R_v^\square \to H^0(\cL, \cO_{\cL})$. 
\begin{proposition}
The ring $R_v^\text{ord}$ defines a local deformation problem. For each minimal prime $Q_v \subset \Lambda_v$, $R_v^\text{ord}/Q_v$ is geometrically irreducible of dimension $4 + 2[F_v : \bbQ_p]$, and its generic point is of characteristic 0. If $x : R_v^\square \to \overline{\bbQ}_p$ is a homomorphism, then $x$ factors through $R_v^\text{ord}$ if and only if $\rho_x = x \circ \rho_v^\square$ is $\GL_2(\overline{\bbZ}_p)$-conjugate to a representation
\[ \rho_x \sim \left(\begin{array}{cc} \psi_1 & \ast \\ 0 & \psi_2 \end{array}\right), \]
where $\psi_1|_{I_{F_v}^\text{ab}(p)} = x \circ  \eta^\text{univ} \circ \Art_{F_v}^{-1}$.
\end{proposition}
\begin{proof}
The fact that $R_v^\text{ord}$ is a local deformation problem follows easily from its definition and Lemma \ref{lem_local_deformation_problem}. The other assertions follow from \cite[Proposition 1.4.4]{All13} and \cite[Proposition 1.4.12]{All13}.
\end{proof}
We define $\cD_v^\text{ord}$ to be the local deformation problem represented by $R_v^\text{ord}$.
\subsubsection{Crystalline non-ordinary deformations}

Let $v \in S_p$, and suppose that $\overline{\rho}|_{G_{F_v}}$ is trivial. Let $\Lambda_v = \cO$.
\begin{proposition}
There is a reduced, $\cO$-torsion free quotient $R_v^\text{non-ord}$ of $R_v^\square$ satisfying the following property: for any coefficient field $L/E$ and any homomorphism $x : R_v^\square \to L$, $x$ factors through $R_v^\text{non-ord}$ if and only if $x \circ \rho^\square_v$ is crystalline of Hodge-Tate weights $\mathrm{HT}_\tau = \{0, 1\}$ and non-ordinary, in the sense of \S \ref{sec_ordinary_representations}. Moreover, if $R_v^\text{non-ord} \neq 0$ then $R_v^\text{non-ord}$ defines a local deformation problem and $R_v^\text{non-ord}$ is integral of dimension $4 + [F_v : \bbQ_p]$.
\end{proposition}
\begin{proof}
See \cite[Corollary 2.5.16]{Kis09} and \cite[Proposition 2.3]{Gee06}.
\end{proof}
In the case that $R_v^\text{non-ord} \neq 0$, we define $\cD_v^\text{non-ord}$ to be the local deformation problem represented by $R_v^\text{non-ord}$.
\subsubsection{Special deformations, case $q_v \equiv 1 \text{ mod }p$}

Let $v \in S - S_p$, and suppose that $q_v \equiv 1 \text{ mod }p$ and $\overline{\rho}|_{G_{F_v}}$ is trivial. Let $\Lambda_v = \cO$. 
\begin{proposition}
There is a reduced, $\cO$-torsion free quotient $R_v^\text{St}$ of $R_v^\square$ satisfying the following property: for any homomorphism $x : R_v^\square \to L$, $x$ factors through $R_v^\text{St}$ if and only if $x \circ \rho_v^\square$ is $\GL_2(\cO_L)$-conjugate to a representation of the form
\[ x \circ \rho_v^\square \sim \left( \begin{array}{cc} \chi & \ast \\ 0 & \chi \epsilon^{-1} \end{array} \right), \]
where $\chi : G_{F_v} \to L^\times$ is an unramified character. Moreover, $R_v^\text{St}$ defines a local deformation problem and $R_v^\text{St}$ is integral of dimension 4.
\end{proposition}
\begin{proof}
See \cite[Proposition 2.6.6]{Kis09}. 
\end{proof}
\subsubsection{Special deformations, case $q_v \equiv -1 \text{ mod } p$}

Let $v \in S - S_p$, and suppose that $q_v \equiv -1 \text{ mod }p$ and that $\overline{\rho}|_{G_{F_v}}$ is unramified, and that $\overline{\rho}(\Frob_v)$ takes two distinct eigenvalues $\alpha_v, \beta_v \in k$ such that $\alpha_v / \beta_v = -1$. Let $\Lambda_v = \cO$. We now define directly a subfunctor $\cD_v^{\St(\alpha_v)}$ of $\cD_v^\square$. Let $R \in \CNL_\cO$, and let $r : G_{F_v} \to \GL_2(R)$ be an element of $\cD_v^\square(R)$. Let $\phi_v \in G_{F_v}$ be a choice of geometric Frobenius element. By Hensel's lemma, $r(\phi_v)$ has characteristic polynomial $(X - A_v)(X - B_v)$, where the elements $A_v, B_v \in R^\times$ lift $\alpha_v, \beta_v$. We say that $r \in \cD_v^{\St(\alpha_v)}(R)$ if $B_v = q_v A_v$ and $I_{F_v}$ acts trivially on $(r(\phi_v) - B_v)R^2$, a direct summand $R$-submodule of $R^2$. This condition is independent of the choice of $\phi_v$. 
\begin{proposition}
The functor $\cD_v^{\St(\alpha_v)}$ is a local deformation problem. The representing object $R_v^{\St(\alpha_v)}$ is formally smooth over $\cO$ of dimension 4.
\end{proposition}
We omit the easy proof.

\subsection{The existence of auxiliary primes}\label{sec_auxiliary_primes}

We continue with the notation of \S \ref{sec_deformation_theory}, and now consider representations $\overline{\rho} : G_F \to \GL_2(k)$ of a particular special form. We assume that $F$ is totally real, and that $\overline{\rho}$ is totally odd, i.e.\ $\mu(c) = -1$ for every choice of complex conjugation $c \in G_F$. We write $\zeta_p \in \overline{F}$ for a $p^\text{th}$ root of unity, $K$ for the unique quadratic subfield of $F(\zeta_p)/F$, and $w \in G_F$ for a fixed choice of element with non-trivial image in $\Gal(K/F)$. We fix a choice $c \in G_F$ of complex conjugation. 

We also assume that the field $K$ is totally real and that $\overline{\rho}|_{G_{F(\zeta_p)}}$ is a direct sum of 2 distinct characters. (We recall that one of our running assumptions is that $\overline{\rho}$ itself is absolutely irreducible.) The assumption that $K$ is totally real is equivalent to the assumption that $4$ divides $[F(\zeta_p) : F]$. In particular, we see that $p \equiv 1 \text{ mod } 4$.

It then follows from Clifford theory that $\overline{\rho} \cong \Ind_{G_{K}}^{G_F} \overline{\chi}$, for some continuous character $\overline{\chi} : G_{K} \to k^\times$. Let $\overline{\gamma} = \overline{\chi} / \overline{\chi}^w$. By hypothesis, $\overline{\gamma}$ is non-trivial, even after restriction to $G_{F(\zeta_p)}$. We have the following elementary observation.
\begin{lemma} We have $\ad^0 \overline{\rho} \cong k(\delta_{K/F}) \oplus \Ind_{G_{K}}^{G_F} \overline{\gamma}$, where $\delta_{K/F} : \Gal(K/F) \to k^\times$ is the unique non-trivial character.
\end{lemma}
We write $M_0 = k(\delta_{K/F})$ and $M_1 = \Ind_{G_{K}}^{G_F} \overline{\gamma}$, so that $\ad^0 \overline{\rho} = M_0 \oplus M_1$. We can assume, after conjugating, that $\overline{\rho}$ has the following form:
\begin{gather*}
 \overline{\rho}(\sigma) = \left( \begin{array}{cc} \overline{\chi}(\sigma) & 0 \\ 0 & \overline{\chi}^w(\sigma) \end{array} \right) \text{ if } \sigma \in G_K; \\
  \overline{\rho}(w) = \left( \begin{array}{cc} 0 & \overline{\chi}(w^2) \\ 1 & 0 \end{array} \right). 
\end{gather*}
After possibly reversing the roles of $\overline{\chi}$ and $\overline{\chi}^w$, we can assume as well that $\overline{\chi}(c) = 1$. Having fixed this form for $\overline{\rho}$, we fix the following standard basis of $\ad^0 \overline{\rho}$:
\[ E = \left( \begin{array}{cc} 0 & 1 \\ 0 &  0 \end{array} \right), \, H =  \left( \begin{array}{cc} 1 & 0 \\ 0 &  -1 \end{array} \right), \, F = \left( \begin{array}{cc} 0 & 0 \\ 1 &  0 \end{array} \right). \]
Then $M_0$ is spanned by $H$ and $M_1$ is spanned by the vectors $E, F$. We observe that under the natural perfect pairing $\ad^0 \overline{\rho} \times \ad^0 \overline{\rho}(1) \to k(\epsilon)$, the spaces $M_0$ and $M_1(1)$ are mutual annihilators; similarly, the spaces $M_1$ and $M_0(1)$ are mutual annihilators.

\begin{lemma}\label{lem_ramakrishna_tangent_space_is_decomposed}
Let $v \nmid p$ be a finite place of $F$, and suppose that the local deformation problem $\cD_v = \cD_v^{\St(\alpha_v)}$ is defined.
\begin{enumerate} \item The subspace $\cL_v \subset H^1(F_v, \ad^0 \overline{\rho})$ respects the decomposition $\ad^0 \overline{\rho} = M_0 \oplus M_1$; that is, we have
\[ \cL_v = \left( \cL_v \cap H^1(F_v, M_0) \right) \oplus \left( \cL_v \cap H^1(F_v, M_1) \right) \subset H^1(F_v, M_0) \oplus H^1(F_v, M_1) = H^1(F_v, \ad^0 \overline{\rho}). \]
\item Similarly, the subspace $\cL_v^\perp \subset H^1(F_v, \ad^0 \overline{\rho}(1))$ respects the decomposition $\ad^0 \overline{\rho}(1) = M_0(1) \oplus M_1(1)$. 
\end{enumerate}
\end{lemma}
\begin{proof}
The second part is the dual of the first part. 
We prove the first part. Our assumption that the local deformation problem $\cD_v = \cD_v^{\St(\alpha_v)}$ is defined means that $q_v \equiv -1 \text{ mod } p$, $\overline{\rho}|_{G_{F_v}}$ is unramified, and $\overline{\rho}(\Frob_v)$ has distinct eigenvalues $\alpha_v$, $\beta_v$ with $\alpha_v / \beta_v = -1$. Since $K$ is assumed to be totally real, the assumption $q_v \equiv -1 \text{ mod }p$ implies that $v$ splits in $K$. We thus have $M_0 = k$, $M_1 = k(\epsilon) \oplus k(\epsilon)$ as $k[G_{F_v}]$-modules; note that $\epsilon = \epsilon^{-1}$ in this case. The subspace $\cL_v \subset H^1(F_v, \ad^0 \overline{\rho})$ is 1-dimensional, and lies in $H^1(F_v, M_1)$, being spanned by either $H^1(F_v, k(\epsilon)) \otimes E$ or $H^1(F_v, k(\epsilon)) \otimes F$. This implies the result.
\end{proof}
\begin{remark}\label{rmk_inducing_extension_is_totally_real}
This lemma is false without the assumption that $K$ is totally real. This is the main reason for making this assumption. 
\end{remark}
Let $\cS = (\overline{\rho}, \mu, S, \{ \Lambda_v \}_{v \in S}, \{ \cD_v \}_{v \in S})$ be a global deformation problem, and let $T \subset S$ be a subset containing all the places above $p$. Suppose that for each $v \in S \setminus T$, $\cD_v = \cD_v^{\St(\alpha_v)}$. The lemma then implies that we can decompose 
\begin{equation*} H^1_{\cS, T}(\ad^0 \overline{\rho}(1)) = \ker \left( H^1(F_S/F, \ad^0 \overline{\rho}(1)) \to \prod_{v \in S - T} H^1(F_v, \ad^0 \overline{\rho}(1)) / \cL_v^\perp \right) = H^1_{\cS, T}( M_0(1)) \oplus H^1_{\cS, T}(M_1(1)),
\end{equation*}
where by definition
\[ H^1_{\cS, T}(M_0(1)) = H^1_{\cS, T}(\ad^0 \overline{\rho}(1)) \cap H^1(F_S/F, M_0(1)) \]
and
\[ H^1_{\cS, T}(M_1(1)) = H^1_{\cS, T}(\ad^0 \overline{\rho}(1)) \cap H^1(F_S/F, M_1(1)). \]
We define $h^1_{\cS, T}(M_0(1))$ and $h^1_{\cS, T}(M_1(1))$ accordingly.
\begin{proposition}\label{prop_existence_of_auxiliary_ramakrishna_primes}
Let $\cS = (\overline{\rho}, \mu, S, \{ \Lambda_v \}_{v \in S}, \{ \cD_v \}_{v \in S})$ be a global deformation problem, and let $T = S$. Let $N_0 \geq 1$ be an integer. Let $\rho : G_F \to \GL_2(\cO)$ be a lifting of type $\cS$. Then for any integer $q \geq h^1_{\cS, T}(M_1(1))$, there exists a set $Q_0$ of primes, disjoint from $S$, and elements $\alpha_v \in k^\times$, satisfying the following conditions:
\begin{enumerate}
\item $\# Q_0 = q$.
\item For each $v \in Q_0$, the local deformation problem $\cD_v^{\St(\alpha_v)}$ is defined. We define the augmented deformation problem
\[ \cS_{Q_0} = (\overline{\rho}, \mu, S \cup Q_0, \{ \Lambda_v \}_{v \in S} \cup \{ \cO \}_{v \in Q_0}, \{ \cD_v \}_{v \in S} \cup \{ \cD_v^{\St(\alpha_v)} \}_{v \in Q_0}). \]
\item Let $\rho_{N_0} = \rho \text{ mod } \lambda^{N_0} : G_F \to \GL_2(\cO/\lambda^{N_0} \cO)$. Then $\rho_{N_0}(c) = \rho_{N_0}(\Frob_v)$ for each $v \in Q_0$, at least up to conjugacy in the image of $\rho_{N_0}$.
\item We have 
$H^1_{\cS_{Q_0}, T}(M_1(1)) = 0$.
\end{enumerate}
\end{proposition}
Before giving the proof of the proposition, we prove some lemmas.
\begin{lemma}\label{lem_modules_associated_to_characters}
Let $\Gamma$ be a group, and $\alpha : \Gamma \to k^\times$ a character. Let $k' \subset k$ be the subfield generated by the values of $\alpha$. Then $k'(\alpha)$ is a simple $\fp[\Gamma]$-module. If $\beta : \Gamma \to k^\times$ is another character, then $k'(\alpha)$ is isomorphic to a $\fp[\Gamma]$-submodule of $k(\beta)$ if and only if there is an automorphism $F$ of $k$ such that $\beta = F \circ \alpha$.
\end{lemma}
\begin{proof}
Since the image of $\alpha$ is cyclic, we can find $x \in \Gamma$ such that $k' = \bbF_p(\alpha(x))$. It follows that $k'(\alpha)$ is generated as a $\fp[\Gamma]$-module by any non-zero element. If there is a non-zero homomorphism $k'(\alpha) \to k(\beta)$ then the same argument shows that we get an embedding $F : k' \hookrightarrow k$ such that $\beta = F \circ \alpha$.
\end{proof}

\begin{lemma}\label{lem_characters_non_isomorphic}
With notation and assumptions as above, the $\fp[G_{K}]$-module $k(\epsilon \overline{\gamma})$ has no Jordan-H\"older factors in common with the modules $k$, $k(\overline{\gamma})$ or $k(\overline{\gamma}^{-1})$. The characters $\overline{\gamma}$ and $\epsilon \overline{\gamma}$ are non-trivial.
\end{lemma}
\begin{proof}
Since $K$ is real, the character $\epsilon \overline{\gamma}$ is totally even, while the characters $\overline{\gamma}$ and $\overline{\gamma}^{-1}$ are totally odd. This rules out a non-trivial morphism of $\fp[G_{K}]$-modules between $k(\epsilon \overline{\gamma})$ and $k(\overline{\gamma})$ or $k(\overline{\gamma}^{-1})$. The characters $\overline{\gamma}$ and $\epsilon \overline{\gamma}$ are non-trivial since we have assumed that $\overline{\gamma}$ remains non-trivial even after restriction to $G_{F(\zeta_p)}$.
\end{proof}

\begin{lemma}\label{lem_restriction_of_cocycles}
Let $\rho : G_F \to \GL_2(\cO)$ be a lifting of type $\cS$. Let $N \geq 1$ be an integer, and let $\rho_N = \rho \text{ mod }\lambda^N$. Let $K_N = F(\zeta_{p^N}, \rho_N)$, i.e.\ $K_N$ is the splitting field of the representation $\rho_N|_{G_{F(\zeta_{p^N})}}$. Then $H^1(\Gal(K_N/F), M_1(1)) = 0$. 
\end{lemma}
\begin{proof}
We have, by Shapiro's lemma and inflation-restriction:
\[ H^1(\Gal(K_N/F), M_1(1)) \cong H^1(\Gal(K_N/K), k(\epsilon\overline{\gamma})) \cong H^1(\Gal(K_N/K_1), k)^{\epsilon \overline{\gamma}}. \]
(The superscript in the last term indicates the subspace where the group $\Gal(K_1/K)$ acts by the character $\epsilon \overline{\gamma}$.) Let $k' \subset k$ denote the subfield generated by the values of $\epsilon \overline{\gamma}$. It suffices to show that 
\[ H^1(\Gal(K_N/K_1), k')^{\epsilon \overline{\gamma}} = 0. \]
We first show that $H^1(\Gal(K_1(\zeta_{p^N})/K_1), k')^{\epsilon \overline{\gamma}} = 0$. Indeed, an element of this group is represented by a homomorphism $f : \Gal(K_1(\zeta_{p^N})/K_1) \to k'$ such that for all $x \in G_{K}$, $y \in \Gal(K_1(\zeta_{p^N})/K_1)$, $f(x y x^{-1}) = \epsilon \overline{\gamma}(x) f(y)$. The conjugation action of $G_{K}$ on the group $\Gal(K_1(\zeta_{p^N})/K_1)$ is trivial, so any such homomorphism must be zero.

It therefore suffices to show that $H^1(\Gal(K_N/K_1(\zeta_{p^N})), k')^{\epsilon \overline{\gamma}} = 0$. An element of this group is again represented by a homomorphism $f : \Gal(K_N/K_1(\zeta_{p^N})) \to k'$ transforming according to the character $\epsilon \overline{\gamma}$ under conjugation by $G_K$. Let $f$ be such a homomorphism, and suppose that $f$ is non-zero; then $f$ factors through an abelian quotient $H_f$ of $\Gal(K_N/K_1(\zeta_{p^N}))$, invariant under conjugation by $G_{K}$, and $f$ induces an isomorphism $H_f \cong k'(\epsilon \overline{\gamma})$ of simple $\fp[G_{K}]$-modules. 

On the other hand, $\rho_N$ induces an injection $\Gal(K_N/K_1(\zeta_{p^N})) \hookrightarrow 1 + M_2(\lambda / \lambda^N )$, which is equivariant for the conjugation action of $G_{K}$. There exists a decreasing filtration $F_\bullet \subset 1 + M_2(\lambda/ \lambda^N)$ by $G_{K}$-invariant normal subgroups, such that each $F_i / F_{i+1}$ is an abelian group, isomorphic to one of $k$, $k(\overline{\gamma})$ or $k(\overline{\gamma}^{-1})$ as $\fp[G_{K}]$-modules. Pulling this filtration back via $\rho_N$ to the group $\Gal(K_N/K_1(\zeta_{p^N}))$, we find that $H_f$ is isomorphic to a submodule of one of these $\fp[G_{K}]$-modules. The desired vanishing statement will therefore follow from the claim that neither $k$, $k(\overline{\gamma})$ nor $k(\overline{\gamma}^{-1})$ contains a non-zero $\fp[G_{K}]$-submodule isomorphic to $k'(\epsilon \overline{\gamma})$. This follows immediately from Lemma \ref{lem_modules_associated_to_characters} and Lemma \ref{lem_characters_non_isomorphic}.
\end{proof}

We now come to the proof of Proposition \ref{prop_existence_of_auxiliary_ramakrishna_primes}.
\begin{proof}[Proof of Proposition \ref{prop_existence_of_auxiliary_ramakrishna_primes}]
We wish to find a set $Q_0$ of auxiliary primes such that $h^1_{\cS_{Q_0}, T}(M_1(1)) = 0$. Suppose $r = h^1_{\cS, T}(M_1(1)) \geq 0$. By induction, it is enough to find a single auxiliary place $v$, satisfying the conditions of the proposition, such that $h^1_{\cS_{\{ v \}}, T}(M_1(1)) = \max(r-1, 0)$. The case $r = 0$ is easy, so let us assume $r \geq 1$.

Let $\varphi$ be a cocycle representing a non-trivial element of $H^1_{\cS, T}(M_1(1))$. It will suffice for our purposes to find a place $v \not\in S$ satisfying the following conditions:
\begin{itemize}
\item We have $\rho_{N_0}(\Frob_v) = \rho_{N_0}(c)$, up to $G_F$-conjugacy.
\item We have $q_v \equiv - 1 \text{ mod } p^{N_0}$.
\item The image of $\varphi$ in $H^1(F_v, M_1(1))$ is non-trivial. 
\end{itemize}
Indeed, the first two conditions imply (at least, up to conjugacy) that $\overline{\rho}(\Frob_v) = \diag(1, -1)$ and the deformation problem $\cD_v^{\St(\alpha_v)}$ is defined, for $\alpha_v = 1$ or $\alpha_v = -1$. There is an exact sequence
\[ 0 \to H^1_{\cS_{\{ v \}}, T}(M_1(1)) \to H^1_{\cS, T}(M_1(1)) \to k, \]
the last map being given on cocycles by $\varphi \mapsto \langle E, \varphi(\Frob_v) \rangle$ (if $\alpha_v = 1$) or by $\varphi \mapsto \langle F, \varphi(\Frob_v) \rangle$ (if $\alpha_v = -1$). By choosing $\alpha_v$ appropriately, we can thus ensure that this sequence is also exact on the right.

By the Chebotarev density theorem, it even suffices to find an element $\sigma \in G_F$ satisfying the following conditions:
\begin{itemize}
\item $\rho_{N_0}(\sigma) = \rho_{N_0}(c)$.
\item $\epsilon(\sigma) \equiv -1 \text{ mod } p^{N_0}$.
\item $\varphi(\sigma) \neq 0$.
\end{itemize}
Choose any element $\sigma_0 \in G_F$ lifting the image of $c$ in $\Gal(K_{N_0}/F)$; then $\sigma_0$ satisfies the first two conditions. If $\sigma_0$ satisfies the third condition, then we can take $\sigma = \sigma_0$. Suppose instead that $\varphi(\sigma_0) = 0$. By Lemma \ref{lem_restriction_of_cocycles}, the image of $\varphi$ in $H^1(K_{N_0}, M_1(1))$ is non-zero; this restriction is represented by a non-zero $G_F$-equivariant homomorphism $f : K_{N_0} \to M_1(1)$. We can therefore find $\tau \in K_{N_0}$ such that $f(\tau) \neq 0$. The element $\sigma = \tau \sigma_0$ now has the desired properties. 
\end{proof}

\begin{proposition}\label{prop_existence_of_auxiliary_taylor_wiles_primes}
Let $\cS = (\overline{\rho}, \mu, S, \{ \Lambda_v \}_{v \in S}, \{ \cD_v \}_{v \in S})$ be a global deformation problem. Let $T \subset S$, and suppose that for each $v \in S - T,$ $\cD_v = \cD_v^{\St(\alpha_v)}$. Suppose further that $h^1_{\cS, T}(M_1(1)) = 0$, and let $N_1 \geq 1$ be an integer. Then there exists a finite set $Q_1$ of finite places of $F$, disjoint from $S$, satisfying the following conditions:
\begin{enumerate}
\item $\# Q_1 = h^1_{\cS, T}(M_0(1))$ and for each $v \in Q_1$, $q_v \equiv 1 \text{ mod } p^{N_1}$ and $\overline{\rho}(\Frob_v)$ has distinct eigenvalues.
\item Define the augmented deformation problem:
\[ \cS_{Q_1} = (\overline{\rho}, \mu, S \cup Q_1, \{ \Lambda_v \}_{v \in S} \cup \{ \cO \}_{v \in Q_1}, \{ \cD_v \}_{v \in S} \cup \{ \cD_v^\square \}_{v \in Q_1}). \]
Then $h^1_{\cS_{Q_1}, T}(\ad^0 \overline{\rho}(1)) = 0$.
\end{enumerate}
\end{proposition}
We first prove a lemma.
\begin{lemma}\label{lem_restriction_of_M0_cocycles}
Let $N \geq 1$. Then $H^1(\Gal(F(\zeta_{p^N})/F), M_0(1)) = 0$.
\end{lemma}
\begin{proof}
By restriction, this group is identified with the group of homomorphisms $f : \Gal(F(\zeta_{p^N})/F(\zeta_p)) \to k$ such that for all $x \in \Gal(F(\zeta_p)/F)$, $y \in \Gal(F(\zeta_{p^N})/F(\zeta_p))$, we have $f(x y x^{-1}) = \delta_{K/F} \epsilon (x) f(y)$. Since the character $\delta_{K/F} \epsilon$ is odd and the conjugation action on $\Gal(F(\zeta_{p^N})/F(\zeta_p))$ is trivial, any such homomorphism must be 0.
\end{proof}

\begin{proof}[Proof of Proposition \ref{prop_existence_of_auxiliary_taylor_wiles_primes}]
Let $r = h^1_{\cS, T}(M_0(1)) = h^1_{\cS, T}(\ad^0 \overline{\rho}(1))$; we may assume that $r \geq 1$. We observe that if $v \not\in S$ is any place satisfying the first point above, then $h^1_{\cS_{\{ v \}}, T}(M_1(1)) = h^1_{\cS, T}(M_1(1)) = 0$. By induction, it is therefore enough to find a place $v \not\in S$ satisfying the first point above, and such that $h^1_{\cS_{\{ v \}}, T}(M_0(1)) = r - 1$. 

Let $\varphi$ be a cocycle representing a non-trivial element of $h^1_{\cS, T}(M_0(1))$. It will suffice for our purposes to find a place $v \not\in S$ satisfying the following conditions:
\begin{itemize}
\item $q_v \equiv 1 \text{ mod }p^{N_1}$.
\item $\overline{\chi}(\Frob_v) \neq \overline{\chi}^w(\Frob_v)$. (This makes sense since the first  point implies that $v$ splits in $K$.)
\item $\varphi(\Frob_v) \neq 0$.
\end{itemize}
Indeed, the first two points show that the place $v$ satisfies the first point of Proposition \ref{prop_existence_of_auxiliary_taylor_wiles_primes}. We then have an exact sequence:
\[ 0 \to H^1_{\cS_{ \{ v \}}, T}(M_0(1)) \to H^1_{\cS, T}(M_0(1)) \to k, \]
the last map being given on cocycles by $\varphi \mapsto \varphi(\Frob_v)$. The third point implies that this sequence is exact on the right. By the Chebotarev density theorem, it will even suffice for our purposes to find an element $\sigma \in G_F$ satisfying the following conditions:
\begin{itemize}
\item $\epsilon(\sigma) \equiv 1 \text{ mod }p^{N_1}$.
\item $\overline{\chi}(\sigma) \neq \overline{\chi}^w(\sigma)$.
\item $\varphi(\sigma) \neq 0$.
\end{itemize}
If $N \geq 1$, write $F_N = F(\zeta_{p^N})$. We have $\overline{\chi}|_{G_{F_1}} \neq \overline{\chi}^w|_{G_{F_1}}$, by assumption. Since $\overline{\chi}$ has order prime to $p$, we also have $\overline{\chi}|_{G_{F_{N_1}}} \neq \overline{\chi}^w|_{G_{F_{N_1}}}$. Let $\sigma_0 \in G_{F_{N_1}}$ be an element such that $\overline{\chi}(\sigma_0) \neq \overline{\chi}^w(\sigma_0)$. If $\varphi(\sigma_0) \neq 0$, then we can take $\sigma = \sigma_0$.

 Assume instead that $ \varphi(\sigma_0) = 0$. We can then take $\sigma = \tau \sigma_0$, where $\tau \in G_{K_1(\zeta_{p^{N_1}})}$ is any element such that $\varphi(\tau) \neq 0$. To see that such an element exists, it is enough to note that the image of $\varphi$ in $H^1(K_1(\zeta_{p^{N_1}}), M_0(1))$ is non-zero, since Lemma \ref{lem_restriction_of_M0_cocycles} and the fact that $\Gal(K_1(\zeta_{p^{N_1}})/F_{N_1})$ has order prime to $p$ together imply that the restriction map is injective.
\end{proof}

\section{$R = \bbT$}\label{sec_r_equals_t}

Let $p$ be an odd prime, let $E$ be a coefficient field, and let $F$ be a totally real number field of even degree over $\bbQ$. We fix an absolutely irreducible and continuous representation $\overline{\rho} : G_F \to \GL_2(k)$, satisfying the following conditions.
\begin{itemize}
\item Let $K \subset F(\zeta_p)$ denote the unique quadratic subfield of $F(\zeta_p)/F$. Then there is a continuous character $\overline{\chi} : G_K \to k^\times$ and an isomorphism $\overline{\rho} \cong \Ind_{G_K}^{G_F} \overline{\chi}$. We write $w \in \Gal(K/F)$ for the non-trivial element.
\item The extension $K/F$ is totally real. In particular, $4$ divides $[F(\zeta_p) : F]$, which in turn divides $(p-1)$.
\item The character $\overline{\gamma} = \overline{\chi}/\overline{\chi}^w : G_K \to k^\times$ is non-trivial, even after restriction to $G_{F(\zeta_p)}$.
\item For each place $v \nmid p$ of $F$, $\overline{\rho}|_{G_{F_v}}$ is unramified.
\item For each place $v | p$ of $F$, $\overline{\rho}|_{G_{F_v}}$ is trivial. 
\item The character $\epsilon \det \overline{\rho}$ is everywhere unramified.
\end{itemize}
We assume $k$ large enough that it contains the eigenvalues of every element of the image of $\overline{\rho}$. We write $\psi : G_F \to \cO^\times$ for the Teichm\"uller lift of $\epsilon \det \overline{\rho}$. In what follows, we will abuse notation slightly by also writing $\psi$ for the character $\psi \circ \Art_F : \bbA_F^{\infty, \times} \to \cO^\times$. We will also suppose given the following data.
\begin{itemize}
\item A finite set $R$ of even cardinality of finite places of $F$, such that for each $v \in R$, $q_v \equiv 1 \text{ mod }p$ and $\overline{\rho}|_{G_{F_v}}$ is trivial. 
\item A finite set $Q_0$ of finite places of $F$ of even cardinality, disjoint from $S_p \cup R$, and a tuple $(\alpha_v)_{v \in Q_0}$ of elements of $k$, such that for each $v \in Q_0$ the deformation problem $\cD_v^{\St(\alpha_v)}$ is defined. (By definition, this means that for each $v \in Q_0$, we have $q_v \equiv -1 \text{ mod }p$ and $\overline{\rho}|_{G_{F_v}}$ is unramified, $\overline{\rho}(\Frob_v)$ having eigenvalues $\alpha_v, -\alpha_v$.)
\item An isomorphism $\iota : \barqp \to \bbC$, and a cuspidal automorphic representation $\pi_0$ of weight 2 satisfying the following conditions.
\begin{itemize} 
\item There is an isomorphism $\overline{r_\iota(\pi_0)} \cong \overline{\rho}$.
\item The central character of $\pi_0$ equals $\iota \psi$.
\item For each finite place $v \not \in S_p \cup R \cup Q_0$ of $F$, $\pi_{0,v}$ is unramified. 
\item For each $v \in R \cup Q_0$, there is an unramified character $\chi_v : F_v^\times \to \overline{\bbQ}_p^\times$ and an isomorphism $\pi_{0,v} \cong \St_2(\iota \chi_v)$. For each $v \in Q_0$, $\chi_v(\varpi_v)$ is congruent to $\alpha_v$ modulo the maximal ideal of $\overline{\bbZ}_p$.
\item Let $\sigma \subset S_p$ denote the set of places $v$ such that $\pi_{0, v}$ is $\iota$-ordinary. For each $v \in \sigma$, $\pi_{0, v}^{U_0(v)} \neq 0$. For each $v \in S_p - \sigma$, $\pi_{0, v}$ is unramified.
\end{itemize}
\end{itemize}
\begin{lemma}\label{lem_auxiliary_place_a}
There exists a place $a \not\in S_p \cup Q_0 \cup R$ such that $q_a > 4^{[F : \bbQ]}$ and $\tr \overline{\rho}(\Frob_a)^2 / \det \overline{\rho}(\Frob_a) \neq (1 + q_a)^2/q_a$.
\end{lemma} 
\begin{proof}
This follows from Chebotarev's theorem and \cite[Lemma 12.3]{Jar99}.
\end{proof}
We fix such a place $a$. 

We consider the global deformation problem 
\begin{multline*} \cS = (\overline{\rho}, \epsilon^{-1} \psi, S_p \cup Q_0 \cup R, \{ \Lambda_v \}_{v \in \sigma} \cup \{ \cO \}_{v \in S_p \cup Q_0 \cup R - \sigma}, \\ \{ \cD_v^{\text{ord}} \}_{v \in \sigma} \cup \{ \cD_v^{\text{non-ord}} \}_{v \in S_p - \sigma} \cup \{ \cD_v^{\St(\alpha_v)} \}_{v \in Q_0} \cup \{ \cD_v^{\St} \}_{v \in R}). \end{multline*}
We set $T = S_p \cup R$. Then $R_\cS$ and $R_\cS^T$ are defined. (We observe that $R_v^\text{non-ord} \neq 0$ for each $v \in S_p - \sigma$, since $r_\iota(\pi_0)|_{G_{F_v}}$ defines a $\barqp$-point of this ring.)

\subsection{Automorphic forms with fixed central character}

Let $B$ be a quaternion algebra ramified exactly at $Q_0 \cup R \cup \{ v | \infty \}$.  (This is the algebra we have denoted by $B_{Q_0 \cup R}$ in \S \ref{sec_shimura_curves_and_varieties}. In this section, the set $Q_0 \cup R$ will remain fixed, and we therefore drop it from the notation.) Let $\cO_B \subset B$ be a maximal order. The choice of $\cO_B$ determines a group $G$ over $\cO_F$, its functor of points being given by $G(R) = (\cO_B \otimes_{\cO_F} R)^\times$.
If $v \not\in Q_0 \cup R$ is a finite place of $F$, then we fix an isomorphism $\cO_B \otimes_{\cO_F} \cO_{F_v} \cong M_2(\cO_{F_v})$; this determines compatible isomorphisms $G(F_v) \cong \GL_2(F_v)$ and $G(\cO_{F_v}) \cong \GL_2(\cO_{F_v})$. We define an open compact subgroup $U = \prod_v U_v \subset G(\bbA_F^\infty)$ as follows:
\begin{itemize}
\item If $v \not\in Q_0 \cup R \cup \{ a \}$, then $U_v = G(\cO_{F_v})=\GL_2(\cO_{F_v}) $.
\item If $v \in Q_0 \cup R$, then $U_v$ is the unique maximal compact subgroup of $G(F_v)$.
\item If $v = a$, then $U_v = U^1_1(a)$. (For the definition of this open compact subgroup of $\GL_2(\cO_{F_v})$, we refer the reader to \S \ref{sec_hecke_operators}.)
\end{itemize}
For technical reasons, we must define spaces of automorphic forms with fixed central character. If $V = \prod_v V_v \subset U$ is any open compact subgroup and $A$ is an $\cO$-module, then we write $H_\psi(V,A)$ for the set of functions $f : G(F) \backslash G(\bbA_F^\infty) \to A$ satisfying the relation $f(zgv) = \psi(z)f(g)$ for all $z \in \bbA_F^{\infty, \times}$, $g \in G(\bbA_F^\infty)$, and $v \in V$. Choosing a double coset decomposition
\[ G(\bbA_F^\infty) = \sqcup_i G(F) t_i V \bbA_F^{\infty,\times}, \]
there is an isomorphism
\[\begin{split} H_\psi(V,A) & \cong  \oplus_i A(\psi^{-1})^{t_i^{-1} G(F) t_i \cap ( V \cdot \bbA_F^{\infty,\times} )},\\ f & \mapsto  (f(t_i))_i, 
\end{split}\]
where $A(\psi^{-1})$ is the $\cO[U \cdot \bbA_F^{\infty, \times}]$-module with underlying $\cO$-module $A$ on which $U$ acts trivially and $\bbA_F^{\infty, \times}$ acts by $\psi^{-1}$. Each of the groups $(t_i^{-1} G(F) t_i \cap ( V \cdot \bbA_F^{\infty,\times} ))/F^\times$ is finite of order prime to $p$ (use \cite[Lemma 1.1]{Tay06} and the fact that $G(F)$ contains no elements of order $p$, as $[F(\zeta_p) : F] \geq 4$). It follows that  the natural maps $H_\psi(V, \cO) \otimes_\cO A \to H_\psi(V, A)$ are isomorphisms. With the notation of \S \ref{sec_hecke_algebras_and_modules}, we see that $H_\psi(U, \cO)$ is a $\bbT^{S_p \cup Q_0 \cup R \cup \{ a \}, \text{univ}}_{Q_0 \cup R}$-submodule of $H_{Q_0 \cup R}(U)$.

For any open compact subgroup $V = \prod_v V_v \subset U$, we define a pairing
\begin{equation}\label{eqn_pairing_on_automorphic_forms_with_fixed_central_character}
\langle \cdot, \cdot \rangle_V : H_\psi(V, \cO) \times H_\psi(V, \cO) \to \cO
\end{equation}
by the formula
\[ \langle f_1, f_2 \rangle_V = \sum_{x} f_1(x) f_2(x) \psi(\det x)^{-1} c_V(x)^{-1}, \]
where the sum ranges over $x \in G(F) \backslash G(\bbA_F^\infty) / V \cdot \bbA_F^{\infty, \times}$ and $c_V(x) = \# ( x^{-1} G(F) x \cap (V \cdot \bbA_F^{\infty, \times} ) ) / F^\times$. As noted above, $c_V(x)$ is a $p$-adic unit, so this pairing is in fact perfect. Moreover, for any finite place $v$ of $F$, $f_1, f_2 \in H_\psi(V, \cO)$ and $g \in G(F_v)$, we have the relation
\[ \langle [ V_v g V_v] f_1, f_2 \rangle_V = \psi(\det g) \langle f_1, [V_v g^{-1} V_v] f_2 \rangle_V \]
(see \cite[p. 741]{Tay06}).
\begin{lemma}\label{lem_automorphic_forms_are_free}
Let $V_1 = \prod_v V_{1, v} \subset V_2 = \prod_v V_{2, v}$ be open compact subgroups of $U$, with $V_1$ normal in $V_2$ and $V_2 \cap \bbA_F^{\infty, \times} = V_1 \cap \bbA_F^{\infty, \times}$. Suppose that the quotient $V_2/V_1$ is abelian of $p$-power order. Then:
\begin{enumerate}
\item The trace map $\tr_{V_1/V_2} : H_\psi(V_1, \cO) \to H_\psi(V_2, \cO)$ factors through an isomorphism $H_\psi(V_1, \cO)_{V_2} \cong H_\psi(V_2, \cO)$.
\item $H_\psi(V_1, \cO)$ is a free $\cO[V_2/V_1]$-module.
\end{enumerate}
\end{lemma}
\begin{proof}
The proof is the same as the proof of \cite[Lemma 2.3]{Tay06}.
\end{proof}

\subsection{Partial Hida families}

We now introduce some more open compact subgroups of $U$. If $n \geq 1$ is an integer, then we define $U_0(\sigma^n) = \prod_v U_0(\sigma^n)_v$ and $U_1(\sigma^n) = \prod_v U_1(\sigma^n)_v$ as follows.
\begin{itemize}
\item If $v \in \sigma$, then $U_0(\sigma^n)_v = U_0(v^n)$ and $U_1(\sigma^n)_v = U_1(v^n)$. 
\item If $v \not\in \sigma$, then $U_0(\sigma^n)_v = U_1(\sigma^n)_v = U_v$.
\end{itemize}
If $n = 1$, then we omit it from the notation. If $v \in \sigma$, then we have defined $\Lambda_v = \cO \llbracket \cO_{F_v}^\times(p) \rrbracket$. We write $\Lambda = \widehat{\otimes}_{v \in \sigma} \Lambda_v$, the completed tensor product being over $\cO$. We emphasize that this algebra depends on the choice of $\sigma$, even though we do not include it in the notation. The universal deformation ring $R_\cS$ is a $\Lambda$-algebra, since (by definition) it classifies deformations to $\Lambda$-algebras. If $S$ is a finite set of finite places of $F$, then we write $\bbT^{\Lambda, S, \text{univ}}$ for the polynomial ring over $\Lambda$ in the infinitely many indeterminates $T_v, S_v$ ($v \not \in S$). If $Q \subset S$, then we write $\bbT_Q^{\Lambda, S, \text{univ}}$ for the polynomial ring over $\bbT^{\Lambda, S, \text{univ}}$ in the indeterminates $\mathbf{U}_v$, $v \in Q$. In fact, we now fix $S = \sigma \cup Q_0 \cup R \cup \{ a \}$. In particular, if $v \in S_p - \sigma$, then $T_v \in \bbT^{\Lambda, S, \text{univ}}$. If $M$ is a $\bbT^{\Lambda, S, \text{univ}}$-module (resp. $\bbT_Q^{\Lambda, S, \text{univ}}$-module), then we write $\bbT^{\Lambda, S}(M)$ (resp. $\bbT_Q^{\Lambda, S}(M)$) for the image of $\bbT^{\Lambda, S, \text{univ}}$ (resp. $\bbT_Q^{\Lambda, S, \text{univ}}$) in $\End_\Lambda(M)$.

We now define a structure of $\bbT_{Q_0}^{\Lambda, S, \text{univ}}$-module on the spaces $H_\psi(U_\ast(\sigma^n), A)$, $n \geq 1$ and $\ast \in \{ 0, 1 \}$. The elements $T_v, S_v, \mathbf{U}_v \in \bbT_{Q_0}^{\Lambda, S, \text{univ}}$ act by the Hecke operators of the same name. To define the $\Lambda$-module structure, it is enough to define an action of the group $\prod_{v \in \sigma} \cO_{F_v}^\times(p)$ on the $\cO$-module $H_\psi(U_\ast(\sigma^n), A)$. We define such an action by letting $\alpha \in \cO_{F_v}^\times(p)$ act by the double coset operator 
\[ \langle \alpha \rangle = [ U_\ast(\sigma^n)_v \left( \begin{array}{cc} \alpha & 0 \\ 0 & 1 \end{array}\right) U_\ast(\sigma^n)_v]. \]
The operators ${\mathbf{U}_v}$ and $\langle \alpha \rangle$ commute with all inclusions $H_\psi(U_\ast(\sigma^n), A) \subset H_\psi(U_\ast(\sigma^{n+1}),A)$, so these maps become maps of $\bbT_{Q_0}^{\Lambda, S, \text{univ}}$-modules. The operator ${\mathbf{U}_v}$ does depend on the choice of uniformizer $\varpi_v$, but this will not be important for us. We set $\mathbf{U}_\sigma = \prod_{v \in \sigma} {\mathbf{U}_v}$.

For each $n \geq 1$, the ordinary idempotent $e = \lim_{N \to \infty} \mathbf{U}_\sigma^{N!}$ acts on $H_\psi(U_1(\sigma^n), \cO)$, and hence on 
\[ H_\psi(U_1(\sigma^n), A) = H_\psi(U_1(\sigma^n), \cO) \otimes_\cO A\]
for any $\cO$-module $A$. We define $H_\psi^\text{ord}(U_1(\sigma^n), A) = e H_\psi(U_1(\sigma^n), A)$, and 
\[ H^\text{ord}_\psi(U_1(\sigma^\infty)) = \ilim_n H^\text{ord}_\psi(U_1(\sigma^n), E/\cO). \]
Then the algebra
\[ \bbT^{\Lambda, S}_{Q_0}(H^\text{ord}_\psi(U_1(\sigma^\infty))) = \plim_n \bbT^{\Lambda, S}_{Q_0}(H^\text{ord}_\psi(U_1(\sigma^n), E/\cO)) \]
is reduced (since it is an inverse limit of reduced algebras).
\begin{lemma}\label{lem_ordinary_control}
For any $\cO$-module $A$ and any integers $n \geq m \geq 1$, the natural inclusions
\begin{gather}
H_\psi^\text{ord}(U_0(\sigma), A) \to H_\psi^\text{ord}(U_0(\sigma^n), A), \\
H_\psi^\text{ord}(U_1(\sigma^m), A) \to H_\psi^\text{ord}(U_1(\sigma^m) \cap U_0(\sigma^n), A)
\end{gather}
are isomorphisms.
\end{lemma}
\begin{proof}
This follows from the same calculation with Hecke operators that appears in \cite[Lemma 2.5.2]{Ger09}.
\end{proof}
If $v \in \sigma$ and $n \geq 1$, we define $\Lambda_{v, n} = \cO[(1 + \varpi_v \cO_{F_v})/(1 + \varpi_v^n \cO_{F_v})]$. (Thus $\Lambda_1 = \cO$.) We define $\Lambda_n = \otimes_{v \in \sigma} \Lambda_{v, n}$. There is a canonical surjection $\Lambda \to \Lambda_n$, and we write $\fra_n$ for the kernel of this surjection. By Lemma \ref{lem_ordinary_control}, for any $n \geq m \geq 1$ we have 
\[ H_\psi^\text{ord}(U_1(\sigma^n), A)[\fra_m] = H_\psi^\text{ord}(U_0(\sigma^n) \cap U_1(\sigma^m), A) = H_\psi^\text{ord}(U_1(\sigma^m), A). \]
\begin{proposition}\label{prop_ordinary_control}
\begin{enumerate}
\item For each $n \geq 1$, there is a canonical isomorphism
\[ H^\text{ord}_\psi(U_1(\sigma^\infty))^\vee / \fra_n H^\text{ord}_\psi(U_1(\sigma^\infty))^\vee \cong \Hom_\cO ( H^\text{ord}_\psi(U_1(\sigma^n), \cO), \cO ). \]
(We recall that $(\cdot)^\vee$ denotes $E$-Pontryagin dual; see \S \ref{sec_normalizations}.)
\item The space $H^\text{ord}_\psi(U_1(\sigma^\infty))^\vee$ is a free $\Lambda$-module of rank $\dim_k H_\psi^\text{ord}(U_1(\sigma), k)$.
\item The algebra $\bbT^{\Lambda, S}_{Q_0}(H^\text{ord}_\psi(U_1(\sigma^\infty)))$ is a finite faithful $\Lambda$-algebra. 
\end{enumerate}
\end{proposition}
\begin{proof}
We have isomorphisms
\[
\begin{split}
 H^\text{ord}_\psi(U_1(\sigma^\infty))^\vee / \fra_n H^\text{ord}_\psi(U_1(\sigma^\infty))^\vee & \cong H^\text{ord}_\psi(U_1(\sigma^\infty))[\fra_n]^\vee \\ & = H^\text{ord}_\psi(U_1(\sigma^n), E/\cO)^\vee \\ & \cong \Hom_\cO(H^\text{ord}_\psi(U_1(\sigma^n), \cO),\cO).  
\end{split} 
\]
by Lemma \ref{lem_ordinary_control}. By Lemma \ref{lem_automorphic_forms_are_free}, the space $H^\text{ord}_\psi(U_1(\sigma^n), \cO)$ is a free $\cO[(U_0(\sigma^n) \cap U_1(\sigma)) / U_1(\sigma^n)]$-module. It follows that for all $n \geq 1$ the module
\[ H^\text{ord}_\psi(U_1(\sigma^\infty))^\vee / \fra_n H^\text{ord}_\psi(U_1(\sigma^\infty))^\vee \]
is free over $\Lambda / \fra_n \Lambda = \Lambda_n$, hence $H^\text{ord}_\psi(U_1(\sigma^\infty))^\vee$ is free over $\Lambda$ of rank $\dim_k H^\text{ord}_\psi(U_1(\sigma), k)$. Combining these facts completes the proof of the proposition.
\end{proof}
Associated to the automorphic representation $\pi_0$ is a homomorphism $\bbT^{\Lambda, S}_{Q_0}(H^\text{ord}_\psi(U_1(\sigma), E/\cO)) \to \barfp$, which in fact takes values in $k$ (since $\overline{r_\iota(\pi_0)} \cong \overline{\rho} \otimes_k \barfp$ can be defined over $k$, and the elements $\alpha_v$ ($v \in Q_0$) lie in $k$ by construction). We write $\ffrm \subset \bbT^{\Lambda, S}_{Q_0}(H^\text{ord}_\psi(U_1(\sigma^\infty)))$ for the kernel of the composite 
\[ \bbT^{\Lambda, S}_{Q_0}(H^\text{ord}_\psi(U_1(\sigma^\infty))) \to \bbT^{\Lambda, S}_{Q_0}(H^\text{ord}_\psi(U_1(\sigma), E/\cO)) \to k. \]
\begin{proposition}
There is a lifting of $\overline{\rho}$ to a continuous representation 
\[ \rho_\ffrm : G_{F} \to \GL_2(\bbT^{\Lambda, S}_{Q_0}(H^\text{ord}_\psi(U_1(\sigma^\infty)))_\ffrm)\]
 such that for all finite places $v \not\in S_p \cup Q_0 \cup R \cup \{ a \}$ of $F$, $\rho_\ffrm$ is unramified and $\rho_\ffrm(\Frob_v)$ has characteristic polynomial $X^2 - T_v X + q_v S_v$. \(This lifting is then unique, up to strict equivalence.\) Moreover, $\rho_\ffrm$ is of type $\cS$.
\end{proposition}
\begin{proof}
It suffices to construct, for each $n \geq 1$, a homomorphism $R_\cS \to \bbT^{\Lambda, S}_{Q_0}(H^\text{ord}_\psi(U_1(\sigma^n), E/\cO))_\ffrm$ satisfying the relation $\tr \rho_\cS(\Frob_v) \mapsto T_v$ and $\det \rho_\cS(\Frob_v) \mapsto q_v S_v$ for all $v \not\in Q_0 \cup R \cup S_p \cup \{ a \}$. Indeed, these conditions characterize such a homomorphism uniquely, because the ring $R_\cS$ is topologically generated as a $\Lambda$-algebra by the traces of Frobenius elements, by \cite{Car94}. We can then pass to the limit to obtain the desired homomorphism $R_\cS \to \bbT^{\Lambda, S}_{Q_0}(H^\text{ord}_\psi(U_1(\sigma^\infty)))_\ffrm$.

For each $n \geq 1$, there is an isomorphism of algebras:
\[ \bbT^{\Lambda, S}_{Q_0}(H^\text{ord}_\psi(U_1(\sigma^n), E/\cO))_\ffrm \otimes_\cO \barqp \cong \oplus_{\pi} \barqp \]
and a corresponding isomorphism of spaces of automorphic forms:
\[ H_\psi^\text{ord}(U_1(\sigma^n), \cO)_\ffrm \otimes_\cO \barqp \cong \oplus_{\pi} (\iota^{-1} \pi^\infty)^{U_1(\sigma^n), \text{ord}}. \]
(Here, we write $(\iota^{-1} \pi^\infty)^{U, \text{ord}} \subset (\iota^{-1} \pi^\infty)^{U}$ for the largest subspace where $\mathbf{U}_\sigma$ operates with eigenvalues that are $p$-adic units.) In each case, the direct sum runs over the set of cuspidal automorphic representations $\pi$ of $\GL_2(\bbA_F)$ of weight 2 which satisfy the following conditions:
\begin{itemize}
\item $(\iota^{-1} \pi^\infty)^{U_1(\sigma^n), \text{ord}} \neq 0$. In particular, for each $v \in \sigma$, $\pi_v$ is $\iota$-ordinary.
\item There is an isomorphism $\overline{r_\iota(\pi)} \cong \overline{\rho}$.
\item The central character of $\pi$ is $\iota\psi$.
\item For each place $v \in S_p - \sigma$, $\pi_v$ is not $\iota$-ordinary.
\item For each finite place $v \not\in \sigma \cup Q_0 \cup R$ of $F$, $\pi_v$ is unramified.
\item For each place $v \in Q_0 \cup R$, $\pi_v$ is an unramified twist of the Steinberg representation. For each place $v \in Q_0$, the eigenvalue of ${\mathbf{U}_v}$ on $\iota^{-1} \pi_v^{U_0(v)}$ is congruent to $\alpha_v$, modulo the maximal ideal of $\overline{\bbZ}_p$.
\end{itemize}
There is a corresponding homomorphism $R_\cS \to \oplus_\pi \barqp$, with image contained in $\bbT^{\Lambda, S}_{Q_0}(H^\text{ord}_\psi(U_1(\sigma^n), E/\cO))_\ffrm$. This gives the desired map $R_\cS \to \bbT^{\Lambda, S}_{Q_0}(H^\text{ord}_\psi(U_1(\sigma^n), E/\cO))_\ffrm$.
\end{proof}
It follows that $H^\text{ord}_\psi(U_1(\sigma^\infty))_\ffrm$ has a structure of $R_\cS$-module. We can now state the main theorem of this section.
\begin{theorem}\label{thm_r_equals_t}
Suppose that $h^1_{\cS, T}(M_1(1)) = 0$ \(notation as in \S \ref{sec_auxiliary_primes}\). Then $\Fitt_{R_\cS} H^\text{ord}_\psi(U_1(\sigma^\infty))^\vee_\ffrm = 0$.
\end{theorem}
The proof of Theorem \ref{thm_r_equals_t} will be given in \S \ref{sec_patching} below.
\begin{corollary}\label{cor_r_equals_t}
Let $C, N, n \geq 1$ be integers. Suppose that $\dim_k H_\psi^\text{ord}(U_1(\sigma^n), k)[\ffrm] \leq C$, and suppose that there is a diagram
\[ \xymatrix{ \Lambda \ar[r] \ar[d] & \Lambda_n \ar[d] \\
R_\cS \ar[r] & \cO/\lambda^N, }\]
corresponding to a lifting $\rho_N : G_F \to \GL_2(\cO/\lambda^N)$ of $\overline{\rho}$ which is of type $\cS$. Let  $I = \ker ( R_\cS \to \cO/\lambda^{\lfloor N/C \rfloor} )$. Then $(H_\psi(U_1(\sigma^n), \cO)_\ffrm \otimes_\cO \cO/\lambda^{\lfloor N/C \rfloor})[I]$ contains an $\cO$-submodule isomorphic to $\cO/\lambda^{\lfloor N/C \rfloor}$, and the map $R_\cS \to \cO/\lambda^{\lfloor N/C \rfloor}$ factors 
\[ R_\cS \twoheadrightarrow \bbT^{\Lambda, S \cup S_p}_{Q_0}(H^\text{ord}_\psi(U_1(\sigma^n), E/\cO)_\ffrm) \to \cO/\lambda^{\lfloor N/C \rfloor}. \]
\end{corollary}
\begin{proof}[Proof of Corollary \ref{cor_r_equals_t}]
Let $J = \ker (R_\cS \to \cO/\lambda^N)$, so that $I = (J, \lambda^{\lfloor N/C \rfloor})$. Since formation of the Fitting ideal commutes with base extension, Theorem \ref{thm_r_equals_t} implies that $\Fitt_{\cO/\lambda^N} H_\psi^\text{ord}(U_1(\sigma^\infty))^\vee_\ffrm \otimes_{R_\cS} \cO/\lambda^N = 0$. We have isomorphisms
\[ \begin{split}
 H_\psi^\text{ord}(U_1(\sigma^\infty))^\vee_\ffrm \otimes_{R_\cS} \cO/\lambda^N & = H_\psi^\text{ord}(U_1(\sigma^\infty))^\vee_\ffrm / J H_\psi^\text{ord}(U_1(\sigma^\infty))^\vee_\ffrm \\
 & \cong H_\psi^\text{ord}(U_1(\sigma^\infty))_\ffrm[J]^\vee \\
 & = H_\psi^\text{ord}(U_1(\sigma^n), \cO/\lambda^N)_\ffrm[J]^\vee.
 \end{split} \]
 Similarly, there is an isomorphism
 \[ H_\psi^\text{ord}(U_1(\sigma^\infty))^\vee_\ffrm \otimes_{R_\cS} k \cong H_\psi^\text{ord}(U_1(\sigma^n), k)_\ffrm[\ffrm]^\vee. \]
It follows that $H_\psi^\text{ord}(U_1(\sigma^n), \cO/\lambda^N)_\ffrm[J]^\vee$ is generated as $\cO/\lambda^N$-module by at most $C$ elements. Since its Fitting ideal as $\cO/\lambda^N$-module is trivial, it contains a copy of $\cO/\lambda^{\lfloor N / C \rfloor}$. The desired result follows on taking Pontryagin duals. 
\end{proof}

\subsection{Patching and the proof of Theorem \ref{thm_r_equals_t}}\label{sec_patching}

We continue with the notation of the previous section. Suppose given a finite set $Q_1$ of finite places of $F$, disjoint from $S_p \cup Q_0 \cup R \cup \{ a \}$, and satisfying the following conditions.
\begin{itemize} 
\item For each $v \in Q_1$, we have $q_v \equiv 1 \text{ mod }p$. 
\item For each $v \in Q_1$, $\overline{\rho}(\Frob_v)$ has distinct eigenvalues $\alpha_v, \beta_v \in k$.
\end{itemize}
We write $\Delta_{Q_1}$ for the maximal $p$-power quotient of $\prod_{v \in Q_1} k(v)^\times$, $\fra_{Q_1} \subset \cO[\Delta_{Q_1}]$ for the augmentation ideal, and $\cS_{Q_1}$ for the Taylor--Wiles augmented deformation problem:
\begin{multline*}
\cS_{Q_1} = (\overline{\rho}, \epsilon^{-1} \psi, S_p \cup Q_0 \cup R \cup Q_1, \{ \Lambda_v \}_{v \in \sigma} \cup \{ \cO \}_{v \in S_p \cup Q_0 \cup R \cup Q_1 - \sigma}, \\ \{ \cD_v^{\text{ord}} \}_{v \in \sigma} \cup \{ \cD_v^{\text{non-ord}} \}_{v \in S_p - \sigma} \cup \{ \cD_v^{\St(\alpha_v)} \}_{v \in Q_0} \cup \{ \cD_v^{\St} \}_{v \in R} \cup \{ \cD_v^\square \}_{v \in Q_1}). 
\end{multline*}
Let $\rho_{\cS_{Q_1}} : G_F \to \GL_2(R_{\cS_{Q_1}})$ be a representative of the universal deformation. Then for each $v \in Q_1$ there are continuous characters $A_v, B_v : G^\text{ab}_{F_v} \to R_{\cS_{Q_1}}^\times$ which modulo $\ffrm_{R_{\cS_{Q_1}}}$ are unramified and take $\Frob_v$ to $\alpha_v, \beta_v$, respectively, and an isomorphism
\[ \rho_{\cS_{Q_1}}|_{G_{F_v}} \sim \left(\begin{array}{cc} A_v & 0 \\ 0 & B_v \end{array}\right). \]
The universal deformation ring $R_{\cS_{Q_1}}$ then acquires a natural structure of $\cO[\Delta_{Q_1}]$-algebra, built from the maps $(v \in Q_1)$:
\begin{gather*}
k(v)^\times \to R_{\cS_{Q_1}}^\times \\
\sigma \mapsto A_v(\Art_{F_v}(\sigma)).
\end{gather*}
Moreover, the map $R_{\cS_{Q_1}}/\fra_{Q_1} R_{\cS_{Q_1}} \to R_\cS$ is an isomorphism. We now introduce auxiliary Hecke modules. Let $H_0 =  H^\text{ord}_\psi(U_1(\sigma^\infty))_\ffrm^\vee$.
\begin{lemma}
With notation as above, there exists an $R_{\cS_{Q_1}}$-module $H_{Q_1}$, free over $\Lambda[\Delta_{Q_1}]$, and equipped with a canonical isomorphism $H_{Q_1} / \fra_{Q_1} H_{Q_1} \cong H_0$ of $R_{\cS}$-modules.
\end{lemma}
\begin{proof}
We define open compact subgroups $U_0(Q_1) = \prod_v U_0(Q_1)_v $ and $U_1(Q_1) = \prod_v U_1(Q_1)_v$ of $U$ as follows. 
\begin{itemize} \item If $v\not\in Q_1$, then $U_0(Q_1)_v = U_1(Q_1)_v = U_v$. If $v \in Q_1$, then $U_0(Q_1)_v = U_0(v)$. 
\item If $v \in Q_1$, then there is a canonical homomorphism from $U_0(v)$ to the maximal $p$-power quotient of $k(v)^\times$, given by
\[ \left( \begin{array}{cc} a & b \\ c & d \end{array} \right) \mapsto a d^{-1}. \]
We define $U_1(Q_1)_v$ to be the kernel of this homomorphism. 
\end{itemize}
Thus $U_1(Q_1)$ is a normal subgroup of $U_0(Q_1)$, and there is a canonical isomorphism $U_0(Q_1)/U_1(Q_1) \cong \Delta_{Q_1}$. The Hecke algebra $\bbT_{Q_0 \cup Q_1}^{\Lambda, S \cup Q_1, \text{univ}}$ acts on each space $H^\text{ord}_\psi(U_1(\sigma^n) \cap U_0(Q_1), A)$ and $H^\text{ord}_\psi(U_1(\sigma^n) \cap U_1(Q_1), A)$.

We recall that there is a homomorphism $f : \bbT^{\Lambda, S}_{Q_0}(H^\text{ord}_\psi(U_1(\sigma), E/\cO)) \to k$ with kernel $\ffrm$. We write $\ffrm_{Q_1}$ for the maximal ideal of $\bbT^{\Lambda, S \cup Q_1}_{Q_0 \cup Q_1}(H^\text{ord}_\psi(U_1(\sigma) \cap U_0(Q_1), E/\cO))$ generated by $\ffrm_\Lambda$, the elements $T_v - f(T_v)$ ($v \not\in S \cup Q_1$) and the elements $\mathbf{U}_v - \alpha_v$ ($v \in Q_0 \cup Q_1$). We also write $\ffrm_{Q_1}$ for the pullback of this maximal ideal to $\bbT^{\Lambda, S \cup Q_1}_{Q_0 \cup Q_1}(H^\text{ord}_\psi(U_1(\sigma) \cap U_1(Q_1), E/\cO))$. Then, just as in \cite[\S 2]{Tay06}, one can show the following facts:
\begin{itemize}
\item for each $n \geq 1$, there is an isomorphism 
\[ \bbT^{\Lambda, S \cup Q_1}_{Q_0 \cup Q_1}(H^\text{ord}_\psi(U_1(\sigma^n) \cap U_0(Q_1), E/\cO))_{\ffrm_{Q_1}} \cong \bbT^{\Lambda, S \cup Q_1}_{Q_0 \cup Q_1}(H^\text{ord}_\psi(U_1(\sigma^n), E/\cO))_\ffrm \]
 of $\Lambda$-algebras, and a corresponding isomorphism 
 \[ H^\text{ord}_\psi(U_1(\sigma^n) \cap U_0(Q_1), E/\cO)_{\ffrm_{Q_1}} \cong H^\text{ord}_\psi(U_1(\sigma^n), E/\cO)_{\ffrm} \]
 of Hecke modules;
\item and for each $n \geq 1$, the $\Lambda$-subalgebra of $\End_\Lambda(H^\text{ord}_\psi(U_1(\sigma^n) \cap U_1(Q_1), E/\cO))_{\ffrm_{Q_1}}$  generated by $\cO[\Delta_{Q_1}]$ is contained inside $\bbT^{\Lambda, S \cup Q_1}_{Q_0 \cup Q_1}(H^\text{ord}_\psi(U_1(\sigma^n) \cap U_1(Q_1), E/\cO))_{\ffrm_{Q_1}}$, and 
there is a natural map $R_{\cS_{Q_1}} \to \bbT^{\Lambda, S \cup Q_1}_{Q_0 \cup Q_1}(H^\text{ord}_\psi(U_1(\sigma^n) \cap U_1(Q_1), E/\cO))_{\ffrm_{Q_1}}$ of $\Lambda[\Delta_{Q_1}]$-algebras satisfying the condition $\tr \rho_{\cS_{Q_1}}(\Frob_v) \mapsto T_v$ and $\det \rho_{\cS_{Q_1}}(\Frob_v) \mapsto S_v$ for all $v \not\in S_p \cup R \cup Q_0 \cup \{ a \} \cup Q_1$.
\end{itemize}
We therefore define
 \[ H^\text{ord}_\psi(U_1(\sigma^\infty) \cap U_0(Q_1)) = \ilim_n H^\text{ord}_\psi(U_1(\sigma^n) \cap U_0(Q_1), E/\cO)\]
so that
 \[ \bbT^{\Lambda, S \cup Q_1}_{Q_0 \cup Q_1}(H^\text{ord}_\psi(U_1(\sigma^\infty) \cap U_0(Q_1))) = \plim_n \bbT^{\Lambda, S \cup Q_1}_{Q_0 \cup Q_1}(H^\text{ord}_\psi(U_1(\sigma^n) \cap U_0(Q_1), E/\cO)) \]
 and
 \[ H^\text{ord}_\psi(U_1(\sigma^\infty) \cap U_1(Q_1)) = \ilim_n H^\text{ord}_\psi(U_1(\sigma^n) \cap U_1(Q_1), E/\cO)\]
so that
 \[ \bbT^{\Lambda, S \cup Q_1}_{Q_0 \cup Q_1}(H^\text{ord}_\psi(U_1(\sigma^\infty) \cap U_1(Q_1))) = \plim_n \bbT^{\Lambda, S \cup Q_1}_{Q_0 \cup Q_1}(H^\text{ord}_\psi(U_1(\sigma^n) \cap U_1(Q_1), E/\cO)) \]
 and finally
 \[ H_{Q_1} = H^\text{ord}_\psi(U_1(\sigma^\infty) \cap U_1(Q_1))_{\ffrm_{Q_1}}^\vee = \left(\ilim_n H^\text{ord}_\psi(U_1(\sigma^n) \cap U_1(Q_1), E/\cO)_{\ffrm_{Q_1}}\right)^\vee. \]
  It follows from Lemma \ref{lem_automorphic_forms_are_free} that each 
\[ H^\text{ord}_\psi(U_1(\sigma^n) \cap U_1(Q_1), E/\cO)^\vee \cong \Hom_\cO(H^\text{ord}_\psi(U_1(\sigma^n) \cap U_1(Q_1), \cO), \cO) \]
is free over $\Lambda_n[\Delta_{Q_1}]$, and hence that $H_{Q_1}$ is free over $\Lambda[\Delta_{Q_1}]$. We already know that $H_{Q_1}$ is a $R_{\cS_{Q_1}}$-module; it remains to show that there is an isomorphism $H_{Q_1}/\fra_{Q_1} H_{Q_1} \cong H_0$ of $R_\cS$-modules. However, we have seen that there is an isomorphism
\[H_{Q_1}^\vee[\fra_{Q_1}] =  H^\text{ord}_\psi(U_1(\sigma^\infty) \cap U_1(Q_1))_{\ffrm_{Q_1}}[\fra_{Q_1}] =  H^\text{ord}_\psi(U_1(\sigma^\infty) \cap U_0(Q_1))_{\ffrm_{Q_1}} \cong H^\text{ord}_\psi(U_1(\sigma^\infty))_{\ffrm} = H_0^\vee
\]
of $R_\cS$-modules, and dualizing this statement now gives the desired result.
\end{proof}
 The following lemma is a consequence of Proposition \ref{prop_existence_of_auxiliary_taylor_wiles_primes} and Proposition \ref{prop_presenting_global_deformation_ring}.
\begin{lemma}
Let $q = h^1_{\cS, T}(M_0(1))$. \(We recall that $h^1_{\cS, T}(M_1(1)) = 0$, by assumption.\)  Then for all integers $N \geq 1$, there exists a set $Q_N$ of primes satisfying the following conditions:
\begin{itemize}
\item $Q_N \cap ( S_p \cup Q_0 \cup R \cup \{ a \})= \emptyset$ and $\# Q_N = q$.
\item For each $v \in Q_N$, $q_v \equiv 1 \text{ mod }p^N$.
\item For each $v \in Q_N$, $\overline{\rho}(\Frob_v)$ has distinct eigenvalues $\alpha_v, \beta_v \in k$.
\item $h^1_{\cS_{Q_N}, T}(\ad^0 \overline{\rho}(1)) = 0$.
\item The ring $R_{\cS_{Q_N}}^T$ can be written as a quotient of a power series ring over $A^T_{\cS_{Q_N}} = A_{\cS}^T$ in $q - [F : \bbQ] - 1 + \#T$ variables.
\end{itemize}
\end{lemma}
We can now prove Theorem \ref{thm_r_equals_t}. Let $R_\infty = A_{\cS}^T \llbracket X_1, \dots, X_{q - [F : \bbQ] - 1 + \# T}\rrbracket$. Then (cf. \cite[Lemma 3.3]{Bar11}) $R_\infty$ is reduced, and for each minimal prime $Q \subset \Lambda$, $\Spec R_\infty/(Q)$ is geometrically irreducible of dimension 
\begin{multline*} \dim A_{\cS}^T + q - [F : \bbQ] - 1 + \# T = 1 + \sum_{v \in \sigma} 2 [F_v : \bbQ] + \sum_{v \in S_p - \sigma} [F_v : \bbQ] + 3 \# T + q - [F : \bbQ] - 1 + \# T \\ = \dim \Lambda + q - 1 + 4 \# T,
\end{multline*}
and its generic point is of characteristic 0. Fix $v_0 \in T$, and let $\cT = \cO \llbracket \{ Y_{i, j}^v \}_{v \in T} \rrbracket/(Y_{1, 1}^{v_0})$. Fix a representative $\rho_{\cS}$ of the universal deformation over $R_\cS$, and for every $N$ a representative $\rho_{\cS_{Q_N}}$ over $R_{\cS_{Q_N}}$ that specializes to this $\rho_{\cS}$. We then get compatible isomorphisms
\begin{gather}
R_\cS^{T} \cong R_\cS \widehat{\otimes}_{\cO} \cT ,\label{eqn_r_s_isomorphism}\\
R_{\cS_{Q_N}}^{T} \cong R_{\cS_{Q_N}} \widehat{\otimes}_{\cO} \cT\label{eqn_r_s_q_isomorphism}
\end{gather}
corresponding to the strict equivalence classes of the universal $T$-framed liftings
\begin{gather*}
(\rho_\cS, (1 + (Y_{i, j}^v))_{v \in T}),\\
(\rho_{\cS_{Q_N}}, (1 + (Y_{i, j}^v))_{v \in T}).
\end{gather*}
We write $\Delta_{\infty} = \bbZ_p^q$, and fix for each $N$ a surjection $\Delta_\infty \to \Delta_{Q_N}$. Let $S_\infty = \Lambda \llbracket \Delta_{\infty}  \rrbracket \widehat{\otimes}_\cO \cT$, and let $\fra_\infty \subset S_\infty$ denote the kernel of the augmentation homomorphism $S_\infty \to \Lambda$. Then the rings $R_\cS^{T}$ and $R_{\cS_{Q_N}}^{T}$ become $S_\infty$-algebras, via the isomorphisms (\ref{eqn_r_s_isomorphism}) and (\ref{eqn_r_s_q_isomorphism}), and the modules $H^{T}_0 = H_0 \otimes_{R_\cS} R_{\cS}^{T}$ and $H^{T}_{Q_N} = H_{Q_N} \otimes_{R_{\cS_{Q_N}}} R_{\cS_{Q_N}}^{T}$ are free over $\Lambda \widehat{\otimes}_\cO \cT [\Delta_{Q_N}]$. By a standard patching argument (cf. \cite[Lemma 6.10]{Tho12} or the proof of \cite[Theorem 4.3.1]{Ger09}), we can construct the following data:
\begin{itemize}
\item A finitely generated $R_\infty$-module $H_\infty$.
\item A homomorphism of $\Lambda$-algebras $S_\infty \to R_\infty$, making $H_\infty$ a free $S_\infty$-module.
\item A surjection $R_\infty/\fra_\infty R_\infty \to R_{\cS}$ and an isomorphism $H_\infty / \fra_\infty H_\infty \cong H_0$ of $R_\cS$-modules. 
\end{itemize}
Let $Q \subset \Lambda$ be a minimal prime. Then $H_\infty/(Q)$ is a free $S_\infty/(Q)$-module and $S_\infty/(Q)$ is a regular local ring. In particular, we have 
\[ \mathrm{depth}_{R_\infty/(Q)} H_\infty /(Q)\geq \mathrm{depth}_{S_\infty/(Q)} H_\infty/(Q) = \dim S_\infty/(Q) = \dim R_\infty/(Q). \]
By \cite[Lemma 2.3]{Tay08}, we see that $H_\infty / (Q)$ is a nearly faithful $R_\infty/(Q)$-module. Since $Q$ was arbitrary, we see that $H_\infty$ is a faithful $R_\infty$-module. ($R_\infty$ is reduced.) It follows that $\Fitt_{R_\infty} H_\infty = 0$, hence
\[ \Fitt_{R_\cS} (H_\infty \otimes_{R_\infty} R_\cS) = \Fitt_{R_\cS} H_0 = 0, \]
as desired.

\section{Deduction of the main theorem}\label{sec_main_theorem}

In this section, we deduce the results stated in the introduction. There are 3 main steps. First, we verify the residual automorphy of the Galois representations under consideration. We then use this to prove an automorphy result under favorable local hypotheses. Finally, we show that the general situation can always be reduced to this one using soluble base change.

\subsection{Some preliminary results}
We will often refer to the following lemma without comment.
\begin{lemma}\label{lem_soluble_base_change} Let $F$ be a totally real number field, and $F'/F$ a soluble totally real extension. Let $p$ be a prime, and fix an isomorphism $\iota : \barqp \to \bbC$.
\begin{enumerate}
\item Let $\pi$ be a cuspidal automorphic representation of $\GL_2(\bbA_F)$ of weight 2, and suppose that $r_\iota(\pi)|_{G_{F'}}$ is irreducible. Then there exists a cuspidal automorphic representation $\pi_{F'}$ of $\GL_2(\bbA_{F'})$ of weight 2, called the base change of $\pi$, such that $r_\iota(\pi_{F'}) \cong r_\iota(\pi)|_{G_{F'}}$.
\item Let $\rho : G_F \to \GL_2(\barqp)$ be a continuous representation such that $\rho|_{G_{F'}}$ is irreducible, and let $\pi'$ be a cuspidal automorphic representation of $\GL_2(\bbA_{F'})$ of weight 2 such that $\rho|_{G_{F'}} \cong r_\iota(\pi')$. Then there exists a cuspidal automorphic representation $\pi$ of $\GL_2(\bbA_F)$ of weight 2 such that $\rho \cong r_\iota(\pi)$.
\end{enumerate}
\end{lemma}
\begin{proof}
The lemma may be deduced from the main results of \cite{Lan80}, using the argument of \cite[Lemma 1.3]{Bar11}.
\end{proof}

\begin{theorem}\label{thm_main_theorem_preliminary_version}
Let $F$ be a totally real field, let $p$ be an odd prime, and let $\rho : G_F \to \GL_2(\barqp)$ be a continuous representation. Suppose that the following conditions hold.
\begin{enumerate}
\item $[F : \bbQ]$ is even, and $[F(\zeta_p) : F]$ is divisible by 4. We write $K$ for the unique quadratic subfield of $F(\zeta_p)/F$, which is therefore totally real.
\item There exists a continuous character $\overline{\chi} : G_K \to \overline{\bbF}_p^\times$ and an isomorphism $\overline{\rho} \cong \Ind_{G_K}^{G_F} \overline{\chi}$. 
\item Let $w \in \Gal(K/F)$ denote the non-trivial element. Then the character $\overline{\gamma} = \overline{\chi}/\overline{\chi}^w$ remains non-trivial, even after restriction to $G_{F(\zeta_p)}$. In particular, $\overline{\rho}$ is irreducible.
\item The character $\psi = \epsilon \det \rho$ is everywhere unramified.
\item The representation $\rho$ is almost everywhere unramified. 
\item For each place $v | p$, $\rho|_{G_{F_v}}$ is semi-stable and $\overline{\rho}|_{G_{F_v}}$ is trivial. For each embedding $\tau : F \hookrightarrow \barqp$, we have $\mathrm{HT}_\tau(\rho) = \{0, 1\}$. 
\item If $v \nmid p$ is a finite place of $F$ at which $\rho$ is ramified, then $q_v \equiv 1 \text{ mod } p$, $\mathrm{WD}(\rho|_{G_{F_v}})^\text{F-ss} \cong \rec^T_{F_v}(\St_2(\chi_v))$, for some unramified character $\chi_v : F_v^\times \to \overline{\bbQ}_p^\times$, and $\overline{\rho}|_{G_{F_v}}$ is trivial. The number of such places is even.
\item There exists a cuspidal automorphic representation $\pi$ of $\GL_2(\bbA_F)$ of weight 2 and an isomorphism $\iota : \barqp \to \bbC$ satisfying the following conditions:
\begin{enumerate}
\item There is an isomorphism $\overline{r_\iota(\pi)} \cong \overline{\rho}$.
\item If $v | p$ and $\rho$ is ordinary, then $\pi_v$ is $\iota$-ordinary and $\pi_v^{U_0(v)} \neq 0$. If $v | p $ and $\rho$ is non-ordinary, then $\pi_v$ is not $\iota$-ordinary and $\pi_v$ is unramified. 
\item If $v \nmid p \infty$ and $\rho|_{G_{F_v}}$ is unramified, then $\pi_v$ is unramified. If $v \nmid p \infty$ and $\rho|_{G_{F_v}}$ is ramified, then $\pi_v$ is an unramified twist of the Steinberg representation.
\end{enumerate}
\end{enumerate}
Then $\rho$ is automorphic: there exists a cuspidal automorphic representation $\pi'$ of $\GL_2(\bbA_F)$, of weight 2 and an isomorphism $\rho \cong r_\iota(\pi')$.
\end{theorem}
\begin{proof}
After possibly replacing $\rho$ by a conjugate, we can find a coefficient field $E$ such that $\rho$ is valued in $\GL_2(\cO)$ and $\overline{\chi}$ is valued in $k^\times$. We prove the theorem by showing that $\rho$ satisfies the conditions of Corollary \ref{cor_automorphy_by_successive_approximation}. To do this, we will apply Theorem \ref{thm_r_equals_t} and its corollary. We therefore fix an integer $N \geq 1$. We write $\sigma \subset S_p$ for the set of places such that $\rho|_{G_{F_v}}$ is ordinary, and $R$ for the set of places not dividing $p$ at which $\pi_v$ is ramified. 

We consider the global deformation problem
\[ \cS = (\overline{\rho}, \epsilon^{-1} \psi, S_p  \cup R, \{ \Lambda_v \}_{v \in \sigma} \cup \{ \cO \}_{v \in S_p \cup R - \sigma}, \{ \cD_v^{\text{ord}} \}_{v \in \sigma} \cup \{ \cD_v^{\text{non-ord}} \}_{v \in S_p - \sigma} \cup \{ \cD_v^{\St} \}_{v \in R}), \]
and set $T = S_p \cup R$. By Proposition \ref{prop_existence_of_auxiliary_ramakrishna_primes}, we can find a finite set $Q_0$ of finite places of $F$, disjoint from $S_p \cup R$, and elements $\alpha_v \in k$ ($v \in Q_0$) all satisfying the following conditions:
\begin{itemize}
\item $\# Q_0 = 2 \lceil h^1_{\cS, T}(M_1(1))/2 \rceil$.
\item For each $v \in Q_0$, the local deformation problem $\cD_v^{\St(\alpha_v)}$ is defined.
\item If $\cS_{Q_0}$ denotes the augmented global deformation problem
\begin{multline*} \cS_{Q_0} = (\overline{\rho}, \epsilon^{-1} \psi, S_p  \cup R \cup Q_0, \{ \Lambda_v \}_{v \in \sigma} \cup \{ \cO \}_{v \in S_p \cup R \cup Q_0 - \sigma}, \\ \{ \cD_v^{\text{ord}} \}_{v \in \sigma} \cup \{ \cD_v^{\text{non-ord}} \}_{v \in S_p - \sigma} \cup \{ \cD_v^{\St} \}_{v \in R} \cup \{ \cD_v^{\St(\alpha_v)} \}_{v \in Q_0}), 
\end{multline*}
then $h^1_{\cS_{Q_0}, T}(M_1(1)) = 0$.
\end{itemize}
By Lemma \ref{lem_level_raising_while_preserving_ordinarity}, we can find an cuspidal automorphic representation $\pi_0$ of $\GL_2(\bbA_F)$ of weight 2, satisfying the following conditions:
\begin{itemize}
\item There is an isomorphism of residual representations $\overline{r_\iota(\pi_0)} \cong \overline{\rho}$.
\item If $v \in \sigma$, then $\pi_{0, v}$ is $\iota$-ordinary and $\pi_{0, v}^{U_0(v)} \neq 0$. If $v \in S_p - \sigma$, then $\pi_{0, v}$ is not $\iota$-ordinary and $\pi_{0, v}$ is unramified.
\item If $v \not\in S_p \cup R \cup Q_0$ is a finite place of $F$, then $\pi_{0, v}$ is unramified. If $v \in R \cup Q_0$, then $\pi_{0, v}$ is an unramified twist of the Steinberg representation. If $v \in Q_0$, then the eigenvalue of ${\mathbf{U}_v}$ on $\iota^{-1} \pi_{0, v}^{U_0(v)}$ is congruent to $\alpha_v$ modulo the maximal ideal of $\overline{\bbZ}_p$. 
\end{itemize}
(In order to apply Lemma \ref{lem_level_raising_while_preserving_ordinarity}, we must first fix a choice of auxiliary place $a$. We choose any place $a$ satisfying the conclusions of Lemma \ref{lem_auxiliary_place_a}.) After replacing $\pi_0$ by a character twist, we can assume in addition that $\pi_0$ has central character $\iota \psi$. (We use here that $p$ is odd.) The hypotheses of Theorem \ref{thm_r_equals_t} are then satisfied with respect to the deformation problem $\cS_{Q_0}$.

Let $S = S_p \cup R \cup \{ a \}$, and let $\ffrm_\emptyset \subset \bbT^{S, \text{univ}}$ be the maximal ideal corresponding to the automorphic representation $\pi$; thus $\ffrm_\emptyset$ is in the support of $H_R(U)$. We set $C_0 = \dim_k (H_R(U) \otimes_\cO k) [\ffrm_\emptyset]$. It follows from Proposition \ref{prop_growth_of_multiplicities} that (notation being as in \S \ref{sec_level_raising})
\[ \dim_k (H_{R \cup Q_0}(U_{Q_0}) \otimes_\cO k) [\ffrm_{Q_0}] \leq 4^{\# Q_0} C_0. \]
We can then apply Corollary \ref{cor_r_equals_t} with $C = 4^{\# Q_0} C_0$ and $n = 1$ (use that, in the notation of that corollary, $H_\psi^\text{ord}(U_{Q_0}, k)$ is a submodule of $H_{R \cup Q_0}(U_{Q_0}) \otimes_\cO k$). It follows that there is a homomorphism $\bbT^{S \cup Q_0}_{Q_0}(H_{R \cup Q_0}(U_{Q_0})) \to \cO/\lambda^{\lfloor N/C \rfloor}$ with the following properties:
\begin{itemize}
\item For each finite place $v \not \in S \cup Q_0$ of $F$, we have $f(T_v) = \tr \rho(\Frob_v) \text{ mod }\lambda^{\lfloor N/C \rfloor}$.
\item $(H_{R \cup Q_0}(U_{Q_0}) \otimes_\cO \cO/\lambda^{\lfloor N/C \rfloor})[I]$ contains an $\cO$-submodule isomorphic to $\cO/\lambda^{\lfloor N/C \rfloor}$.
\end{itemize}
The conditions of Corollary \ref{cor_automorphy_by_successive_approximation} are therefore satisfied. Indeed, it suffices to remark that the quantity 
\[ C = 4^{\# Q_0} C_0 = 4^{2 \lceil h^1_{\cS, T}(M_1(1))/2 \rceil} \dim_k (H_R(U) \otimes_\cO k) [\ffrm_\emptyset] \]
is independent of $N$. This completes the proof.
\end{proof}

\begin{lemma}\label{lem_existence_of_ordinary_automorphic_lifts}
Let $F$ be a totally real field, let $p$ be an odd prime, let $\iota : \barqp \to \bbC$ be an isomorphism, and let $\overline{\rho} : G_F \to \GL_2(\overline{\bbF}_p)$ be a continuous representation. Suppose that $\det \overline{\rho}$ is totally odd, that $\overline{\rho}$ is absolutely irreducible and unramified at the places of $F$ above $p$, and that there exists a quadratic extension $K/F$ such that $\overline{\rho}|_{G_K}$ is reducible. Then there exists a cuspidal automorphic representation $\pi$ of $\GL_2(\bbA_F)$ of weight 2 and satisfying the following conditions.
\begin{enumerate}
\item For each place $v | p$, $\pi_v$ is $\iota$-ordinary.
\item There is an isomorphism $\overline{r_\iota(\pi)} \cong \overline{\rho}$.
\item For each finite place $v \not\in S_p$ of $F$ at which $\overline{\rho}$ is unramified, $\pi_v$ is unramified.
\end{enumerate}
\end{lemma}
\begin{proof}
By automorphic induction, there exists a Hilbert modular newform $\mathbf{f}$ of parallel weight 1 (in the sense of \cite{Wil88}) associated to the Teichm\"uller lift of $\overline{\rho}$, which has level prime to $p$. By \cite[Theorem 3]{Wil88}, there exists a cuspidal automorphic representation $\pi$ of $\GL_2(\bbA_F)$ of weight 2 such that $\overline{r_\iota(\pi)} \cong \overline{\rho}$, with associated Hilbert modular newform contained in a $\Lambda$-adic eigenform $\cF$ passing through $\mathbf{f}$. In particular, $\pi$ is $\iota$-ordinary by construction, and is ramified only at places of $F$ dividing $p$ or at which $\overline{\rho}$ is ramified.
\end{proof}

\begin{corollary}\label{cor_existence_of_partially_ordinary_lifts} Let $F$ be a totally real field, let $p$ be an odd prime, let $\iota : \barqp \to \bbC$ be an isomorphism, and let $\overline{\rho} : G_F \to \GL_2(\overline{\bbF}_p)$ be a continuous representation. Suppose that $\det \overline{\rho}$ is totally odd, that $\overline{\rho}$ is absolutely irreducible and unramified at the places of $F$ above $p$, and that there exists a quadratic extension $K/F$ such that $\overline{\rho}|_{G_K}$ is reducible. Let $V_0$ be a finite set of finite places of $F$, not containing any place above $p$ or at which $\overline{\rho}$ is ramified, and let $\sigma$ be a set of places dividing $p$. Then there exists a soluble totally real extension $F'/F$ in which every place of $V_0$ splits, and a cuspidal automorphic representation $\pi'$ of $\GL_2(\bbA_{F'})$ of weight 2 and satisfying the following conditions.
\begin{enumerate}
\item If $v$ is a place of $F'$ above a place of $\sigma$, then $\pi'_v$ is $\iota$-ordinary. If $v$ is a place of $F'$ dividing $p$ but not dividing a place of $\sigma$, then $\pi'_v$ is unramified and not $\iota$-ordinary.
\item If $v \nmid p \infty$, then $\pi'_v$ is unramified.
\item There is an isomorphism $\overline{r_\iota(\pi')} \cong \overline{\rho}|_{G_{F'}}$, and $\overline{\rho}|_{G_{F'}}$ is absolutely irreducible.
\end{enumerate}
\end{corollary}
\begin{proof}
Let $E/F$ denote the extension of $F$ cut out by $\overline{\rho}$. After possibly enlarging $V_0$, we can assume that it satisfies the following conditions:
\begin{itemize}
\item For any Galois subextension $E/M/F$ with $\Gal(M/F)$ simple and non-trivial, there exists $v \in V_0$ which does not split in $M$.
\item The set $V_0$ does not contain any place dividing $p$ or at which $\overline{\rho}$ is ramified. 
\end{itemize}
Then any Galois extension $F'/F$ in which each place of $V_0$ splits is linearly disjoint from $E$; in particular, $\overline{\rho}|_{G_{F'}}$ will be absolutely irreducible. 

Let $\pi$ be an automorphic representation of $\GL_2(\bbA_F)$ satisfying the conclusion of Lemma \ref{lem_existence_of_ordinary_automorphic_lifts}. After possibly replacing $F$ with a preliminary $V_0$-split soluble extension (and $\pi$ by its base change), we can assume that $[F : \bbQ]$ is even, and that for each place $v$ of $F$, either $\pi_v$ is unramified, or $v \in S_p$ and $\pi_v$ is an unramified twist of the Steinberg representation, or $q_v \equiv 1 \text{ mod }p$ and $\pi_v$ is an unramified twist of the Steinberg representation. By Lemma \ref{lem_convert_ordinary_to_supercuspidal}, we can find a cuspidal automorphic representation $\pi''$ of $\GL_2(\bbA_F)$ of weight 2 with the following properties:
\begin{itemize}
\item If $v \in \sigma$ then $\pi''_v$ is $\iota$-ordinary. If $v \in S_p - \sigma$ then $\pi''_v$ is supercuspidal.
\item If $v \nmid p\infty$ and $\pi_v$ is ramified, then $\pi''_v$ is a ramified principal series representation.
\item There is an isomorphism of residual representations $\overline{r_\iota(\pi)} \cong \overline{r_\iota(\pi'')}$.
\end{itemize}
It now suffices to choose $F'/F$ to be any soluble $V_0$-split extension with the property that $\pi' = \pi''_{F'}$ is everywhere unramified (except possibly at the places of $F'$ above $\sigma$). Indeed, it only remains to check that if $v$ is a place of $F'$ above $S_p - \sigma$, then $r_\iota(\pi')|_{G_{F'_v}}$ is crystalline non-ordinary, and this follows from Lemma \ref{lem_ordinarity_supercuspidality_and_base_change}.
\end{proof}

\subsection{The main theorem}

\begin{theorem}\label{thm_main_theorem}
Let $F$ be a totally real number field, let $p$ be an odd prime, and let $\rho : G_F \to \GL_2(\barqp)$ be a continuous representation satisfying the following conditions.
\begin{enumerate}
\item The representation $\rho$ is almost everywhere unramified. 
\item For each place $v | p$ of $F$, $\rho|_{G_{F_v}}$ is de Rham. For each embedding $\tau : F \hookrightarrow \barqp$, we have $\mathrm{HT}_\tau(\rho) = \{0, 1\}$.
\item For each complex conjugation $c \in G_F$, we have $\det \rho(c) = -1$.
\item The residual representation $\overline{\rho}$ is absolutely irreducible, yet $\overline{\rho}|_{G_{F(\zeta_p)}}$ is a direct sum of two distinct characters. The unique quadratic subfield $K$ of $F(\zeta_p)/F$ is \emph{totally real}.
\end{enumerate}
Then $\rho$ is automorphic: there exists a cuspidal automorphic representation $\pi$ of $\GL_2(\bbA_F)$ of weight 2, an isomorphism $\iota : \barqp \to \bbC$, and an isomorphism $\rho \cong r_\iota(\pi)$.
\end{theorem}
\begin{proof}
Replacing $\rho$ by a twist, we can assume that $\epsilon \det \rho$ has order prime to $p$. Let $E/F$ denote the extension of $F(\zeta_p)$ cut out by $\overline{\rho}|_{G_{F(\zeta_p)}}$. Let $V_0$ be a finite set of finite places of $F$ satisfying the following conditions:
\begin{itemize}
\item For any Galois subextension $E/M/F$ with $\Gal(M/F)$ simple and non-trivial, there exists $v \in V_0$ which does not split in $M$.
\item The set $V_0$ does not contain any place dividing $p$ or at which $\rho$ is ramified.
\end{itemize}
To prove the theorem, it will suffice to construct a soluble extension $F'/F$ in which every place in $V_0$ splits, and such that $\rho|_{G_{F'}}$ satisfies the conditions of Theorem \ref{thm_main_theorem_preliminary_version}. Indeed, it then follows that there exists an automorphic representation $\pi'$ of $\GL_2(\bbA_{F'})$ of weight 2 such that $\rho|_{G_{F'}} \cong r_\iota(\pi')$, and the automorphy of $\rho$ follows by Lemma \ref{lem_soluble_base_change}. (We note that since $F'$ is $V_0$-split it is linearly disjoint from the extension $E/F$, and in particular the representation $\rho|_{G_{F'}}$ is irreducible, even after reduction modulo $p$.)

In fact, it will even suffice to construct a possibly non-Galois extension $F'/F$ which is $V_0$-split, which has soluble Galois closure, and such that $\rho|_{G_{F'}}$ satisfies the conditions of Theorem \ref{thm_main_theorem_preliminary_version}. Indeed, the preceding argument can then be applied to the Galois closure of $F'/F$. We now construct such an extension.

Replacing $F$ by a $V_0$-split soluble extension, we can assume that $\rho$ satisfies in addition the following conditions:
\begin{itemize}
\item For each place $v | p$ of $F$, $\rho|_{G_{F_v}}$ is semi-stable and $\overline{\rho}|_{G_{F_v}}$ is trivial.
\item If $v \nmid p$ is a finite place of $F$ at which $\rho$ is ramified, then $\mathrm{WD}(\rho|_{G_{F_v}})^\text{F-ss} \cong \rec^T_{F_v}(\St(\chi_v))$, for some unramified character $\chi_v : F_v^\times \to \overline{\bbQ}_p^\times$, and $\overline{\rho}|_{G_{F_v}}$ is trivial.
\end{itemize}
Fix an isomorphism $\iota : \barqp \to \bbC$. By Corollary \ref{cor_existence_of_partially_ordinary_lifts}, we can assume after again replacing $F$ by a soluble extension that there exists a cuspidal automorphic representation $\pi''$ of weight 2 satisfying the following conditions:
\begin{itemize}
\item There is an isomorphism $\overline{r_\iota(\pi'')} \cong \overline{\rho}$.
\item If $v | p$ and $\rho$ is ordinary, then $\pi''_v$ is $\iota$-ordinary.
\item If $v | p$ and $\rho$ is non-ordinary, then $\pi''_v$ is unramified and $r_\iota(\pi'')|_{G_{F_v}}$ is non-ordinary.
\item If $v \nmid p \infty$ then $\pi''_v$ is unramified.
\end{itemize}
Arguing as in the proof of \cite[Lemma 3.5.3]{Kis09} and replacing $F$ by a further soluble extension, we can assume that there exists a cuspidal automorphic representation $\pi$ of weight 2 satisfying the following conditions:
\begin{itemize}
\item There is an isomorphism $\overline{r_\iota(\pi)} \cong \overline{\rho}$.
\item If $v | p$ and $\rho$ is ordinary, then $\pi_v$ is $\iota$-ordinary and $\pi_v^{U_0(v)} \neq 0$.
\item If $v | p$ and $\rho$ is non-ordinary, then $\pi_v$ is unramified and $r_\iota(\pi)|_{G_{F_v}}$ is non-ordinary.
\item If $v \nmid p \infty$ and $\rho$ is unramified, then $\pi_v$ is unramified. If $v \nmid p \infty$ and $\rho$ is ramified, then $\pi_v$ is an unramified twist of the Steinberg representation.
\end{itemize}
The hypotheses of Theorem \ref{thm_main_theorem_preliminary_version} are now satisfied, and this completes the proof.
\end{proof}

\subsection{Application to elliptic curves}
We now prove Theorem \ref{thm_curve_intro_thm}.
\begin{theorem}\label{thm_main_thm_applied_to_elliptic_curves}
Let $F$ be a totally real number field, and let $E$ be an elliptic curve over $F$. Suppose that:
\begin{enumerate}
\item 5 is not a square in $F$;
\item and $E$ has no $F$-rational 5-isogeny.
\end{enumerate}
Then $E$ is modular.
\end{theorem}
\begin{proof}
Let $\rho: G_F \to \GL_2(\bbQ_5)$ denote the representation associated to the action of Galois on the \'etale cohomology $H^1(E_{\overline{F}}, \bbZ_5)$, after a choice of basis. We must show that $\rho$ is automorphic. The condition that $E$ has no $F$-rational 5-isogeny is equivalent to the assertion that $\overline{\rho}$ is irreducible, hence absolutely irreducible (because of complex conjugation). If $\overline{\rho}|_{G_{F(\zeta_5)}}$ is absolutely irreducible, then $\rho$ is automorphic by \cite[Theorem 3]{Fre13}. 

We therefore assume that $\overline{\rho}|_{G_{F(\zeta_5)}}$ is absolutely reducible. It then follows from \cite[Proposition 9.1, (b)]{Fre13} that the projective image of $\rho$ in $\PGL_2(\bbF_5)$ is isomorphic to $\bbF_2 \times \bbF_2$. In particular, $\overline{\rho}|_{G_{F(\zeta_5)}}$ is non-scalar, as $\Gal(F(\zeta_5)/F)$ is cyclic. The hypotheses of Theorem \ref{thm_main_theorem} are now satisfied, and we deduce the automorphy of $\rho$ in this case as well.
\end{proof}
As an example of Theorem \ref{thm_main_thm_applied_to_elliptic_curves}, we consider the elliptic curve over $\bbQ$:
\[ E : y^2 + y = x^3 - 350x + 1776. \]
Let $\rho = \rho_{E, 5} : G_\bbQ \to \GL_2(\bbQ_5)$ denote the representation afforded by  $H^1(E_{\overline{\bbQ}}, \bbZ_5)$ after a choice of basis. The $j$-invariant of $E$ is $552960000/161051 = 2^{15} \times 3^3 \times 5^4 \times 11^{-5}$. The curve $E$ acquires good supersingular reduction over the extension $\bbQ_5(\sqrt[6]{5})$ of $\bbQ_5$. One can check that the curve $E$ has no 5-isogeny defined over $\bbQ$, but acquires a 5-isogeny over $\bbQ(\sqrt{5})$. (We performed these calculations using {\tt sage} \cite{sage}.)

More generally, if $F$ is any totally real number field such that $F \cap \bbQ(\zeta_5) = \bbQ$ and $\overline{\rho}|_{G_{F}}$ is irreducible, then the elliptic curves $E'$ over $F$ with $\overline{\rho}_{E',5} \cong \overline{\rho}|_{G_{F}}$ are parameterized by $\bbP^1(F)$ (see the proof of \cite[Lemma 3.49]{Dar97}), and Theorem \ref{thm_main_thm_applied_to_elliptic_curves} implies that all of these curves are modular. 

\section*{Acknowledgments}

During the period this research was conducted, Jack Thorne served as a Clay Research Fellow. I would like to thank Chandrashekhar Khare for inspiring conversations, and for showing me his papers \cite{Kha03}, \cite{Kha04}. I would also like to thank Toby Gee and David Geraghty for useful conversations, and the anonymous referee for their detailed comments and corrections.

\newcommand{\etalchar}[1]{$^{#1}$}
\def\cprime{$'$}

\end{document}